\documentclass[11pt, reqno]{amsart}
\pdfoutput=1
\usepackage{amsmath}
\usepackage{amssymb}
\usepackage{amsthm}
\usepackage{enumerate}
\usepackage{amscd}
\usepackage{graphicx}
\usepackage{tabu}
\usepackage[top=1.50in, bottom=1.50in, left=1.25in, right=1.25in]{geometry}
\usepackage[all]{xy}
\usepackage{longtable}
\usepackage[utf8]{inputenc}
\usepackage{listings}
\usepackage{filecontents}
\usepackage{scrextend}
\usepackage{url}
\usepackage[all]{xy}
\usepackage{stmaryrd}
\usepackage{bbm}
\usepackage{relsize}
\usepackage{mathrsfs}
\usepackage{dsfont}
\usepackage{color}
\pagecolor{white}

\DeclareMathOperator{\colim}{colim}
\DeclareMathOperator{\id}{id}

\DeclareMathOperator{\Vect}{Vect}
\DeclareMathOperator{\PGL}{PGL}

\DeclareMathOperator{\SL}{SL}

\DeclareMathOperator{\af}{af}

\DeclareMathOperator{\sh}{sh}
\DeclareMathOperator{\der}{der}
\DeclareMathOperator{\sep}{sep}
\DeclareMathOperator{\Sym}{Sym}
\DeclareMathOperator{\Hom}{Hom}

\DeclareMathOperator{\GL}{GL}
\DeclareMathOperator{\length}{length}

\DeclareMathOperator{\Spec}{Spec}

\DeclareMathOperator{\Rep}{Rep}

\DeclareMathOperator{\End}{End}

\DeclareMathOperator{\sem}{ss}
\DeclareMathOperator{\cok}{cok}

\DeclareMathOperator{\Gr}{Gr}

\DeclareMathOperator{\Ext}{Ext}

\DeclareMathOperator{\IC}{IC}

\DeclareMathOperator{\RHom}{RHom}
\DeclareMathOperator{\lfgu}{lfgu}

\let\phi\varphi

\theoremstyle{definition}
\newtheorem{defn}{Definition}[section]
\newtheorem{exmp}[defn]{Example}

\newtheorem{rmrk}[defn]{Remark}

\theoremstyle{plain}
\newtheorem{thm}[defn]{Theorem}
\newtheorem{prop}[defn]{Proposition}
\newtheorem{lem}[defn]{Lemma}
\newtheorem{cor}[defn]{Corollary}

\newcommand\restr[2]{{
  \left.\kern-\nulldelimiterspace 
  #1 
  \vphantom{\big|} 
  \right|_{#2} 
  }}

\newenvironment{myproof}[1][\proofname]{%
  \begin{myproof2}[#1]We proceed by a series of reductions.$ $\par\nobreak\ignorespaces
}{%
  \end{myproof2}
}

\makeatletter
\newenvironment{myproof2}[1][\proofname] {\par\pushQED{\qed}\normalfont\topsep6\p@\@plus6\p@\relax\trivlist\item[\hskip\labelsep\bfseries#1\@addpunct{\textbf{.}}]\ignorespaces}{\popQED\endtrivlist\@endpefalse}
\makeatother

\title{Perverse $\mathbb{F}_p$-Sheaves on the Affine Grassmannian}
\author{Robert Cass}
\address{Department of Mathematics, California Institute of Technology, Pasadena, CA 91125}
\email{rcass@caltech.edu}
 

\numberwithin{equation}{section}

\begin{document}

\begin{abstract}
For a reductive group over an algebraically closed field of characteristic $p > 0$ we construct the abelian category of perverse $\mathbb{F}_p$-sheaves on the affine Grassmannian that are equivariant with respect to the action of the positive loop group. We show this is a symmetric monoidal category, and then we apply a Tannakian formalism to show this category is equivalent to the category of representations of a certain affine monoid scheme. We also show that our work provides a geometrization of the inverse of the mod $p$ Satake isomorphism. Along the way we prove that affine Schubert varieties are globally $F$-regular and we apply Frobenius splitting techniques to the theory of perverse $\mathbb{F}_p$-sheaves.
\end{abstract}

\maketitle

\thispagestyle{empty}

\setcounter{tocdepth}{3}
\tableofcontents
\newpage
\section{Introduction} \label{Intro}

\subsection{Main results} \label{Intro1.1}
Let $G$ be a reductive group over an algebraically closed field $k$, and let $\Gr$ be the associated affine Grassmannian. The ind-scheme $\Gr$ is a union of affine Schubert varieties which are analogous to classical Schubert varieties in flag varieties, and it has important applications in arithmetic geometry and representation theory. For example, fix a prime $\ell$ different from the characteristic of $k$. Let $P_{L^+G}(\Gr, \overline{\mathbb{Q}}_\ell)$ be the category of $L^+G$-equivariant perverse $\overline{\mathbb{Q}}_\ell$-sheaves on $\Gr$, and let $\Rep_{\overline{\mathbb{Q}}_\ell}(\hat{G})$ be the category of representations of the dual group $\hat{G}$ on finite-dimensional $\overline{\mathbb{Q}}_\ell$-vector spaces. One version of the geometric Satake equivalence of Mirkovi\'c and Vilonen \cite{GeometricSatake} states that there is an equivalence of symmetric monoidal categories $P_{L^+G}(\Gr, \overline{\mathbb{Q}}_\ell) \xrightarrow{\sim} \Rep_{\overline{\mathbb{Q}}_\ell}(\hat{G})$. Moreover, the following diagram commutes up to natural isomorphism.
$$\xymatrix{
P_{L^+G}(\Gr, \overline{\mathbb{Q}}_\ell) \ar[rd]^{\sim} \ar[rr]^{\bigoplus_i R^i \Gamma(-)} & & \Vect_{\overline{\mathbb{Q}}_\ell} \\
& \Rep_{\overline{\mathbb{Q}}_\ell}(\hat{G}) \ar[ru]_{\text{Forget}} &
}$$

One can also get an analogous statement using coefficients in the finite field $\mathbb{F}_\ell$. The geometric Satake equivalence can be viewed as both providing a canonical construction of $\hat{G}$ from $G$, and as a geometrization of the Satake isomorphism. Recently, Vincent Lafforgue \cite{V.L.} has used the geometric Satake equivalence (among other things) to prove the automorphic to Galois direction of the Langlands correspondence for reductive groups over global function fields.

Now suppose $k$ has characteristic $p > 0$. The category of \'{e}tale $\mathbb{F}_p$-sheaves on a scheme $X$ over $k$ is markedly different from the category of \'{e}tale $\mathbb{F}_\ell$-sheaves or $\overline{\mathbb{Q}}_\ell$-sheaves for $\ell \neq p$. However, one can still define perverse $\mathbb{F}_p$-sheaves on $X$. We will define perverse $\mathbb{F}_p$-sheaves in Section \ref{PerverseFpSheaves}, but for now we note that when $X$ is smooth, perverse $\mathbb{F}_p$-sheaves on $X$ are anti-equivalent to the category of lfgu $\mathcal{O}_{F,X}$-modules introduced by Emerton-Kisin \cite{EmertonKisin}.

In this paper we define the category $P_{L^+G}(\Gr, \mathbb{F}_p)$ of $L^+G$-equivariant perverse $\mathbb{F}_p$-sheaves on $\Gr$. We then construct a convolution product $*$ making $P_{L^+G}(\Gr, \mathbb{F}_p)$ into a symmetric monoidal category. Our main theorem is the following.

\begin{thm} \label{repthm}
The functor $H: = \bigoplus_i R^i \Gamma ( - ) \colon P_{L^+G}(\Gr, \mathbb{F}_p) \rightarrow \Vect_{\mathbb{F}_p}$ is exact, faithful, and can be naturally upgraded to a symmetric monoidal functor. The functor on $\mathbb{F}_p$-algebras whose $R$-points are the tensor endomorphisms of $H( - ) \otimes_{\mathbb{F}_p} R \colon P_{L^+G}(\Gr, \mathbb{F}_p) \rightarrow \text{\normalfont Mod}_R$ is represented by an affine monoid scheme $M_G$ over $\mathbb{F}_p$. Finally, there is an equivalence of symmetric monoidal categories $P_{L^+G}(\Gr, \mathbb{F}_p) \xrightarrow{\sim} \Rep_{\mathbb{F}_p}(M_G)$ such that the following diagram commutes up to natural isomorphism.
$$\xymatrix{
P_{L^+G}(\Gr, \mathbb{F}_p) \ar[rd]^{\sim} \ar[rr]^{\bigoplus_i R^i \Gamma(-)} & & \Vect_{\mathbb{F}_p} \\
& \Rep_{\mathbb{F}_p}(M_G) \ar[ru]_{\text{\normalfont Forget}} &
}$$
\end{thm}

As in the case of $\overline{\mathbb{Q}}_\ell$-coefficients, the simple objects in $P_{L^+G}(\Gr, \mathbb{F}_p)$ are the IC sheaves supported on $L^+G$-orbit closures. After fixing a maximal torus and a Borel $T \subset B \subset G$ we may parametrize $L^+G$-orbit closures in $\Gr$ by dominant cocharacters $\mu \in X_*(T)^+$ (see Section \ref{GlobFReg1.1}). Let $\IC_\mu$ denote the IC sheaf supported on the reduced orbit closure $\Gr_{\leq \mu}$ associated to $\mu \in X_*(T)^+$. Unlike the case of $\overline{\mathbb{Q}}_\ell$-coefficients, the convolution product of simple objects remains simple.

\begin{thm} \label{convsimp}
Let $\mu_1$, $\mu_2 \in X_*(T)^+$. Then
$$\IC_{\mu_1} * \IC_{\mu_2} = \IC_{\mu_1+\mu_2}.$$ Furthermore, for every $\mu \in X_*(T)^+$,
$$\dim_{\mathbb{F}_p} H(\IC_\mu) = 1.$$
\end{thm}

Our proofs of Theorems \ref{repthm} and \ref{convsimp} rely crucially on the nature of the singularities of the affine Schubert varieties $\Gr_{\leq \mu}$. In particular, we use Frobenius splitting techniques to determine the simple objects in $P_{L^+G}(\Gr, \mathbb{F}_p)$. This will be explained in Section \ref{Fresults} below. 

\subsection{Connections with the $p$-adic Langlands program}
Theorem \ref{repthm} is an initial result in an ongoing project to give a categorification of the representation theory of affine mod $p$ Hecke algebras. This project is motivated by the idea of applying methods from the geometric Langlands program to the $p$-adic Langlands program. As a first step in this direction, we will show that Theorem \ref{convsimp} provides a geometrization of the mod $p$ Satake isomorphism established by Herzig \cite{modpsatake} and Henniart-Vign\'{e}ras \cite{modpsatake2} for the special case of split reductive groups over local fields of equal characteristic.

To explain this result, let $F$ be a local field of equal characteristic $p$ with ring of integers $\mathcal{O}$, and let $G$ be a split reductive group defined over the residue field $\mathbb{F}_q$ of $F$. Let $w_0$ be the longest element of the Weyl group of $G$. Then one can consider the mod $p$ Hecke algebra $$\mathcal{H}_G := \{f \colon G(F) \rightarrow \mathbb{F}_p \: : \: f \text{ has compact support and is } G(\mathcal{O}) \text{ bi-invariant}\}.$$ The ring structure on $\mathcal{H}_G$ is defined by convolution. The algebra $\mathcal{H}_G$ plays a crucial role in the classification of admissible smooth mod $p$ representations of $G(F)$.

The mod $p$ Satake isomorphism is an isomorphism $\mathcal{S} \colon \mathcal{H}_G \xrightarrow{\sim} \mathbb{F}_p[X_*(T)_-]$ where $X_*(T)_-$ is the monoid of anti-dominant cocharacters. We review the definition of $\mathcal{S}$ in Section \ref{SatakeSection}. Given a perverse sheaf $\mathcal{F}^\bullet \in P_{L^+G}(\Gr, \mathbb{F}_p)$, where we view $\Gr$ as an ind-scheme over an algebraic closure $\overline{\mathbb{F}}_q$ of $\mathbb{F}_q$, we describe in Section \ref{SatakeSection} a natural way to associate a function $\mathcal{T}(\mathcal{F}^\bullet) \in \mathcal{H}_G$. This procedure induces a map of $\mathbb{F}_p$-vector spaces $$\mathcal{T} \colon K_0(P_{L^+G}(\Gr, \mathbb{F}_p)) \otimes \mathbb{F}_p \rightarrow \mathcal{H}_G,$$ where $K_0(P_{L^+G}(\Gr, \mathbb{F}_p))$ is the Grothendieck group of $P_{L^+G}(\Gr, \mathbb{F}_p)$. Note that we can equip $K_0(P_{L^+G}(\Gr, \mathbb{F}_p))$ with the structure of a ring where the multiplication is defined by convolution of perverse sheaves. 

\begin{thm} \label{Satthm}
There is a natural isomorphism of $\mathbb{F}_p$-algebras
$$\mathbb{F}_p[X_*(T)^+] \xrightarrow{\sim} K_0(P_{L^+G}(\Gr, \mathbb{F}_p)) \otimes \mathbb{F}_p.$$ If $\alpha \colon \mathbb{F}_p[X_*(T)_-] \rightarrow \mathbb{F}_p[X_*(T)^+]$ is the isomorphism which sends $\mu$ to $w_0(\mu)$, then $\mathcal{S}^{-1}$ is given by the composition
$$\mathbb{F}_p[X_*(T)_-] \xrightarrow{\alpha} \mathbb{F}_p[X_*(T)^+] \xrightarrow{\sim} K_0(P_{L^+G}(\Gr, \mathbb{F}_p)) \otimes \mathbb{F}_p \xrightarrow{\mathcal{T}} \mathcal{H}_G.$$
\end{thm}

Building on results in this paper, in \cite{CentralCass} we construct perverse $\mathbb{F}_p$-sheaves on an affine flag variety which are central with respect to the convolution product. Our construction is an adaptation of a method of Gaitsgory \cite{Gaitsgorycentral} in the mod $p$ setting.  Moreover, we use the function-sheaf correspondence to derive an explicit formula for central elements in the Iwahori mod $p$ Hecke algebra. This allows us to give a geometric proof of certain combinatorial identities in mod $p$ Hecke algebras which are related to work of Vign\'{e}ras \cite{Vigneras2005} and Ollivier \cite{Ollivier2014}. We also explain how to use Theorem \ref{convsimp} to derive the existence of an isomorphism of $\mathbb{F}_p$-algebras $\mathbb{F}_p[X_*(T)_-] \cong \mathcal{H}_G$ without using the existence of the mod $p$ Satake isomorphism.

In future joint work with C. P\'{e}pin and T. Schmidt we aim to construct a mod $p$ version of Bezrukavnikov's work \cite{twohecke}, which is a tamely ramified local geometric Langlands correspondence. In Bezrukavnikov's work, perverse sheaves on an affine flag variety are a categorification of the automorphic side (i.e., representations of Hecke modules), and coherent sheaves on the Steinberg variety of the Langlands dual group are a categorification of the arithmetic side (i.e., tamely ramified representations of a local Galois group). That such an equivalence is possible in the mod $p$ setting is suggested by the results in \cite{PS20} and \cite{CentralCass}. Ideally, such an equivalence would give a geometric construction of some instances of Grosse-Kl\"{o}nne's functor \cite{GK2016} from supersingular mod $p$ Hecke modules to Galois representations.

\subsection{$F$-singularities and perverse $\mathbb{F}_p$-sheaves} \label{Fresults}
Let $k$ be a perfect field of characteristic $p > 0$. In order to prove Theorems \ref{repthm} and \ref{convsimp} we employ Frobenius splitting techniques. To state our results in this direction, note that for every parabolic subgroup $P \subset G$ containing $B$ there is an associated parahoric positive loop group $\mathcal{P} \subset L^+G$ and an affine flag variety $\mathcal{F}\ell_{\mathcal{P}}$
with an action of an Iwahori group $\mathcal{B}$. Given $w$ in the Iwahori-Weyl group, let $S_w \subset \mathcal{F}\ell_{\mathcal{P}}$ be the reduced $\mathcal{B}$-orbit closure associated to $w$. Affine Schubert varieties are related to local models of Shimura varieties (see \cite{GortzShimura}), and as such it is desirable to understand their singularities and cohomology. The affine Schubert varieties $\Gr_{\leq \mu}$ arise in the special case $P = G$ (see Section \ref{GlobFReg1.1}). We will prove the following theorem about the $F$-singularities of affine Schubert varieties.

\begin{thm} \label{thm1}
Suppose that $p \nmid |\pi_1(G_{\der})|$. Then the affine Schubert varieties $S_w$ are globally $F$-regular, and hence also strongly $F$-regular and $F$-rational. 
\end{thm}

We refer the reader to Section \ref{GlobFReg1.3} for the definition of global $F$-regularity. Theorem \ref{thm1} adds to results of Faltings \cite{FaltingsLoop} and Pappas-Rapoport \cite{PappasRapoport} in positive characteristic. In particular, they showed that affine Schubert varieties are normal, Cohen-Macaulay, and Frobenius split if $p \nmid |\pi_1(G_{\der})|$. While global $F$-regularity is stronger than these properties, our proof relies on their results. We also prove that Schubert subvarieties of a particular Beilinson-Drinfeld Grassmannian are $F$-rational in Theorem \ref{BDprops2}. In \cite{CentralCass} we use Theorem \ref{thm1} to prove that certain equal characteristic analogues of local models of Shimura varieties are strongly $F$-regular. 

For the rest of this section we assume that $k$ is algebraically closed. We note that the reduced closure of every $\mathcal{P}$-orbit in $\mathcal{F}\ell_{\mathcal{P}}$ can be identified with $S_w$ for some $w$. Using Theorem \ref{thm1} and the fact that the stabilizers for the action of $\mathcal{P}$ on $\mathcal{F}\ell_{\mathcal{P}}$ are connected, we prove the following result.

\begin{thm} \label{irreducibleobjects}
The simple equivariant perverse sheaves in $P_{\mathcal{P}}(\mathcal{F}\ell_{\mathcal{P}}, \mathbb{F}_p)$ are the shifted constant sheaves $\mathbb{F}_p[\dim S_w]$ supported on those $S_w$ which are $\mathcal{P}$-orbit closures. 
\end{thm}
In the special case $P=G$ we let $$\IC_\mu := \mathbb{F}_p[\dim \Gr_{\leq \mu}]$$ be the shifted constant sheaf supported on $\Gr_{\leq \mu}$.

Theorems \ref{thm1} and \ref{irreducibleobjects} generalize results of Lauritzen, Raben-Pedersen, and Thomsen \cite{SchubertFreg} in the case of classical Schubert varieties in flag varieties. While our proof of Theorem \ref{thm1} is analogous to that in \cite{SchubertFreg}, we derive Theorem \ref{irreducibleobjects} as a consequence of the following more general results which may be of independent interest. 

\begin{thm} \label{mainthm}
If $X$ is a Cohen-Macaulay $k$-scheme of finite type and equidimension $d$ then the complex of \'{e}tale sheaves $\mathbb{F}_p[d] \in D_c^b(X, \mathbb{F}_p)$ is perverse.
\end{thm}

\begin{thm} \label{mainthm2}
If $X$ is an integral $F$-rational $k$-scheme of finite type and dimension $d$ then $\mathbb{F}_p[d]$ is simple as a perverse sheaf. 
\end{thm}

The property of $F$-rationality enters through its connection to the Frobenius structure on local cohomology. See Theorem \ref{FRatDef} for a precise definition of $F$-rationality. We prove Theorems \ref{mainthm} and \ref{mainthm2} by explicit computations with \'{e}tale $\mathbb{F}_p$-sheaves.

\subsection{Outline}
In Sections \ref{PerverseFpSheaves} and \ref{EquivSheaves} we prove basic facts about the category $P_c^b(X, \mathbb{F}_p)$ of bounded constructible perverse $\mathbb{F}_p$-sheaves on a scheme $X$ of finite type over $k$ which are well-known for perverse $\overline{\mathbb{Q}}_\ell$-sheaves. For example, we show that for an action of a smooth connected group scheme $G$ on $X$, the $G$-equivariant perverse sheaves form a full subcategory of $P_c^b(X, \mathbb{F}_p)$ which is stable under taking subquotients. The usual proofs for $\overline{\mathbb{Q}}_\ell$-sheaves do not always work in our setting, largely because only three of the  six functors preserve constructibility. Our main  tool is the Riemann-Hilbert correspondence of Emerton-Kisin \cite{EmertonKisin}.

In Section \ref{PerverseFrat} we prove Theorems \ref{mainthm} and \ref{mainthm2}. For the proofs we use the Artin-Schreier sequence to transform the problem into deciding whether certain local cohomology modules vanish or have any nonzero elements fixed by the Frobenius endomorphism. We then give a direct argument using Frobenius splitting techniques. 

In Section \ref{GlobFReg} we introduce the affine Grassmannian and prove that affine Schubert varieties are globally $F$-regular. We first apply a result of Pappas-Rapoport \cite{PappasRapoport} to reduce to the case when $G$ is simply connected. After that our proof is analogous to that of Lauritzen, Raben-Pedersen, and Thomsen \cite{SchubertFreg} in the classical setting. 

In Section \ref{FponGr} we define the category $P_{\mathcal{P}}(\mathcal{F}\ell_{\mathcal{P}}, \mathbb{F}_p)$, and we prove Theorem \ref{irreducibleobjects}. Then we define the convolution product of two perverse sheaves in $P_{L^+G}(\Gr, \mathbb{F}_p)$. The fact that the convolution product preserves perversity follows from the fact that $\IC_{\mu_1} * \IC_{\mu_2} = \IC_{\mu_1+\mu_2}$ (Theorem \ref{convsimp}). By considering the Artin–Schreier sequence, to prove that $\IC_{\mu_1} * \IC_{\mu_2} = \IC_{\mu_1+\mu_2}$ it suffices to show that the derived pushforward along a convolution morphism preserves the structure sheaf. For this we appeal to a general result due to Kovács \cite{ratsing} on rational singularities.

In Theorem \ref{cohvan}  we show that $R \Gamma (\IC_{\mu}) = \mathbb{F}_p[\dim \Gr_{\leq \mu}]$, which completes the proof of Theorem \ref{convsimp}. By contrast, in the case of $\overline{\mathbb{Q}}_\ell$-coefficients the IC sheaves generally have nonzero hypercohomology in multiple degrees. Theorem \ref{cohvan} is an immediate consequence of the global $F$-regularity of $\Gr_{\leq \mu}$. We then apply Theorem \ref{cohvan} to show that the functor $H: P_{L^+G}(\Gr, \mathbb{F}_p) \rightarrow \Vect_{\mathbb{F}_p}$  is exact and faithful. Finally, we investigate extensions between objects in $P_{L^+G}(\Gr, \mathbb{F}_p)$ in Section \ref{extsection}.

In Section \ref{AssocSection} we construct a commutativity constraint on $P_{L^+G}(\Gr, \mathbb{F}_p)$, and we show that the global cohomology functor $H$ is a symmetric monoidal functor. Our proof follows the standard technique of interpreting the convolution product as a fusion product on the Beilinson-Drinfeld Grassmannian, and we again make use of Frobenius splitting techniques (see Lemma \ref{mainlemconv}). At this point the same Tannakian formalism as in the case of $\overline{\mathbb{Q}}_\ell$-coefficients provides us with an affine monoid scheme $M_G$ such that $P_{L^+G}(\Gr, \mathbb{F}_p) \cong \Rep_{\mathbb{F}_p}(M_G)$. We conclude Section \ref{AssocSection} with a few results about $M_G$. 

Finally, in Section \ref{SatakeSection} we deduce Theorem \ref{Satthm} from Theorem \ref{convsimp} and an explicit formula for the inverse of the mod $p$ Satake isomorphism due to Herzig \cite{modpclass}.
\\
\\
\textbf{Acknowledgments.} I would first like to thank my advisor Mark Kisin for suggesting a topic that led to this paper and for his consistent guidance and encouragement. I would also like to express my gratitude to C\'{e}dric P\'{e}pin, Timo Richarz, Tobias Schmidt, and Xinwen Zhu for their comments and insights on an earlier version of this paper. I thank the referee for their careful reading of the paper and several helpful suggestions and corrections. Finally, it is a pleasure to thank Bhargav Bhatt, Justin Campbell, Dennis Gaitsgory, Michel Gros, Michael Harris, Florian Herzig, Koji Shimizu, Karen Smith, David Yang, Zijian Yao, and Yifei Zhao for their interest and helpful conversations. This material is based upon work supported by the National Science Foundation Graduate Research Fellowship Program under Grant No. DGE-1144152.

\section{Perverse $\mathbb{F}_p$-sheaves} \label{PerverseFpSheaves} For Sections \ref{PerverseFpSheaves} - \ref{PerverseFrat} we fix an algebraically closed field $k$ of characteristic $p > 0$. By a $k$-scheme we will always mean a separated scheme of finite type over $k$ unless we specify otherwise. In this section we introduce the reader to the category of perverse $\mathbb{F}_p$-sheaves on a $k$-scheme $X$, and we prove some basic facts analogous to those for $\overline{\mathbb{Q}}_\ell$-coefficients as in \cite{BBD}.

We let $D^*(X, \mathbb{F}_p)$ for $* = \emptyset$, $+$, $-$, $b$ denote the corresponding derived category of \'{e}tale $\mathbb{F}_p$-sheaves on $X$. Let $D_c^b(X, \mathbb{F}_p)$ denote the triangulated subcategory of $D^b(X,\mathbb{F}_p)$ consisting of objects having constructible cohomology sheaves. For a morphism $f \colon X \rightarrow Y$ between $k$-schemes, we have the three functors
\begin{align*}  Rf^* & \colon D(Y, \mathbb{F}_p) \rightarrow D(X, \mathbb{F}_p) \\  Rf_! & \colon D^+(X, \mathbb{F}_p) \rightarrow D^+(Y, \mathbb{F}_p) \\ - \overset{L}{\otimes}_{\mathbb{F}_p} - & \colon  D^{-}(X, \mathbb{F}_p) \times D^{-}(X, \mathbb{F}_p) \rightarrow D^{-}(X, \mathbb{F}_p). 
\end{align*}
Each of these functors restricts to a functor between $D_c^b(X, \mathbb{F}_p)$ and $D_c^b(Y, \mathbb{F}_p)$. We also have the three functors
\begin{align*}  Rf^! & \colon D^+(Y, \mathbb{F}_p) \rightarrow D^+(X, \mathbb{F}_p) \\  Rf_* & \colon D^+(X, \mathbb{F}_p) \rightarrow D^+(Y, \mathbb{F}_p) \\ R\mathscr{H}\text{om}_{\mathbb{F}_p}(-, -) & \colon  D(X, \mathbb{F}_p) \times D^{+}(X, \mathbb{F}_p) \rightarrow D(X, \mathbb{F}_p).
\end{align*}
Contrary to the case of $\mathbb{F}_\ell$-sheaves for $\ell \neq p$, these last three functors do not in general preserve constructibility. For example, the Artin-Schreier sequence implies that the affine line $\mathbb{A}^1$ has infinite-dimensional $\mathbb{F}_p$-\'{e}tale cohomology. Hence the second cohomology of $\mathbb{P}^1$ with support in a closed point is also infinite-dimensional. Moreover, one can show that there is no Verdier duality functor, and $Rf^!$ does not in general agree with a shift of $Rf^*$ when $f$ is smooth. 

Despite these drawbacks, we can still define perverse $\mathbb{F}_p$-sheaves (for the middle perversity function) in a way similar to perverse $\overline{\mathbb{Q}}_\ell$-sheaves. Fix a point $x \in X$, and let $Y$ be the reduced closed subscheme with topological space $\overline{\{x\}}$. For a complex $\mathcal{F}^\bullet \in D^+(X, \mathbb{F}_p)$ and $i \in \mathbb{Z}$ we say that $H^i (Ri_x^* \mathcal{F}^\bullet) = 0$ if the following condition holds: 
\begin{itemize}
\item There exists a dense open subscheme $U \subset Y$ such that $H^i(Ri_U^* \mathcal{F}^\bullet) =0$, where $i_U \colon U \rightarrow X$ is the immersion.
\end{itemize}
 We can similarly define what it means to have $H^i (Ri_x^! \mathcal{F}^\bullet) = 0$. Now consider the following two subcategories of $D^+(X,\mathbb{F}_p)$:
\begin{align*}
^{p}D^{\leq 0}(X, \mathbb{F}_p) & = \{\mathcal{F}^\bullet \in D^+(X, \mathbb{F}_p) \: : \:H^i (Ri_x^* \mathcal{F}^\bullet) = 0 \text{ for all } x \in X \text{ and } i > -\dim \overline{\{x\}}  \}\\
^{p}D^{\geq 0}(X, \mathbb{F}_p) & = \{\mathcal{F}^\bullet \in D^+(X, \mathbb{F}_p) \: : \: H^i(Ri_x^!\mathcal{F}^\bullet) = 0 \text{ for all } x \in X \text{ and } i < -\dim \overline{\{x\}} \}.
\end{align*}

Gabber \cite{Gabbert} has shown that these subcategories underlie a $t$-structure on $D^+(X,\mathbb{F}_p)$. Moreover, he has shown that the truncation functors preserve the subcategory $D_c^b(X, \mathbb{F}_p)$. We denote the category of perverse sheaves in $D^+(X, \mathbb{F}_p)$ by $$P^+(X, \mathbb{F}_p): = {^{p}D}^{\leq 0}(X, \mathbb{F}_p) \cap {^{p}D}^{\geq 0}(X, \mathbb{F}_p). $$ Let $P^b(X, \mathbb{F}_p)$ (resp. $P_c^b(X, \mathbb{F}_p)$) be the full subcategory of perverse sheaves in $D^b(X, \mathbb{F}_p)$ (resp. $D_c^b(X, \mathbb{F}_p)$). 

\begin{rmrk} \label{defrem}
For $x \in X$, let $j \colon \Spec(\mathcal{O}_{X,x}^{\text{sh}}) \rightarrow X$ be the strict henselization of the local ring of $X$ at $x$, and let $i \colon \overline{x} \rightarrow \Spec(\mathcal{O}_{X,x}^{\text{sh}})$ be the inclusion of the closed point. Then the condition $H^i(Ri^*_x \mathcal{F}^\bullet) = 0$ is equivalent to $H^i(Ri^* (Rj^* \mathcal{F}^\bullet)) = 0$. Similarly, the condition $H^i(Ri^!_x \mathcal{F}^\bullet) = 0$ is equivalent to $H^i(Ri^!( Rj^* \mathcal{F}^\bullet)) = 0$ (see {\cite{Gabbert}}).
\end{rmrk}

\begin{rmrk} When $X$ is smooth, Emerton and Kisin have given an alternative description of $P_c^b(X, \mathbb{F}_p)$ as follows. For any $k$-algebra $R$ let $R[F]$ be the non-commutative polynomial algebra in one variable $F$ such that $Fa = a^p F$ for $a \in R$. Let $\mathcal{O}_{F,X}$ be the sheaf of non-commutative rings whose value on an open affine $\Spec(R)$ is $R[F]$. If $\mathcal{M}$ is an $\mathcal{O}_{F,X}$-module there is a natural map of $\mathcal{O}_X$-modules $$\phi_M \colon F^* \mathcal{M} \rightarrow \mathcal{M},$$ where $F$ is the absolute Frobenius endomorphism on $X$. The $\mathcal{O}_{F,X}$-module $\mathcal{M}$ is said to be \emph{locally finitely generated unit} (lfgu) if it is quasi-coherent as an $\mathcal{O}_X$-module, $\phi_M$ is an isomorphism and $\mathcal{M}$ is, locally on $X$, finitely generated over $\mathcal{O}_{F,X}$. Let $D_{\text{lfgu}}^b(\mathcal{O}_{F,X})$ be the bounded derived category of $\mathcal{O}_{F,X}$-modules whose cohomology sheaves are lfgu $\mathcal{O}_{F,X}$-modules. When $X$ is smooth the category $D_c^b(X, \mathbb{F}_p)$ is anti-equivalent to the category $D_{\text{lfgu}}^b(\mathcal{O}_{F,X})$ by \cite[11.3]{EmertonKisin}. Under this equivalence, the perverse $t$-structure on $D_c^b(X, \mathbb{F}_p)$ corresponds to the standard $t$-structure on $D_{\text{lfgu}}^b(\mathcal{O}_{F,X})$ by \cite[11.5.4]{EmertonKisin}. If $f \colon Y \rightarrow X$ is a morphism between smooth $k$-schemes then Emerton and Kisin have also given a description of the functors $Rf^*$, $Rf_!$, and - $\overset{L}{\otimes}_{\mathbb{F}_p} - $ in terms of functors on quasi-coherent sheaves.
\end{rmrk}

We now establish some basic properties of perverse $\mathbb{F}_p$-sheaves. To begin, we have the following result of Gabber, which is in fact valid for a larger class of noetherian schemes defined over $\mathbb{F}_p$.

\begin{thm}[{\cite[12.4]{Gabbert}}] \label{artinian} Let $X$ be a $k$-scheme. \leavevmode
\begin{enumerate}[{\normalfont (i)}]

\item Every object of $P_c^b(X, \mathbb{F}_p)$ is artinian and noetherian.
\item Every perverse subquotient of an object of $P_c^b(X, \mathbb{F}_p)$ that lies in $P^+(X, \mathbb{F}_p)$ in fact lies in $P_c^b(X, \mathbb{F}_p)$.
\end{enumerate}
\end{thm}
It follows that  $P_c^b(X, \mathbb{F}_p)$ is an abelian category. Given a complex $\mathcal{F}^\bullet \in D^+(X, \mathbb{F}_p)$, we let $^{p}H^i(\mathcal{F}^\bullet) \in P^+(X, \mathbb{F}_p)$ denote its $i$th perverse cohomology sheaf. If $\mathcal{F}^\bullet \in D_c^b(X, \mathbb{F}_p)$, then $^{p}H^i(\mathcal{F}^\bullet) \in P_c^b(X, \mathbb{F}_p)$. For a morphism between $k$-schemes $f \colon X \rightarrow Y$, define $$^pf_*\mathcal{F}^\bullet := {^pH}^0(Rf_* \mathcal{F}^\bullet) \in P^+(Y, \mathbb{F}_p).$$ By this recipe we get functors
\begin{align*}
{^pf}_*, {^pf}_! & \colon  P^+(X, \mathbb{F}_p) \rightarrow P^+(Y, \mathbb{F}_p) \\
{^pf}^*, {^pf}^! & \colon P^+(Y, \mathbb{F}_p) \rightarrow P^+(X, \mathbb{F}_p).
\end{align*}
In general, two of these restrict to functors
\begin{align*}
{^pf}_! & \colon  P_c^b(X, \mathbb{F}_p) \rightarrow P_c^b(Y, \mathbb{F}_p) \\
{^pf}^* & \colon P_c^b(Y, \mathbb{F}_p) \rightarrow P_c^b(X, \mathbb{F}_p).
\end{align*}

Let $j \colon U \rightarrow X$ be an open immersion of $k$-schemes, and let $i \colon Z \rightarrow X$ be a complementary closed immersion. In the following lemma, all functors are defined on the categories $D^+(X, \mathbb{F}_p)$, $D^+(U, \mathbb{F}_p)$, $D^+(Z, \mathbb{F}_p)$, or their corresponding hearts for the perverse $t$-structure.

\begin{lem} \label{properties} \leavevmode
\begin{enumerate}[{\normalfont (i)}]
\item $Rj^*$ and $Ri_*$ are $t$-exact, and the functors ${^pj}^*$ and ${^pi}_*$ are exact.
\item $Rj_!$ and $Ri^*$ are right $t$-exact, i.e. $$Rj_!({^pD}^{\leq 0}(U,\mathbb{F}_p)) \subset {^pD}^{\leq 0}(X, \mathbb{F}_p), \quad \quad \quad \quad  Ri^*({^pD}^{\leq 0}(X,\mathbb{F}_p)) \subset {^pD}^{\leq 0}(Z, \mathbb{F}_p),$$ and the functors ${^pj}_!$ and ${^pi}^*$ are right exact.
\item $Rj_*$ and $Ri^!$ are left t-exact, and the functors ${^pj}_*$ and ${^pi}^!$ are left exact.
\item The triples $({^pj}_!, {^pj}^*, {^pj}_*)$ and $({^pi}^*, {^pi}_*, {^pi}^!)$ are adjoint sequences.
\item We have
$${^pj}^* \circ {^pi}_* = 0, \quad \quad {^pi}^* \circ {^pj}_! = 0, \quad \quad {^pi}^! \circ {^pj}_* = 0.$$
\item We have
$${^pj}^* \circ {^pj}_* = \id, \quad \quad {^pj}^* \circ {^pj}_! = \id, \quad \quad {^pi}^* \circ {^pi}_* = \id, \quad \quad {^pi}^! \circ {^pi}_* = \id. $$
\item For $\mathcal{F}^\bullet \in P^+(X, \mathbb{F}_p)$ there are exact sequences
$$ 0 \rightarrow {^pi}_* {^pH}^{-1}(Ri^* \mathcal{F}^\bullet) \rightarrow {^pj}_! {^pj}^* \mathcal{F}^\bullet \rightarrow \mathcal{F}^\bullet \rightarrow {^pi}_* {^pi}^* \mathcal{F}^\bullet \rightarrow 0$$ and
$$0 \rightarrow {^pi}_* {^pi}^! \mathcal{F}^\bullet \rightarrow \mathcal{F}^\bullet \rightarrow {^pj}_* {^pj}^* \mathcal{F}^\bullet \rightarrow {^pi}_* {^pH}^1(Ri^!\mathcal{F}^\bullet) \rightarrow 0.$$
\item The functor ${^pi}_*$ is fully faithful. Its essential image consists of perverse sheaves $\mathcal{F}^{\bullet} \in P^+(X, \mathbb{F}_p)$ such that ${^pj}^* \mathcal{F}^\bullet = 0$.
\end{enumerate}
\end{lem}

\begin{proof}
The analogous facts are true for $\overline{\mathbb{Q}}_\ell$-coefficients \cite{BBD}. The proofs are the same, so we only provide a sketch. Parts (i) - (iii) are straightforward. The proof of parts (iv) - (vi) uses (i) - (iii) and the analogous facts for the corresponding non-truncated functors. For part (vii), apply the other parts to the exact triangles
$$Rj_! Rj^* \mathcal{F}^\bullet \rightarrow \mathcal{F}^\bullet \rightarrow Ri_* Ri^* \mathcal{F}^\bullet$$ and
$$Ri_* Ri^! \mathcal{F}^\bullet \rightarrow \mathcal{F}^\bullet \rightarrow Rj_* Rj^* \mathcal{F}^\bullet.$$ For part (viii) we note that ${^pi}_*$ is fully faithful because of part (i) and the fact that $Ri_*$ is fully faithful. We have already seen that ${^pj}^* \circ {^pi}_* = 0$, so let $\mathcal{F}^\bullet \in P^+(X, \mathbb{F}_p)$ be such that ${^pj}^* \mathcal{F}^\bullet =0$. Now apply the first exact sequence in part (vii) to see that $\mathcal{F}^\bullet$ is in the essential image of ${^pi}_*$. 
\end{proof}

When $X$ is smooth, Emerton-Kisin \cite{EmertonKisinIntro} defined the intermediate extension $j_{!*} \mathcal{F}^\bullet$ of a perverse sheaf $\mathcal{F}^\bullet \in P_c^b(U, \mathbb{F}_p)$ along an immersion $j \colon U \rightarrow X$, where $U$ is also smooth. The construction uses their theory of unit $\mathcal{O}_{F,X}$-modules. They also showed that when $X$ is smooth, a perverse sheaf $\mathcal{F}^\bullet \in P_c^b(X,\mathbb{F}_p)$ is simple if and only if it is of the form $j_{!*} \mathcal{L}[\dim U]$ where $U$ is smooth, $j \colon U \rightarrow X$ is an immersion, and $\mathcal{L}$ is a simple local system of \'{e}tale $\mathbb{F}_p$-sheaves on $U$. We now investigate the properties of the intermediate extension functor for general $k$-schemes. 

Let $X$ be a $k$-scheme and let $j \colon U \rightarrow X$ be an immersion. Suppose we have a factorization of $j$ into a closed immersion $i \colon U \rightarrow Z$ followed by an open immersion $h \colon Z \rightarrow X$. By Lemma \ref{properties} (i) we have canonical isomorphisms of functors ${^ph}_! \circ {^pi}_* \cong {^pj}_!$ and ${^ph}_* \circ {^pi_*} \cong {^pj}_*$. Thus by Lemma \ref{properties} (iv) there is a natural map ${^pj}_! \mathcal{F}^\bullet \rightarrow {^pj}_* \mathcal{F}^\bullet$ for $\mathcal{F}^\bullet \in P^+(U, \mathbb{F}_p)$.

\begin{defn} The \emph{intermediate extension} of $\mathcal{F}^\bullet \in P^+(U, \mathbb{F}_p)$ along the immersion $j \colon U \rightarrow X$ is
$$j_{!*} \mathcal{F}^\bullet := \text{im} ({^pj}_! \mathcal{F}^\bullet \rightarrow {^pj}_* \mathcal{F}^\bullet).$$  
\end{defn} This defines a functor $${j_{!*}} \colon P^+(U, \mathbb{F}_p) \rightarrow P^+(X, \mathbb{F}_p).$$ We note that $j_{!*} \mathcal{F}^\bullet \in P_c^b(X, \mathbb{F}_p)$ if $\mathcal{F}^\bullet \in P_c^b(U, \mathbb{F}_p)$ by Theorem \ref{artinian} (ii). By \cite[4.3.1]{EmertonKisinIntro}, we recover Emerton-Kisin's intermediate extension functor when $U$ and $X$ are smooth. The following lemma can be proved using the properties established in Lemma \ref{properties}.

\begin{lem} \label{compintext}
Suppose we have a factorization of $j \colon U \rightarrow X$ as a composition of immersions $j= j_1 \circ \cdots \circ j_n$ between $k$-schemes. Then we have a canonical isomorphism of functors
$${j_{!*}} \xrightarrow{\sim} j_{1, !*} \circ \cdots \circ j_{n, !*}. $$
\end{lem}

If $\mathcal{F}^\bullet \in D(U, \mathbb{F}_p)$ we say that $\mathcal{G}^\bullet \in D(X, \mathbb{F}_p)$ is an \emph{extension} of $\mathcal{F}^\bullet$ if there is an isomorphism $Rj^* \mathcal{G}^\bullet \cong \mathcal{F}^\bullet$. As in the case of $\overline{\mathbb{Q}}_\ell$-sheaves, the next lemma shows that the intermediate extension can be characterized as a unique minimal extension. 

\begin{lem} \label{intext}
Let $j \colon U \rightarrow X$ be an open immersion of $k$-schemes and let $\mathcal{F}^\bullet \in P^+(U, \mathbb{F}_p)$. Let $i \colon Z \rightarrow X$ be a complementary closed immersion. The intermediate extension $j_{!*} \mathcal{F}^\bullet$ is characterized by either of the two equivalent conditions:

\begin{enumerate}[{\normalfont (i)}]
\item The perverse sheaf $j_{!*} \mathcal{F}^\bullet$ is the unique perverse extension $\mathcal{G}^\bullet$ of $\mathcal{F}^\bullet$ such that $$Ri^* \mathcal{G}^\bullet \in {^pD}^{\leq -1}(Z, \mathbb{F}_p)\quad \quad \text{ and } \quad \quad Ri^! \mathcal{G}^\bullet \in {^pD}^{\geq 1}(Z, \mathbb{F}_p).$$
\item The perverse sheaf $j_{!*} \mathcal{F}^\bullet$ is the unique perverse extension $\mathcal{G}^\bullet$ of $\mathcal{F}^\bullet$ such that $\mathcal{G}^\bullet$ has no nonzero subobject or quotient in the essential image of the functor ${^pi}_* \colon P^+(Z, \mathbb{F}_p) \rightarrow P^+(X, \mathbb{F}_p)$. 
\end{enumerate}
\end{lem}

\begin{proof}
We first show that $j_{!*} \mathcal{F}^\bullet$ satisfies the conditions in (i) except for possibly uniqueness. The complex $j_{!*} \mathcal{F}^\bullet$ is an extension of $\mathcal{F}^\bullet$ because ${^pj}^*$ is exact and ${^pj}^* \circ {^pj}_! = {^pj}^* \circ {^pj}_* = \id$. As ${^pi}^*$ is right exact we get a surjection ${^pi}^* ({^pj}_! \mathcal{F}^\bullet) \twoheadrightarrow {^pi}^*(j_{!*} \mathcal{F}^\bullet)$. Since ${^pi}^* \circ {^pj}_! = 0$ it follows that ${^pi}^*(j_{!*} \mathcal{F}^\bullet) = 0$. Similarly by applying the left exact functor ${^pi}^!$ to the injection $ j_{!*} \mathcal{F}^\bullet \hookrightarrow {^pj}_* \mathcal{F}^\bullet $ we get that ${^pi}^!( j_{!*} \mathcal{F}^\bullet )= 0$. Now because $Ri^*$ is right $t$-exact and $Ri^!$ is left $t$-exact we conclude that $Ri^* (j_{!*} \mathcal{F}^\bullet) \in {^pD}^{\leq -1}(Z, \mathbb{F}_p)$ and $Ri^! (j_{!*} \mathcal{F}^\bullet) \in {^pD}^{\geq 1}(Z, \mathbb{F}_p)$. 

We now make a general observation. Suppose $\mathcal{G}^\bullet \in P^+(X, \mathbb{F}_p)$ is an extension of $\mathcal{F}^\bullet$. Then from the natural maps ${^pj}_! {^pj}^* \mathcal{G}^\bullet \rightarrow \mathcal{G}^\bullet$ and $\mathcal{G}^\bullet \rightarrow {^pj}_* {^pj}^* \mathcal{G}^\bullet$, we get maps ${^pj}_! \mathcal{F}^\bullet \rightarrow \mathcal{G}^\bullet$ and $\mathcal{G}^\bullet \rightarrow {^pj}_* \mathcal{F}^\bullet$ such that the composition is the natural map ${^pj}_! \mathcal{F}^\bullet \rightarrow {^pj}_* \mathcal{F}^\bullet$. Now suppose $\mathcal{G}^\bullet$ satisfies the conditions in (i) except for possibly uniqueness. Then from the above considerations and the exact sequences in Lemma \ref{properties} (vii) we get maps
$${^pj}_! \mathcal{F}^\bullet \twoheadrightarrow \mathcal{G}^\bullet \hookrightarrow {^pj}_* \mathcal{F}^\bullet$$ such that the composition is the natural map ${^pj}_! \mathcal{F}^\bullet \rightarrow {^pj}_* \mathcal{F}^\bullet$. From this it follows that $\mathcal{G}^\bullet \cong j_{!*} \mathcal{F}^\bullet$. 

For (ii) we start with a couple of general observations. For $\mathcal{G}^\bullet \in P^+(X, \mathbb{F}_p)$ we have a surjection $\mathcal{G}^\bullet \twoheadrightarrow {^pi}_* {^pi}^* \mathcal{G}^\bullet$ by Lemma \ref{properties} (vii). By adjunction and Lemma \ref{properties} (viii) it follows that ${^pi}_* {^pi}^* \mathcal{G}^\bullet$ is the largest quotient of $\mathcal{G}^\bullet$ in the essential image of ${^pi}_* \colon P^+(Z, \mathbb{F}_p) \rightarrow P^+(X, \mathbb{F}_p)$. Similarly ${^pi}_* {^pi}^! \mathcal{G}^\bullet$ is the largest subobject of $\mathcal{G}^\bullet$ in the essential image of ${^pi}_*$. Now the fact that $j_{!*} \mathcal{F}^\bullet$ is uniquely characterized by (ii) follows from (i). 
\end{proof}

\begin{lem} \label{injsur}
Let $j \colon U \rightarrow X$ be an immersion of $k$-schemes. Then the intermediate extension functor $j_{!*} \colon P^+(U, \mathbb{F}_p) \rightarrow P^+(X, \mathbb{F}_p)$ preserves injections and surjections.
\end{lem}

\begin{proof}
This follows from the definitions, and the fact that ${^pj}_!$ is right exact and ${^pj}_*$ is left exact. 
\end{proof}

\begin{lem} \label{simplelem}
Let $U$ be a $k$-scheme and let $\mathcal{F}^\bullet \in P^+(U, \mathbb{F}_p)$ be simple. Suppose $j \colon U \rightarrow X$ is an immersion into a $k$-scheme $X$. Then $j_{!*} \mathcal{F}^\bullet$ is simple.
\end{lem}
\begin{proof}
It suffices to treat the case of an open immersion and a closed immersion separately. First suppose $j$ is an open immersion, and that we have an exact sequence
$$0 \rightarrow \mathcal{G}^\bullet \rightarrow {j_{!*}}\mathcal{F}^\bullet \rightarrow \mathcal{H}^\bullet \rightarrow 0$$ in $P^+(X, \mathbb{F}_p)$. Applying the exact functor ${^pj}^*$, we get that ${^pj}^* \mathcal{G}^\bullet = 0$ or ${^pj}^* \mathcal{H}^\bullet = 0$ because ${^pj}^* (j_{!*}\mathcal{F}^\bullet) = \mathcal{F}^\bullet$ is simple. By Lemma \ref{properties} (viii) and Lemma \ref{intext}, this means that $\mathcal{G}^\bullet = 0$ or $\mathcal{H}^\bullet = 0$. 

If $j$ is a closed immersion, let $h \colon Z \rightarrow X$ be a complementary open immersion. Then $j_{!*} \mathcal{F}^\bullet = {^pj}_* \mathcal{F}^\bullet$, and so ${^ph}^* (j_{!*}\mathcal{F}^\bullet)  = 0$ by Lemma \ref{properties} (v). Now starting with an exact sequence as above, it follows that ${^ph}^* \mathcal{G}^\bullet  = 0$ and ${^ph}^* \mathcal{H}^\bullet = 0$. Hence by Lemma \ref{properties} (vi) and (viii), the given exact sequence comes from applying ${^pj}_*$ to an exact sequence $$ 0 \rightarrow \mathcal{G}^{' \bullet} \rightarrow \mathcal{F}^\bullet \rightarrow \mathcal{H}^{' \bullet} \rightarrow 0$$ in $P^+(U, \mathbb{F}_p)$. Here we have that $\mathcal{G}^{' \bullet} = 0$ or $\mathcal{H}^{' \bullet} = 0$ because $\mathcal{F}^\bullet$ is simple.
\end{proof} 

\begin{lem} \label{simplelem2}
Let $j \colon U \rightarrow X$ be an open immersion of $k$-schemes, and suppose $\mathcal{F}^\bullet \in P^+(X, \mathbb{F}_p)$ is simple. Then if ${^pj}^* \mathcal{F}^\bullet \neq 0$, we have $\mathcal{F}^\bullet = j_{!*} ({^pj}^* \mathcal{F}^\bullet )$ and ${^pj}^* \mathcal{F}^\bullet$ is simple. 
\end{lem}

\begin{proof}
For the first part, since $\mathcal{F}^\bullet$ is an extension of ${^pj}^* \mathcal{F}^\bullet$, then as in the proof of Lemma \ref{intext} (ii) we have natural maps ${^pj}_!( {^pj}^* \mathcal{F}^\bullet) \rightarrow \mathcal{F}^\bullet$ and $\mathcal{F}^\bullet \rightarrow {^pj}_* ({^pj}^* \mathcal{F}^\bullet)$ such that the composition is the natural map ${^pj}_!( {^pj}^* \mathcal{F}^\bullet)  \rightarrow  {^pj}_* ({^pj}^* \mathcal{F}^\bullet)$. Because $\mathcal{F}^\bullet$ is simple, and each of these two maps restricts to a nonzero map under ${^pj}^*$, it follows that ${^pj}_!( {^pj}^* \mathcal{F}^\bullet) \rightarrow \mathcal{F}^\bullet$ is surjective and  $\mathcal{F}^\bullet \rightarrow {^pj}_* ({^pj}^* \mathcal{F}^\bullet)$ is injective. Thus $\mathcal{F}^\bullet = j_{!*} ({^pj}^* \mathcal{F}^\bullet )$. To see that ${^pj}^* \mathcal{F}^\bullet$ is simple, it now suffices to note that $j_{!*}$ preserves injections and surjections (Lemma \ref{injsur}), and that ${^pj}^* \circ j_{!*} = \id$. 
\end{proof}

\begin{lem} \label{homintext}
Let $j \colon U \rightarrow X$ be an immersion of $k$-schemes, and suppose $\mathcal{F}_1^\bullet$, $\mathcal{F}_2^\bullet \in P^+(U, \mathbb{F}_p)$. Then
$$\Hom(\mathcal{F}_1^\bullet, \mathcal{F}_2^\bullet) = \Hom(j_{!*} \mathcal{F}_1^\bullet, j_{!*} \mathcal{F}_2^\bullet).$$
\end{lem}

\begin{proof} This follows from the definition of $j_{!*}$ and the adjunctions in Lemma \ref{properties} (iv).
\end{proof}

\begin{defn}
Let $X$ be a $k$-scheme and let $\mathcal{F}^\bullet \in D^b_c(X, \mathbb{F}_p)$. The \emph{support} of $\mathcal{F}^\bullet$ is the union of the supports of each of the cohomology sheaves $H^i(\mathcal{F}^\bullet)$. 
\end{defn}

The support of $\mathcal{F}$ is a constructible subset of $X$. We now prove an analogue of the classification of simple perverse sheaves in \cite[4.3.1 (ii)]{BBD}.

\begin{thm} \label{simplechar}
Let $X$ be a $k$-scheme, and let $j \colon U \rightarrow X$ be an immersion from an irreducible smooth $k$-scheme $U$. Suppose $\mathcal{L}$ is a simple local system of \'{e}tale $\mathbb{F}_p$-sheaves on $U$. Then the complex $\mathcal{L}[\dim U] \in D^b_c(U, \mathbb{F}_p)$ is perverse, and $j_{!*} \mathcal{L}[\dim U]$ is simple. Conversely, every simple perverse sheaf in $P^b_c(X, \mathbb{F}_p)$ is of the form $j_{!*} \mathcal{L}[\dim U]$ for $j \colon U \rightarrow X$ an immersion from an irreducible smooth $k$-scheme and $\mathcal{L}$ a simple local system of \'{e}tale $\mathbb{F}_p$-sheaves on $U$.
\end{thm}

\begin{proof}
It is shown in \cite[4.3.3]{EmertonKisinIntro} that if $U$ is an irreducible smooth $k$-scheme and $\mathcal{L}$ is a simple local system of \'{e}tale $\mathbb{F}_p$-sheaves on $U$, then $\mathcal{L}[\dim U]$ is a simple perverse sheaf. Hence if $j \colon U \rightarrow X$ is an immersion then $j_{!*} \mathcal{L}[\dim U]$ is simple by Lemma \ref{simplelem}. Conversely, suppose $\mathcal{F}^\bullet \in P_c^b(X, \mathbb{F}_p)$ is simple. If we let $Z$ be the closure of the support of $\mathcal{F}^\bullet$ (with its reduced induced structure), then we can replace $X$ by $Z$ so we can assume the support of $\mathcal{F}^\bullet$ is dense in $X$. Since $k$ is perfect, we may choose an open immersion $j \colon U \rightarrow X$ such that $U$ is a smooth $k$-scheme and $U$ is contained in the support of $\mathcal{F}^\bullet$. By Lemma \ref{simplelem2} we can replace $X$ by $U$ and so we can assume $X$ is smooth. Now the fact that $\mathcal{F}^\bullet$ is of the desired form is proved in  \cite[4.3.3]{EmertonKisinIntro}. 
\end{proof}

We now investigate the behavior of perverse sheaves under smooth pullback. 

\begin{lem} \label{etalep}
Let $f \colon Y \rightarrow X$ be an \'{e}tale morphism between $k$-schemes, and let $\mathcal{F}^\bullet \in P^+(X, \mathbb{F}_p)$. Then $Rf^* \mathcal{F}^\bullet \in P^+(Y, \mathbb{F}_p)$.
\end{lem}

\begin{proof}
This follows from the definitions and the fact that $Rf^* = Rf^!$.
\end{proof}

\begin{lem} \label{smoothp}
Let $f \colon Y \rightarrow X$ be a smooth morphism of relative dimension $d$ between $k$-schemes and let $\mathcal{F}^\bullet \in P_c^b(X, \mathbb{F}_p)$. Then $Rf^*[d] \mathcal{F}^\bullet \in P_c^b(Y, \mathbb{F}_p)$.
\end{lem}

\begin{proof}
We know that $Rf^*[d] \mathcal{F}^\bullet \in D_c^b(Y, \mathbb{F}_p)$ so we only need to prove that $Rf^*[d] \mathcal{F}^\bullet$ is perverse. The problem is Zariski local on $Y$ so we may assume that we have a factorization $Y \xrightarrow{g} \mathbb{A}^d_X \xrightarrow{\pi} X$ where $g$ is \'{e}tale and $\pi$ is the projection. Thus by Lemma \ref{etalep} it suffices to prove the result for the projection $\pi \colon \mathbb{A}_X^d \rightarrow X$. Observe that we can further reduce to the case $X = \Spec(A)$. Now let $i \colon X \rightarrow \mathbb{A}^n$ be a closed immersion for some $n$, and consider the Cartesian diagram
$$\xymatrix{
\mathbb{A}^d_X  \ar[r]^{i'} \ar[d]^{\pi} & \mathbb{A}^{n+d} \ar[d]^{\pi'} \\ 
X \ar[r]^{i} &  \mathbb{A}^n.
}$$ The map $\pi'$ is the projection onto the first $n$ coordinates. Note that by Lemma  \ref{properties}, the complex $R \pi^*[d] \mathcal{F}^\bullet$ is perverse if and only if $R i'_*(R \pi^*[d] \mathcal{F}^\bullet)$ is perverse. Thus, by the proper base change theorem we can reduce to the case of the projection $\pi \colon  \mathbb{A}^{n+d} \rightarrow \mathbb{A}^n$ onto the first $n$ coordinates. In particular, we have reduced to the case in which $X$ and $Y$ are smooth. Now we can appeal to the Riemann-Hilbert correspondence of Emerton-Kisin \cite{EmertonKisin}. The functor $Rf^*[d] \colon D_c^b(X, \mathbb{F}_p) \rightarrow D_c^b(Y, \mathbb{F}_p)$ corresponds to the functor $$f^![-d] \colon D_{\lfgu}^b(\mathcal{O}_{F,X}) \rightarrow D_{\lfgu}^b(\mathcal{O}_{F,Y}).$$ On the level of quasi-coherent sheaves the functor $f^![-d]$ is the left derived functor of the usual pullback functor for quasi-coherent sheaves. Because $f$ is flat this functor is already exact, and so $f^![-d]$ is $t$-exact for the standard $t$-structures on $D_{\lfgu}^b(\mathcal{O}_{F,X})$ and $D_{\lfgu}^b(\mathcal{O}_{F,Y})$. These $t$-structures correspond to the perverse $t$-structures on $D_c^b(X, \mathbb{F}_p)$ and $D_c^b(Y, \mathbb{F}_p)$. 
\end{proof}

We now show that intermediate extensions commute with smooth pullbacks. St\"{a}bler \cite{StabPerv} has given a proof of this result for $k$-schemes that admit closed immersions into smooth $k$-schemes. His proof uses the category of Cartier crystals, which we will not use here. In the first part of the proof we reduce to the case of smooth $k$-schemes. The theorem then follows from St\"{a}bler's result. We also give a direct argument in the smooth case using unit $\mathcal{O}_{F,X}$-modules in Lemma \ref{smoothlem}.

\begin{thm} \label{smoothint}
Let $f \colon Y \rightarrow X$ be a smooth morphism of relative dimension $d$ between $k$-schemes, and let $j \colon U \rightarrow X$ be an immersion. Consider the Cartesian diagram
$$\xymatrix{
f^{-1}(U)  \ar[r]^{j'} \ar[d]^{f'} & Y \ar[d]^{f} \\ 
U \ar[r]^{j} &  X
}$$
There is an isomorphism of functors $P_c^b(U, \mathbb{F}_p) \rightarrow P_c^b(Y, \mathbb{F}_p)$: $$Rf^*[d] \circ j_{!*} \xrightarrow{\sim} j'_{!*} \circ Rf'^*[d].$$  
\end{thm}

\begin{proof}
It suffices to treat the case of a closed immersion and open immersion separately. If $j$ is a closed immersion then $j_{!*} = {^pj}_* = Rj_*$ and so the result is a special case of the proper base change theorem. Now suppose $j$ is an open immersion and let $\mathcal{F}^\bullet \in P_c^b(U, \mathbb{F}_p)$. We will show that $Rf^*[d] ( j_{!*} \mathcal{F}^\bullet) \cong j'_{!*} ( Rf'^*[d] \mathcal{F}^\bullet)$. First observe that $Rf^*[d](j_{!*} \mathcal{F}^\bullet)$ is an extension of $Rf'^*[d] \mathcal{F}^\bullet$. Moreover, from the characterization of $j'_{!*}(Rf'^*[d] \mathcal{F}^\bullet)$ in Lemma \ref{intext} the problem is Zariski local on $Y$. Thus we may assume we have a factorization $Y \xrightarrow{g} \mathbb{A}^d_X \xrightarrow{\pi} X$ where $g$ is \'{e}tale and $\pi$ is the projection. It suffices to treat the cases when $f$ is \'{e}tale and when $f$ is the projection $\mathbb{A}_X^d \rightarrow X$ separately. 

For the case when $f$ is \'{e}tale, let $i \colon Z \rightarrow X$ be a closed immersion complementary to $j$ (with the reduced induced structure on $Z$), and consider the Cartesian diagram
$$
\xymatrix{
Y   \ar[d]^{f} & f^{-1}(Z) \ar[d]^{g} \ar[l]_{i'} \\ 
X  &  Z \ar[l]^{i}
}
$$
Using the fact that $Rg^*$ is $t$-exact we verify that $$Ri'^*(Rf^* (j_{!*} \mathcal{F}^\bullet)) =  Rg^*(Ri^* (j_{!*} \mathcal{F}^\bullet)) \in {^pD}^{\leq -1}(f^{-1}(Z), \mathbb{F}_p).$$ Since $Rf^* = Rf^!$ for \'{e}tale morphisms we can similarly verify that $$Ri'^!(Rf^* (j_{!*} \mathcal{F}^\bullet)) \in {^pD}^{\geq 1}(f^{-1}(Z), \mathbb{F}_p).$$ This takes care of the case when $f$ is \'{e}tale. 

For later use we note that if we only assume $f$ is smooth then by Lemma \ref{intext} (i) and Lemma \ref{smoothp}, 
$$Ri'^*(Rf^*[d] (j_{!*} \mathcal{F}^\bullet)) =  Rg^*[d](Ri^* (j_{!*} \mathcal{F}^\bullet)) \in {^pD}^{\leq -1}(f^{-1}(Z), \mathbb{F}_p).$$ Since $Rf^*[d](j_{!*} \mathcal{F}^\bullet)$ is an extension of $Rf'^*[d] \mathcal{F}^\bullet$, then by Lemma \ref{properties} (vii) there is a surjection
$${^p}j'_!(Rf'^*[d] \mathcal{F}^\bullet) \twoheadrightarrow Rf^*[d](j_{!*} \mathcal{F}^\bullet).$$ As in the proof of Lemma \ref{intext} (i) there is also a map $Rf^*[d](j_{!*} \mathcal{F}^\bullet) \rightarrow {^p}j'_* (Rf'^*[d] \mathcal{F}^\bullet)$ whose composition with the above map gives the natural map ${^p}j'_!(Rf'^*[d] \mathcal{F}^\bullet) \rightarrow {^p}j'_* (Rf'^*[d] \mathcal{F}^\bullet)$. Thus, only assuming $f$ is smooth we always have a surjection
\begin{equation} \label{step1}
Rf^*[d](j_{!*} \mathcal{F}^\bullet) \twoheadrightarrow j'_{!*}(Rf'^*[d] \mathcal{F}^\bullet).
\end{equation}

For the case when $f$ is the projection $\mathbb{A}_X^d \rightarrow X$, first note that we can assume $d=1$ by factoring $f$ as a composition of projections of relative dimension $1$. Let $h \colon \Spec(A) \rightarrow X$ be an open affine, and consider the Cartesian diagram
$$\xymatrix{
\Spec(A) \times \mathbb{A}^1  \ar[r] \ar[d] & X \times \mathbb{A}^1 \ar[d]^{f}  \\ 
\Spec(A) \ar[r]^h &  X
}$$
Note that $Rh^*( j_{!*}\mathcal{F}^\bullet)$ is the intermediate extension of $\mathcal{F}^\bullet$ restricted to $\Spec(A) \cap U$ along the map $\Spec(A) \cap U \rightarrow \Spec(A)$. Thus we can further assume $X=\Spec(A)$ is affine. 

Let $\alpha \colon \Spec(A) \rightarrow \Spec(B)$ be a closed immersion into a smooth affine $k$-scheme $\Spec(B)$, and let $\pi \colon \Spec(B) \times \mathbb{A}^1 \rightarrow \Spec(B)$ be the projection. Using proper base change one can verify that $Rf^*[1] ( j_{!*} \mathcal{F}^\bullet) \cong j'_{!*} ( Rf'^*[1] \mathcal{F}^\bullet)$ if and only if $R\pi^*[1](R \alpha_* (j_{!*} \mathcal{F}^\bullet))$ is the intermediate extension of $Rf'^*[1] \mathcal{F}^\bullet$ along the composition $$f^{-1}(U) \rightarrow \Spec(A) \times \mathbb{A}^1 \rightarrow \Spec(B)\times \mathbb{A}^1.$$
Thus we can replace $\Spec(A)$ by $\Spec(B)$ at the cost of $j$ now being an open immersion followed by a closed immersion. However, we can refactor $j$ as a closed immersion followed by an open immersion. We already proved the result for closed immersions, and the desired result holds for a composition of immersions if it holds for each immersion separately. Thus, we have reduced to the following setup: $X = \Spec(A)$ is affine and smooth, $Y = X \times \mathbb{A}^1$, $f \colon Y \rightarrow X$ is the projection, and $j \colon U \rightarrow X$ is an open immersion. As mentioned before we can now appeal to \cite{StabPerv} to get an isomorphism $$Rf^*[1] ( j_{!*} \mathcal{F}^\bullet) \xrightarrow{\sim} j'_{!*} (Rf'^*[1] \mathcal{F}^\bullet),$$ or alternatively use Lemma \ref{smoothlem} below. From the considerations in the proof of Lemma \ref{intext} (i) these isomorphisms can be chosen so as to give an isomorphism of functors. The point is that an isomorphism as above is uniquely determined by its restriction to $f^{-1}(U)$, where it is chosen to induce the identity morphism.
\end{proof}

\begin{lem} \label{smoothlem}
Let $X = \Spec(A)$ be affine and smooth, $Y = X \times \mathbb{A}^1$, $f \colon Y \rightarrow X$ the projection, and $j \colon U \rightarrow X$ an open immersion. Let $\mathcal{F}^\bullet \in P_c^b(U, \mathbb{F}_p)$. Then there is an isomorphism 
$$Rf^*[1] ( j_{!*} \mathcal{F}^\bullet) \cong j'_{!*} ( Rf'^*[1] \mathcal{F}^\bullet).$$
\end{lem}

\begin{proof}
 
We will argue using the Riemann-Hilbert correspondence of Emerton-Kisin. The essential ideas are the same as in \cite{StabPerv}. First, we note that for a morphism $g \colon S \rightarrow T$ between smooth $k$-schemes the usual pullback functor for quasi-coherent sheaves induces a functor
$$\{ \lfgu \: \mathcal{O}_{F,T} \text{ - modules}\} \xrightarrow{g^*} \{ \lfgu \: \mathcal{O}_{F,S} \text{ - modules}\}.$$ Let $M$ be the $\lfgu$ $\mathcal{O}_{F,U}$-module corresponding to $\mathcal{F}^\bullet$. The intermediate extension $j_{!*} \mathcal{F}^\bullet$ corresponds to the smallest $\lfgu$ $\mathcal{O}_{F,X}$ submodule $N$ of $j_* M$ such that $j^*N = M$. Denote the intermediate extension by ${j_{!+}} M$. The complex $Rf^*[1] (j_{!*} \mathcal{F}^\bullet)$ corresponds to $$f^* ({j_{!+}} M) = {j_{!+}} M \otimes_A A[x],$$ where $\mathbb{A}^1 = \Spec(k[x])$. Note that $f^*({j_{!+}} M)$ is an extension of $f'^*M$. By (\ref{step1}) there is a surjection $$Rf^*[1] (j_{!*} \mathcal{F}^\bullet) \twoheadrightarrow j'_{!*}(Rf'^*[1]  \mathcal{F}^\bullet).$$ As the Riemann-Hilbert correspondence is an anti-equivalence, this corresponds to an injection
$$j_{!+}' ( f'^* M) \hookrightarrow f^* (j_{!+} M) =  {j_{!+}} M \otimes_A A[x].$$
It therefore suffices to show that if $V \subset  j_{!+} M \otimes_A A[x]$ is a sub $\lfgu$ $\mathcal{O}_{F,Y}$-module that is an extension of $f'^* M$, then $V = j_{!+} M \otimes_A A[x]$. 

To prove this, let $i \colon X \rightarrow X \times \mathbb{A}^1$ be the inclusion corresponding to the map $A[x] \rightarrow A$ which sends $x$ to $0$. Then we have a diagram with Cartesian squares:
$$\xymatrix{
\Spec(A) \ar[r]^i & \Spec(A[x]) \ar[r]^f & \Spec(A) \\
U \ar[u]^j \ar[r]^{i'} & f^{-1}(U) \ar[r]^{f'} \ar[u]_{j'} & U \ar[u]_{j}
}$$
Note that $i^*V = V \cap j_{!+} M$ is an $\lfgu$ $\mathcal{O}_{F,X}$-module which is an extension of $M$. Because $V \cap j_{!+} M \subset j_{!+} M$ and $j_{!+} M$ is the intermediate extension of $M$, it follows that $$V \cap j_{!+} M = j_{!+} M.$$ Now because $V$ is an $A[x]$-module we must have $V = j_{!+} M \otimes_A A[x]$, as desired. 
\end{proof}

We conclude this section by showing that one can glue perverse sheaves in the smooth topology.

\begin{lem} \label{dec2}
Let $\{ \phi_i \colon Y_i \rightarrow X\}_{i \in I}$ be a family of \'{e}tale morphisms from $k$-schemes $Y_i$ such that $X = \bigcup \phi_i(Y_i)$. Let $\mathcal{F}^\bullet \in D^+(X, \mathbb{F}_p)$ be such that $R \phi_i^* \mathcal{F}^\bullet \in P^+(Y_i, \mathbb{F}_p)$ for all $i$. Then $\mathcal{F}^\bullet \in P^+(X, \mathbb{F}_p)$. 
\end{lem}

\begin{proof}
This follows from the definitions.
\end{proof}

\begin{lem} \label{etdec}
Let $\{ \phi_i \colon Y_i \rightarrow X\}_{i \in I}$ be a finite family of \'{e}tale morphisms from $k$-schemes $Y_i$ such that $X = \bigcup \phi_i(Y_i)$. Then the category $P^b(X, \mathbb{F}_p)$ satisfies descent with respect to the cover $\{ \phi_i \colon Y_i \rightarrow X\}$. The same is true for $P^b_c(X, \mathbb{F}_p)$
\end{lem}

\begin{proof}
We use \cite[3.2.2, 3.2.4]{BBD} to do the gluing. Note that the condition in \cite{BBD} on vanishing $\Ext$ terms is satisfied by the definition of a $t$-structure.  Lemma \ref{dec2} guarantees that the result of gluing perverse sheaves is a perverse sheaf.
\end{proof}

\begin{lem}
Let $\{ \phi_i \colon Y_i \rightarrow X\}_{i \in I}$ be a finite family of smooth morphisms of relative dimension $d_i$ from $k$-schemes $Y_i$ such that $X = \bigcup \phi_i(Y_i)$. Let $\mathcal{F}^\bullet \in D_c^b(X, \mathbb{F}_p)$ be such that $R \phi_i^*[d_i] \mathcal{F}^\bullet \in P_c^b(Y_i, \mathbb{F}_p)$ for all $i$. Then $\mathcal{F}^\bullet \in P_c^b(X, \mathbb{F}_p)$. 
\end{lem}

\begin{proof}
As the question is local on $X$, it suffices to consider the case where $I = \{*\}$ and $\phi$ factors as $$Y \xrightarrow{g} \mathbb{A}_X^d \xrightarrow{\pi} X,$$ where $g$ is \'{e}tale, $\pi$ is the projection, and $\pi \circ g$ is surjective.

First, we claim that $R\pi^*[d] \mathcal{F}^\bullet \in P_c^b(\mathbb{A}_X^d, \mathbb{F}_p)$. By Lemma \ref{dec2}, the restriction of $R\pi^*[d] \mathcal{F}^\bullet$ to $g(Y)$ is perverse. Now the claim follows because $\mathbb{A}_X^d$ is covered by the translates of $g(Y)$ under the automorphisms of $\mathbb{A}_X^d$ which preserve the projection $\mathbb{A}_X^d \rightarrow X$. 

Thus, we have reduced to the case when $\phi$ is the projection $\pi \colon \mathbb{A}_X^d \rightarrow X$. We may also assume $X$ is affine, so that there is a closed embedding $i \colon X \rightarrow \mathbb{A}^n$ for some $n$. As $\mathcal{F}^\bullet$ is perverse if and only if $Ri_* \mathcal{F}^\bullet$  is perverse, then by the proper base change theorem we can reduce to  the case when $\phi$ is the projection $\mathbb{A}^{n+d} \rightarrow \mathbb{A}^n$ onto the first $n$ coordinates. In particular, we can assume $X$ is smooth. Now the result follows by the Riemann-Hilbert correspondence of Emerton-Kisin, as in the proof of Lemma \ref{smoothp}.
\end{proof}

Because smooth morphisms admit sections \'{e}tale locally, we can now extend our \'{e}tale descent result (Lemma \ref{etdec}) to the smooth setting. Note that we have to restrict to the constructible bounded derived category because this is the only context in which we have proved that pullback along smooth morphisms preserves perversity (Lemma \ref{smoothp}). 

\begin{cor} \label{smoothdec}
Let $\{ \phi_i \colon Y_i \rightarrow X\}_{i \in I}$ be a finite family of smooth morphisms from $k$-schemes $Y_i$ such that $X = \bigcup \phi_i(Y_i)$. Then the category $P^b_c(X, \mathbb{F}_p)$ satisfies descent with respect to the cover $\{\phi_i\}$.
\end{cor}

\begin{rmrk}
Later we will be concerned with extensions between perverse sheaves. We record here that  for $\mathcal{F}^\bullet$, $\mathcal{G}^\bullet \in P^+(X, \mathbb{F}_p)$ we have
$$\Ext^1_{P^+(X, \mathbb{F}_p)}(\mathcal{F}^\bullet, \mathcal{G}^\bullet) = \Hom_{D(X, \mathbb{F}_p)}(\mathcal{F}^\bullet, \mathcal{G}^\bullet[1]).$$ This is a general fact about $t$-structures. We do not know if a similar statement holds for higher $\Ext$ terms. 
\end{rmrk}

\section{Equivariant perverse $\mathbb{F}_p$-sheaves} \label{EquivSheaves}
In this section we define equivariant perverse sheaves and establish some of their basic properties. We fix a $k$-scheme $S$ and an $S$-group scheme $G$ of relative dimension $d$ such that the map $G \rightarrow S$ is smooth, affine, and has geometrically connected fibers. A $G$-scheme is an $S$-scheme of finite type equipped with an action of $G$. If $X$ is a $G$-scheme, we let $\rho \colon G \times_S X \rightarrow X$ be the action map, and $\pi \colon G \times_S X \rightarrow X$ be the projection. 

\begin{defn} \label{equivdef}
A $G$\emph{-equivariant perverse sheaf} on $X$ is a perverse sheaf $\mathcal{F}^\bullet \in P_c^b(X, \mathbb{F}_p)$ such that there exists an isomorphism $R \rho^* \mathcal{F}^\bullet \cong R \pi^* \mathcal{F}^\bullet$ in $D_c^b(G \times_S X, \mathbb{F}_p)$. 
\end{defn}

\begin{lem}\label{equivfun}
Let $f \colon Y \rightarrow X$ be a $G$-equivariant morphism between $G$-schemes. 
\begin{enumerate}[{\normalfont (i)}]
\item If $\mathcal{F}^\bullet \in P_c^b(X, \mathbb{F}_p)$ is $G$-equivariant and $Rf^*[n] \mathcal{F}^\bullet \in P_c^b(Y, \mathbb{F}_p)$ for some integer $n$, then $Rf^*[n] \mathcal{F}^\bullet$ is a $G$-equivariant perverse sheaf.
\item If $\mathcal{F}^\bullet \in P_c^b(Y, \mathbb{F}_p)$ is $G$-equivariant and $Rf_![n] \mathcal{F}^\bullet \in P_c^b(X, \mathbb{F}_p)$ for some integer $n$, then $Rf_![n] \mathcal{F}^\bullet$ is a $G$-equivariant perverse sheaf.
\end{enumerate}
\end{lem}

\begin{proof}
We have Cartesian diagrams
$$\xymatrix{
G \times_S Y \ar[r]^(.6){\rho_Y} \ar[d]_{\id_G \times f} & Y \ar[d]^f  & & & G \times_S Y \ar[r]^(.6){\pi_Y} \ar[d]_{\id_G \times f} & Y \ar[d]^f \\
G \times_S X \ar[r]^(.6){\rho_X} & X & & & G \times_S X \ar[r]^(.6){\pi_X} & X 
}$$
Now (i) follows from a straightforward diagram chase, and (ii) follows similarly after applying the proper base change theorem.
\end{proof}

\begin{lem} \label{equivint}
Let $j \colon U \rightarrow X$ be a $G$-equivariant immersion of $G$-schemes, and let $\mathcal{F}^\bullet \in P_c^b(U, \mathbb{F}_p)$ be $G$-equivariant. Then $j_{!*} \mathcal{F}^\bullet$ is $G$-equivariant. 
\end{lem}

\begin{proof}
The proof is the same as Lemma \ref{equivfun} (ii), except we use that taking intermediate extensions commutes with smooth pullback (Theorem \ref{smoothint}) instead of the proper base change theorem.
\end{proof}

The following Lemma \ref{fullyflem} is the key input in this section. Our proof uses Theorem \ref{smoothint} to reduce to a statement about local systems. The proof of the analogue of  Lemma \ref{fullyflem} for $\overline{\mathbb{Q}}_{\ell}$-sheaves in \cite[III.11.2]{KW01} uses the adjunction between $R\pi_!$ and $R\pi^!$, which is useful because $R\pi^*[d] = R\pi^![-d]$ when $\pi$ is smooth of relative dimension $d$. This equality fails for $\mathbb{F}_p$-sheaves, which also complicated our proof of Theorem \ref{smoothint}. 

\begin{lem} \label{fullyflem}
Let $X$ be an $S$-scheme, and let $Y$ be a smooth $S$-scheme with geometrically connected fibers of relative dimension $d$. Then if $Y$ admits an $S$-point and $\pi \colon Y \times_S X \rightarrow X$ is the projection, the functor $$R\pi^*[d] \colon P_c^b(X, \mathbb{F}_p) \rightarrow P_c^b(Y \times_S X, \mathbb{F}_p)$$ is fully faithful. This functor also preserves simple objects.
\end{lem}

\begin{proof}
We first show that $R\pi^*[d]$ is fully faithful when restricted to simple objects. Let $\mathcal{F}^\bullet \in P_c^b(X, \mathbb{F}_p)$ be simple. Then by Theorem \ref{simplechar} there is a smooth irreducible $k$-scheme $U$, an immersion $j \colon U \rightarrow X$, and a simple local system $\mathcal{L}$ of \'{e}tale $\mathbb{F}_p$-sheaves on $U$ such that $\mathcal{F}^\bullet \cong j_{!*} \mathcal{L}[\dim U]$. We have a Cartesian diagram
$$\xymatrix{
Y \times_S U \ar[r]^{j'} \ar[d]^{\pi'} & Y \times_S X \ar[d]^{\pi} \\
U \ar[r]^j & X
}$$
By Theorem \ref{smoothint} we have an isomorphism
$$R \pi^*[d] \mathcal{F}^\bullet \cong j'_{!*} (R\pi'^* \mathcal{L}[\dim U + d]).$$ We claim that $Y \times_S U$ is connected. To prove this, note that each connected component of $Y \times_S U$ surjects onto a dense open subset of $U$ since the map $\pi'$ is a smooth cover and $U$ is irreducible. Since the fibers of $\pi'$ are geometrically connected then $Y \times_S U$ is connected. 

Let $\overline{y}$ be a geometric point of $Y \times_S U$ and let $\overline{x} = \pi'(\overline{y})$. Suppose $\mathcal{L}$ corresponds to a simple finite-dimensional representation $V$ of $\pi_{1, \text{\'{e}t}}(U, \overline{x})$. Because $Y \times_S U$ is connected then the local system $\pi'^* \mathcal{L}$ corresponds to $V$ viewed as a representation of $\pi_{1, \text{\'{e}t}}(Y \times_S U, \overline{y})$ via the induced map $ \pi_{1, \text{\'{e}t}}(Y \times_S U, \overline{y}) \rightarrow \pi_{1, \text{\'{e}t}}(U, \overline{x})$. This map is surjective because $\pi'$ admits a section. By Lemma \ref{homintext},
$$\Hom(R\pi^*[d] \mathcal{F}^\bullet, R\pi^*[d] \mathcal{F}^\bullet) = \Hom(\pi'^*\mathcal{L}, \pi'^*\mathcal{L}) = \Hom_{\pi_{1, \text{\'{e}t}}(Y \times_S U, \overline{y})}(V,V).$$ Similarly, $\Hom(\mathcal{F}^\bullet, \mathcal{F}^\bullet) = \Hom_{\pi_{1, \text{\'{e}t}}( U, \overline{x})}(V,V).$  Now because $ \pi_{1, \text{\'{e}t}}(Y \times_S U, \overline{y}) \rightarrow \pi_{1, \text{\'{e}t}}(U, \overline{x})$ is surjective it follows that
$$\Hom(R\pi^*[d] \mathcal{F}^\bullet, R\pi^*[d] \mathcal{F}^\bullet) = \Hom(\mathcal{F}^\bullet, \mathcal{F}^\bullet).$$ Moreover, we see that $\pi'^* \mathcal{L}$ is also a simple local system, so $R\pi^*[d]$ preserves simple objects. Finally, if $\mathcal{F}^\bullet$ and $\mathcal{G}^\bullet$ are non-isomorphic simple objects then 
$$\Hom(R\pi^*[d] \mathcal{F}^\bullet, R\pi^*[d] \mathcal{G}^\bullet) = \Hom(\mathcal{F}^\bullet, \mathcal{G}^\bullet) = 0.$$ This follows from the fact that both pullbacks are simple, and the map $\pi$ admits a section. Thus $R\pi^*[d]$ is fully faithful when restricted to simple objects.

Now let $\mathcal{F}^\bullet \in P_c^b(X, \mathbb{F}_p)$ be simple and let $\mathcal{G}^\bullet \in P_c^b(X, \mathbb{F}_p)$ have length $n >1$. We will show by induction on $n$ that the map $$\Hom(\mathcal{F}^\bullet, \mathcal{G}^\bullet) \rightarrow \Hom(R\pi^*[d] \mathcal{F}^\bullet, R \pi^*[d] \mathcal{G}^\bullet)$$ is an isomorphism. Let
$$ 0 \rightarrow \mathcal{G}_1^\bullet \rightarrow \mathcal{G}^\bullet \rightarrow \mathcal{G}^\bullet_2 \rightarrow 0$$ be an exact sequence in $P_c^b(X, \mathbb{F}_p)$ where the outer two terms have length $< n$. Then we get a morphism between exact triangles in the derived category of abelian groups: 
$$\resizebox{15cm}{!}{\xymatrix{
\RHom(\mathcal{F}^\bullet, \mathcal{G}_1^\bullet) \ar[r] \ar[d] & \RHom(\mathcal{F}^\bullet, \mathcal{G}^\bullet) \ar[r] \ar[d] & \RHom(\mathcal{F}^\bullet, \mathcal{G}_2^\bullet) \ar[d] \\
\RHom(R \pi^*[d] \mathcal{F}^\bullet, R \pi^*[d]\mathcal{G}_1^\bullet) \ar[r]  & \RHom(R \pi^*[d] \mathcal{F}^\bullet, R \pi^*[d] \mathcal{G}^\bullet) \ar[r] \ & \RHom(R \pi^*[d] \mathcal{F}^\bullet, R \pi^*[d] \mathcal{G}_2^\bullet) }
}$$

These complexes are concentrated in non-negative degrees because we are dealing with perverse sheaves. Moreover, the induced map $$\Ext^1(\mathcal{F}^\bullet, \mathcal{G}^\bullet_1) \rightarrow \Ext^1(R\pi^*[d] \mathcal{F}^\bullet, R \pi^*[d] \mathcal{G}^\bullet_1)$$ is injective because $\pi$ admits a section. By induction and the five lemma the map  $$\Hom(\mathcal{F}^\bullet, \mathcal{G}^\bullet) \xrightarrow{\sim} \Hom(R\pi^*[d] \mathcal{F}^\bullet, R \pi^*[d] \mathcal{G}^\bullet)$$ is an isomorphism. We can complete the proof by a similar induction on the length of $\mathcal{F}^\bullet$. 
\end{proof}

\begin{lem} \label{stabsub}
The essential image of the functor in Lemma \ref{fullyflem} is an abelian subcategory of $P_c^b(Y \times_S X, \mathbb{F}_p)$ which is stable under taking subquotients of objects. 
\end{lem}

\begin{proof}
Because $R\pi^*[d]$ is exact and preserves simple objects, then it preserves the length of objects. In fact, applying $R\pi^*[d]$ to a composition series for $\mathcal{F}^\bullet$ gives a composition series for $R\pi^*[d] \mathcal{F}^\bullet$. By the Jordan-H\"{o}lder theorem, this implies that every simple subquotient of $R\pi^*[d] \mathcal{F}^\bullet$ is in the essential image of $R\pi^*[d]$.

For the general case, suppose $\mathcal{F}^\bullet \in P_c^b(X, \mathbb{F}_p)$ has length $n > 1$. Let $\mathcal{G}^\bullet \subseteq R\pi^*[d] \mathcal{F}^\bullet$ be a subobject of length $m$. We will prove by induction on the integer $c = n-m$ that $\mathcal{G}^\bullet$ is in the essential image of $R\pi^*[d]$. The case $c=0$ is clear. If $c>0$, let $\mathcal{F}'^\bullet$ be a subobject of $R\pi^*[d] \mathcal{F}^\bullet$ of length $m+1$ containing $\mathcal{G}^\bullet$. By the induction hypothesis, $\mathcal{F}'^\bullet$ is in the essential image of $R\pi^*[d]$. We have an exact sequence 
$$0 \rightarrow \mathcal{G}^\bullet \rightarrow \mathcal{F}'^\bullet \rightarrow \mathcal{H}^\bullet \rightarrow 0$$ where $\mathcal{H}^\bullet$ is simple. We have already shown that $\mathcal{H}^\bullet$ is in the essential image of $R\pi^*[d]$. Because $R\pi^*[d]$ is exact and fully faithful, it follows that $\mathcal{G}^\bullet$ is in the essential image of $R\pi^*[d]$. This completes the induction. Now that we know the essential image of $R\pi^*[d]$ is stable under taking subobjects, the fact that it is stable under taking subquotients also follows from the fact that $R\pi^*[d]$ is exact and fully faithful. 
\end{proof}

\begin{lem} \label{4.2.3}
Let $G$ act on $G \times_S X$ by left translation on the first coordinate, and let $\pi \colon G \times_S X \rightarrow X$ be the projection. If $G$ has relative dimension $d$ over $S$ then the functor $R\pi^*[d]$ induces an equivalence of categories between $P_c^b(X, \mathbb{F}_p)$ and the full subcategory of $P_c^b(G \times_S X, \mathbb{F}_p)$ consisting of $G$-equivariant perverse sheaves. If $i \colon X \rightarrow G \times_S X$ is the inclusion $x \mapsto (1,x)$, then $Ri^*[- d]$ is an inverse to $R\pi^*[d]$. 
\end{lem}

\begin{proof}
The same proof as in \cite[4.2.3]{Letellier} works, except we replace the reference to \cite[4.2.5]{BBD} by Lemma \ref{fullyflem}. 
\end{proof}

\begin{lem}\label{cocycle}
Let $\mathcal{F}^\bullet$ be a $G$-equivariant perverse sheaf on $X$. Denote by $m \colon G \times_S G \rightarrow G$ the multiplication map, $p_2 \colon G \times_S G \rightarrow G$ the projection onto the second factor, and $i \colon X \rightarrow G \times_S X$ the inclusion $x \mapsto (1,x)$. Then there is a unique isomorphism $\phi \colon R\pi^*\mathcal{F}^\bullet \xrightarrow{\sim} R\rho^*\mathcal{F}^\bullet$ such that
$Ri^*(\phi) \colon \mathcal{F}^\bullet \xrightarrow{\sim} \mathcal{F}^\bullet$ is the identity. This isomorphism also satisfies
$$R(m \times \id_X)^* \phi = R(\id_G \times \rho)^*(\phi) \circ R(p_2 \times \id_X)^*(\phi).$$
\end{lem}

\begin{proof}
The same proof as in \cite[4.2.4]{Letellier} works.
\end{proof}

\begin{defn} \label{equivcatdef}
We define the category $P_G(X, \mathbb{F}_p)$ of $G$-equivariant perverse sheaves as follows. Its objects are $G$-equivariant perverse sheaves in $P_c^b(X, \mathbb{F}_p)$. To define morphisms, let $\mathcal{F}^\bullet$ and $\mathcal{G}^\bullet$ be $G$-equivariant perverse sheaves on $X$. Let $\phi \colon R\pi^*\mathcal{F}^\bullet \rightarrow R\rho^*\mathcal{F}^\bullet$ and $\psi \colon R\pi^*\mathcal{G}^\bullet \rightarrow R\rho^*\mathcal{G}^\bullet$ be the two unique isomorphisms satisfying the cocycle conditions as in Lemma \ref{cocycle}. Then morphisms $\mathcal{F}^\bullet \rightarrow \mathcal{G}^\bullet$ in $P_G(X, \mathbb{F}_p)$ are morphisms $\tau \colon \mathcal{F}^\bullet \rightarrow \mathcal{G}^\bullet$ between perverse sheaves such that the following diagram commutes:
$$\xymatrix{ R\pi^* \mathcal{F}^\bullet \ar[d]^{\phi} \ar[r]^{R\pi^*(\tau)} & R \pi^* \mathcal{G}^\bullet \ar[d]^{\psi} \\ 
R\rho^* \mathcal{F}^\bullet \ar[r]^{R\rho^*(\tau)} & R\rho^* \mathcal{G}^\bullet
}$$
\end{defn}

\begin{prop} \label{fullsub}
The category $P_G(X, \mathbb{F}_p)$ is a full subcategory of $P_c^b(X, \mathbb{F}_p)$. 
\end{prop}

\begin{proof}
We give the same proof as in \cite[4.2.7]{Letellier}. In the notation of Definition \ref{equivcatdef}, we need to show that if $\tau \colon \mathcal{F}^\bullet \rightarrow \mathcal{G}^\bullet$ is a morphism between perverse sheaves then $$R\rho^*(\tau) \circ \phi = \psi \circ R \pi^*(\tau).$$ Equality certainly holds after applying $Ri^*$. Thus, by appealing to Lemma \ref{4.2.3} it suffices to show that each of $R\pi^*[d] \mathcal{F}^\bullet$, $R \pi^*[d] \mathcal{G}^\bullet$, $R\rho^*[d] \mathcal{F}^\bullet$, $R\rho^*[d] \mathcal{G}^\bullet$ is $G$-equivariant on $G \times X$, where $G$ only acts by left translation on the first factor. The map $\rho$ is $G$-equivariant, and the map $\pi$ is $G$-equivariant when we give $X$ the trivial action of $G$. Thus the equivariance of these perverse sheaves follows from Lemma $\ref{equivfun}$.
\end{proof}

\begin{prop} \label{subquotprop}
If $\mathcal{F}^\bullet \in P_G(X, \mathbb{F}_p)$ then any subquotient of $\mathcal{F}^\bullet$ is $G$-equivariant.
\end{prop}

\begin{proof}
We give the same proof as in \cite[4.2.13]{Letellier}. If $\mathcal{G}^\bullet$ is a subquotient of $\mathcal{F}^\bullet$, then $R \rho^*[d] \mathcal{G}^\bullet$ is a subquotient of $R \rho^*[d] \mathcal{F}^\bullet$. Because $\mathcal{F}^\bullet$ is $G$-equivariant, $R \rho^*[d] \mathcal{G}^\bullet$ is a subquotient of $R \pi^*[d] \mathcal{F}^\bullet$. Hence by Lemma \ref{stabsub} there is some $\mathcal{H}^\bullet \in P_c^b(X, \mathbb{F}_p)$ such that $$R \pi^*[d] \mathcal{H}^\bullet \cong R \rho^*[d] \mathcal{G}^\bullet.$$ Apply the functor $Ri^*[-d]$ to this isomorphism, where $i \colon X \rightarrow G \times_S X$ is the inclusion $x \mapsto (1,x)$. This gives $\mathcal{H}^\bullet \cong \mathcal{G}^\bullet$. Now by the above isomorphism it follows that $\mathcal{G}^\bullet$ is $G$-equivariant. 
\end{proof}

Let $\mathcal{F}^\bullet$, $\mathcal{G}^\bullet \in P_G(X, \mathbb{F}_p)$. Then an extension of  $\mathcal{F}^\bullet$ by $\mathcal{G}^\bullet$ is not necessarily $G$-equivariant. Now let $R$ be an affine $S$-scheme, and let $f_R \colon X_R \rightarrow X$ be the base change of $R \rightarrow S$ by $X \rightarrow S$. Then the group $G(R)$ naturally acts on $\Ext^1_{D^b_c(X_R, \mathbb{F}_p)}(Rf_R^*\mathcal{F}^\bullet, Rf_R^*\mathcal{G}^\bullet)$. The next proposition gives a criterion for determining when an extension is equivariant.

\begin{prop} \label{equivextprop}
Let $\mathcal{F}^\bullet$, $\mathcal{G}^\bullet \in P_G(X, \mathbb{F}_p)$. Then an element $s \in \Ext^1_{D^b_c(X, \mathbb{F}_p)}(\mathcal{F}^\bullet, \mathcal{G}^\bullet)$ represents a $G$-equivariant extension if and only if the element
$$f_R^* (s) \in \Ext^1_{D^b_c(X_R, \mathbb{F}_p)}(Rf_R^*\mathcal{F}^\bullet, Rf_R^*\mathcal{G}^\bullet)$$ is fixed by $G(R)$ for every affine $S$-scheme $R$. 
\end{prop}

\begin{proof}
Suppose $s$ represents the extension $$0 \rightarrow \mathcal{G}^\bullet \rightarrow \mathcal{H}^\bullet \rightarrow \mathcal{F}^\bullet \rightarrow 0.$$ If $\mathcal{H}^\bullet$ is $G$-equivariant then it follows from Proposition \ref{fullsub} that $f_R^* (s)$ is fixed by $G(R)$ for every $R$. To prove the converse, consider the case $R=G$ and the element $\phi \in G(R)$ corresponding to the identity map of $G$. The fact that $\phi$ fixes $s$ implies that we have an isomorphism $R\pi^* \mathcal{H}^\bullet \cong R \rho^* \mathcal{H}^\bullet$. Thus $\mathcal{H}^\bullet$ is $G$-equivariant. 
\end{proof}

\begin{lem} \label{factoreq}
Suppose $X$ has the structure of a $G$-scheme and an $H$-scheme, and suppose the action of $G$ on $X$ factors through a morphism $f \colon G \rightarrow H$. Then there is a natural functor $P_H(X, \mathbb{F}_p) \rightarrow P_G(X, \mathbb{F}_p)$. 
\end{lem} 

\begin{proof}
In view of Proposition \ref{fullsub}, we only need to show that if $\mathcal{F}^\bullet \in P_c^b(X, \mathbb{F}_p)$ is $H$-equivariant then it is also $G$-equivariant. We have commutative diagrams
$$\xymatrix{
G \times_S X \ar[rd]^{\rho_G} \ar[dd]_{f \times \id_X} &  & & G \times_S X \ar[rd]^{\pi_G} \ar[dd]_{f \times \id_X} &  & &\\
 & X & & &  X\\
 H \times_S X \ar[ru]^{\rho_H} & & &  H \times_S X \ar[ru]^{\pi_H}
}$$
Thus from an isomorphism $R \rho_H^* \mathcal{F}^\bullet \cong R \pi_H^* \mathcal{F}^\bullet$ we get an isomorphism $R \rho_G^* \mathcal{F}^\bullet \cong R \pi_G^* \mathcal{F}^\bullet$. 
\end{proof}

\begin{rmrk} \label{indequiv} Consider the following setup (see also \cite{RicharzGS}). Let $G_i$ be an inverse system of smooth affine $S$-schemes with geometrically connected fibers, and let $G = \varprojlim G_i$. Let $X = \varinjlim {X_i}$ be an ind-scheme, where each $X_i$ is an $S$-scheme of finite type and the transition maps are closed immersions of $S$-schemes. Suppose that $G$ acts on each $X_i$ and that furthermore the action of $G$ on each $X_i$ factors through the quotient $G_i$. We also assume that for $j \geq i$ the immersion $X_i \rightarrow X_j$ is $G$-equivariant. Define
$$P_c^b(X, \mathbb{F}_p) := \varinjlim P_c^b(X_i, \mathbb{F}_p).$$ Note that if $j \geq i$ then by Lemma \ref{equivfun} and Lemma \ref{factoreq} there is a natural fully faithful functor
$$P_{G_i}(X_i, \mathbb{F}_p) \rightarrow P_{G_j}(X_j, \mathbb{F}_p).$$ Thus we can define
$$P_{G}(X, \mathbb{F}_p) := \varinjlim P_{G_i}(X_i, \mathbb{F}_p).$$
\end{rmrk}

\begin{rmrk}
In our definition of $P_{G}(X, \mathbb{F}_p)$ in Remark \ref{indequiv} we note that it is possible that there exists an object in $P_{G}(X, \mathbb{F}_p)$ supported on $X_i$ but that is not in the essential image of  $P_{G_i}(X_i, \mathbb{F}_p) \rightarrow P_G(X, \mathbb{F}_p)$. This is because there could be a perverse sheaf $\mathcal{F}^\bullet \in P_c^b(X_i, \mathbb{F}_p)$ which is $G_j$-equivariant but not $G_i$-equivariant for some $j > i$. Thus, the subcategory of $P_{G}(X, \mathbb{F}_p)$ consisting of perverse sheaves supported on $X_i$ should be thought of as those perverse sheaves in $P_c^b(X_i, \mathbb{F}_p)$ which are $G_j$-equivariant for large enough $j \geq i$. 
\end{rmrk}

\begin{prop} \label{welldef}
Suppose in the notation of Lemma \ref{factoreq} that $f$ is surjective and the kernel $N := \ker f$ is a smooth, affine $S$-scheme with geometrically connected fibers. Then the morphism $P_H(X, \mathbb{F}_p) \rightarrow P_G(X, \mathbb{F}_p)$ is an equivalence.
\end{prop}

\begin{proof}
Let $\mathcal{F}^\bullet \in P_G(X, \mathbb{F}_p)$, and $$\alpha : = f \times_S \id_X \colon G \times_S X \rightarrow H \times_S X.$$ Note that $f$ is smooth because $N$ is smooth and $f$ is an $N$-torsor. Then $\alpha$ is a smooth cover of $H \times_S X$. As in the proof of Lemma \ref{factoreq} we have $R \alpha^* R \rho_H^* = R \rho_G^*$ and $R \alpha^* R \pi_H^* = R \pi_G^*$. We want to descend the isomorphism $\beta \colon R \rho_G^* \mathcal{F}^\bullet \cong R \pi_G^* \mathcal{F}^\bullet$ along $\alpha$ to an isomorphism $R \rho_H^* \mathcal{F}^\bullet \cong R \pi_H^* \mathcal{F}^\bullet$. Consider the Cartesian diagram
$$\xymatrix{
G \times_S X \ar[d]^{\alpha} & N \times_S G \times_S X \ar[l]_{p} \ar[d]^{m} \\
H \times_S X & G \times_S X \ar[l]^\alpha
}$$
Here $p$ is the projection onto $G \times_S X$ and $m(n,g,x) = (ng,x)$. 
By smooth descent (Corollary \ref{smoothdec}) we only need to check that the following diagram commutes
$$\xymatrix{
Rp^* R\alpha^* R \rho_H^* \mathcal{F}^\bullet \ar[d]^{\gamma} \ar[r]^{R p^* \beta} & Rp^* R \alpha^* R \pi_H ^* \mathcal{F}^\bullet  \ar[d]^{\delta} \\
Rm^* R\alpha^* R \rho_H^* \mathcal{F}^\bullet \ar[r]^{Rm^* \beta} & Rm^* R \alpha^* R \pi_H ^* \mathcal{F}^\bullet
}$$
Here the vertical maps are the natural isomorphisms coming from commutativity of the previous diagram. Let $N$ act on $G \times_S X$ by $n \cdot (g,x) = (ng,nx) = (ng,x)$.  The projection $\pi_G \colon G \times_S X \rightarrow X$ is $N$-equivariant, so $R \alpha^* R\pi_H^* \mathcal{F}^\bullet = R\pi_G^* \mathcal{F}^\bullet$ is an $N$-equivariant perverse sheaf on $G \times_S X$ (up to a shift). By Lemma \ref{cocycle}, the identity $$\delta = (Rm^* \beta) \circ \gamma \circ (Rp^* \beta)^{-1}$$ holds because after applying $Ri^*$, both morphisms reduce to the identity map of $R \pi_G^* \mathcal{F}^\bullet$. 
\end{proof}

For the rest of this section we assume $S=k$.

\begin{lem} \label{simpeq}
Suppose $S=k$ and that $\mathcal{F}^\bullet \in P_c^b(X, \mathbb{F}_p)$ is simple and $G$-equivariant. Then there exists a $G$-equivariant immersion $j \colon U \rightarrow X$ from a smooth irreducible $G$-scheme $U$ and a simple $G$-equivariant local system $\mathcal{L}$ on $U$ such that $\mathcal{F}^\bullet \cong j_{!*}(\mathcal{L}[\dim U])$. 
\end{lem}

\begin{proof}
As our classification of the simple objects in $P_c^b(X, \mathbb{F}_p)$ in Theorem \ref{simplechar} is analogous to the case of $\overline{\mathbb{Q}}_\ell$-coefficients, the same proof as in \cite[12.20]{nilpotentorbits} works, except we replace the reference to 12.14(5) by Lemma \ref{simplelem2}.
\end{proof}

\begin{lem}
Suppose $S=k$ and that $G$ acts on $X$ with finitely many orbits. Then if $\mathcal{F}^\bullet \in P_c^b(X, \mathbb{F}_p)$ is simple and $G$-equivariant, there is a unique orbit $j \colon \mathcal{O} \rightarrow X$ and a unique simple $G$-equivariant local system $\mathcal{L}$ on $\mathcal{O}$ such that $\mathcal{F}^\bullet \cong j_{!*}(\mathcal{L}[\dim \mathcal{O}])$. Conversely, every such pair $(\mathcal{O}, \mathcal{L})$ determines a unique simple $G$-equivariant perverse sheaf.
\end{lem}

\begin{proof}
Let $\mathcal{F}^\bullet \in P_c^b(X, \mathbb{F}_p)$ be simple and $G$-equivariant. Then in the notation of Lemma \ref{simpeq}, since $G$ acts with finitely many orbits there is an open dense orbit $\mathcal{O}$ in $U$. By Lemma \ref{simplelem2} we may replace $U$ by $\mathcal{O}$ and $\mathcal{L}$ by its restriction to $\mathcal{O}$. For the uniqueness of $\mathcal{O}$, note that $\mathcal{O}$ is uniquely determined as the orbit which is open in the support of $\mathcal{F}^\bullet$. Then $\mathcal{L}$ is determined by the restriction of $\mathcal{F}^\bullet$ to $\mathcal{O}$. 
\end{proof}

\begin{lem}
Suppose $S=k$ and that $G$ acts transitively on $X$ with connected stabilizers. Then the only $G$-equivariant local systems $\mathcal{L}$ on $X$ are constant.
\end{lem}

\begin{proof}
We follow the proof in \cite[12.7]{nilpotentorbits}. Let $x \in X(k)$ and let $H \subset G$ be the closed reduced subgroup such that $H(k)$ is the stabilizer of $x$. Then the map $G \rightarrow X$, $g \mapsto gx$ factors as a composition
$$G \xrightarrow{p_1} G/H \xrightarrow{p_2} X$$ where $p_1$ is an $H$-torsor and $p_2$ is a purely inseparable bijective morphism. Each of these maps is $G$-equivariant, and $p_2^* \mathcal{L}$ and $p_1^* p_2^* \mathcal{L}$ are $G$-equivariant. By \cite[IX 4.10]{SGA1} the map $p_2$ induces an isomorphism of fundamental groups, so $\mathcal{L}$ is constant if and only if $p_2^* \mathcal{L}$ is constant. Because $G/H$ is smooth then $p_2^* \mathcal{L}$ and constant sheaves on $G/H$ are perverse (when shifted by $\dim G/H$).  We claim that $p_1^*p_2^* \mathcal{L}$ is constant. Assuming this claim for a moment, let $\mathcal{F}$ be the constant sheaf on $G$ associated to the stalk of $p_1^*p_2^* \mathcal{L}$ at the identity element. Then we have a canonical isomorphism $\mathcal{F} \cong p_1^*p_2^* \mathcal{L}$. Since $H$ is connected one can descend this isomorphism to $G/H$ using the same ideas as in the proof of Proposition \ref{welldef},  so $p_2^*\mathcal{L}$ is constant. Thus $\mathcal{L}$ is also constant. 

To prove the claim that $p_1^* p_2^* \mathcal{L}$ is constant, we will in fact show that the only $G$-equivariant local systems $\mathcal{L}$ on $G$ for the usual left action of $G$ are constant sheaves. Consider the map $j \colon G \rightarrow G \times G$ given by $g \mapsto (g^{-1},g)$. Then $\rho \circ j$ is the constant map $G \rightarrow G$, $g \mapsto 1$, so $j^* \rho^* \mathcal{L}$ is a constant sheaf. The map $\pi \circ j$ is the identity map $\id_G$. Thus, applying $j^*$ to an isomorphism $\rho^* \mathcal{L} \cong \pi^* \mathcal{L}$ gives that $\mathcal{L}$ is constant.
\end{proof}

\begin{cor} \label{simpobjs}
Suppose $S=k$ and that $G$ acts on $X$ with finitely many orbits, each having connected stabilizers. Then the simple $G$-equivariant perverse sheaves in $P_c^b(X, \mathbb{F}_p)$ are the intermediate extensions of the constant sheaves $\mathbb{F}_p[\dim \mathcal{O}]$ along the orbits $\mathcal{O}$. 
\end{cor}
\section{Perverse $\mathbb{F}_p$-sheaves on Cohen-Macaulay and $F$-rational schemes} \label{PerverseFrat}

In this section we prove Theorems \ref{mainthm} and \ref{mainthm2}. To begin, we will take the following theorem of Smith as our definition of $F$-rationality.

\begin{thm} \cite[2.6]{SmithFrat} \label{FRatDef}
Let $(A, \mathfrak{m})$ be an excellent Cohen-Macaulay local ring of dimension $d$ and characteristic $p > 0$. Then $A$ is $F$-rational if and only if the local cohomology module $H_{\mathfrak{m}}^d(A)$ is a simple left $A[F]$-module.
\end{thm}
The original definition of $F$-rationality is related to tight closure theory (see \cite{Frationalloc}). A $k$-scheme is $F$-rational if all of its local rings are $F$-rational. The idea of the proofs will be to show all of the hypotheses are preserved under passage to strict henselizations. We will then use the Artin-Schreier sequence to show that the necessary cohomology sheaves vanish.

\begin{lem} \label{shprop}
Let $(A,\mathfrak{m})$ be an excellent local ring of characteristic $p>0$, and let $A^{\sh}$  be a strict henselization of $A$. Then $A^{\sh}$ is an excellent  local ring of characteristic $p>0$. Moreover, if $A$ has one of the following four properties then $A^{\sh}$ has the same property:
\begin{enumerate}[{\normalfont (i)}]
\item $A$ is Cohen-Macaulay,
\item $A$ is normal,
\item $\dim(A) = d$,
\item $A$ is $F$-rational.
\end{enumerate}
\end{lem}

\begin{proof}
The ring $A^{\sh}$ is noetherian by \cite[18.8.8]{EGA4IV}. The excellence of $A^{\sh}$ is not shown in \emph{loc. cit.}, but this can be found, for example, in \cite[Ch. 1,  \S 1]{KiehlEtale}. The Cohen-Macaulay property and normality are preserved by \cite[18.8.13]{EGA4IV}. The dimension is preserved by \cite[06LK]{stacks-project}.

To see that $A^{\sh}$ is $F$-rational if $A$ is $F$-rational, we will use the construction of $A^{\sh}$ as in \cite[04GP]{stacks-project}. Fix a separable closure $k(\mathfrak{m})^{\sep}$ of the residue field $k(\mathfrak{m})$ and consider the collection of triples $(S,\mathfrak{q}, \alpha)$ where $A \rightarrow S$ is \'{e}tale, $\mathfrak{q}$ is the only prime of $S$ lying over $\mathfrak{m}$, and $\alpha \colon   k(\mathfrak{q}) \rightarrow k(\mathfrak{m})^{\sep}$ is an embedding of $k(\mathfrak{m})$-algebras. Note that $\mathfrak{q} = \mathfrak{m}S$ (see \cite[04GN]{stacks-project}). 

Then $A^{\sh}$ is identified with the filtered colimit $\colim_{(S,\mathfrak{q}, \alpha)} S$.  Note that $A \rightarrow A^{\sh}$ is flat and $\mathfrak{m} A^{\sh}$ is the maximal ideal of $A^{\sh}$. Then $$H_{\mathfrak{m}A^{\sh}}^d(A^{\sh}) = H_{\mathfrak{m}}^d(A) \otimes_A A^{\sh}.$$ As tensor products commute with colimits, $$H_{\mathfrak{m}}^d(A) \otimes_A A^{\sh} = \colim_{(S, \mathfrak{q}, \alpha)} H_{\mathfrak{m}}^d(A) \otimes_A S = \colim_{(S, \mathfrak{q}, \alpha)} H_{\mathfrak{q}  }^d(S).$$

For each $(S, \mathfrak{q}, \alpha)$ the map $S \rightarrow A^{\sh}$ factors through the localization map $S \rightarrow S_{\mathfrak{q}}$. Hence the map $H_{\mathfrak{q}}^d(S) \rightarrow H_{\mathfrak{m}A^{\sh}}^d(A^{\sh})$ factors through the map $H_{\mathfrak{q}}^d(S) \rightarrow H_{\mathfrak{q}S_{\mathfrak{q}}}^d(S_{\mathfrak{q}})$. By \cite[Thm. 3.1]{Frationalbasechange}, the ring $S$ is $F$-rational. Now we would like to apply \cite[Thm. 4.2]{Frationalloc} to conclude $S_{\mathfrak{q}}$ is also $F$-rational, but to do this we need to know that $S$ is a homomorphic image of a Cohen-Macaulay ring. Since $A$ is local and $F$-rational then it is already Cohen-Macaulay by \cite[Thm. 6.27]{Frationalloc}. Thus $S$ is also Cohen-Macaulay, so $S_{\mathfrak{q}}$ is $F$-rational and $H_{\mathfrak{q}S_{\mathfrak{q}}}^d(S_{\mathfrak{q}})$ is a simple $S_{\mathfrak{q}}[F]$-module. 

At this point we have shown $H_{\mathfrak{m}A^{\sh}}^d(A^{\sh}) = \colim_{(S, \mathfrak{q}, \alpha)} H_{\mathfrak{q} }^d(S)$ where the image of each $H_{\mathfrak{q}  }^d(S)$ generates a simple $S_{\mathfrak{q}}[F]$-submodule of $H_{\mathfrak{m}A^{\sh}}^d(A^{\sh})$. It follows that $H_{\mathfrak{m}A^{\sh}}^d(A^{\sh})$ is a simple $A^{\sh}[F]$-module, so $A^{\sh}$ is $F$-rational. 
\end{proof}

The following proposition is the main computation from which we will derive the necessary cohomological vanishing results. 

\begin{prop} \label{locvan} \leavevmode \begin{enumerate}[{\normalfont (i)}]
\item Let $(A, \mathfrak{m})$ be a Cohen-Macaulay, noetherian local ring of characteristic $p > 0$ and dimension $d$. Let $i \colon k(\mathfrak{m}) \rightarrow \Spec(A)$ be the inclusion of the closed point and let $j \colon U= \Spec(A) - k(\mathfrak{m}) \rightarrow \Spec(A)$ be the inclusion of the complement. Then

$$R^m\Gamma (Ri^! \mathbb{F}_p[0]) = 0, \quad \quad \quad \quad m < d.$$

\item If $\Spec(A)$ is furthermore irreducible and $d> 0$ then 
$$R^d \Gamma (Ri^! \mathbb{F}_p[0]) = \left\{
        \begin{array}{ll}
            \ker(H_{\textnormal{\'{e}t}}^1(\Spec(A), \mathbb{F}_p) \rightarrow H_{\textnormal{\'{e}t}}^1(U, \mathbb{F}_p))& \quad d =1 \\
            \ker(H_{\mathfrak{m}}^d(A) \xrightarrow{F-\id} H_{\mathfrak{m}}^d(A)) & \quad d > 1. 
        \end{array}
    \right.
$$
\end{enumerate}
\end{prop}

\begin{proof}
If $d=0$ the result is clear, so we now assume $d \geq 1$. Note that $$R^m \Gamma (Ri^! \mathbb{F}_p[0]) = R^m \Gamma (Ri_* Ri^! \mathbb{F}_p[0]).$$ We can compute the latter group by applying $R\Gamma$ to the exact triangle
\begin{equation} \label{triangle}
Ri_*\:  Ri^! \mathbb{F}_p[0] \rightarrow \mathbb{F}_p[0] \rightarrow Rj_* \mathbb{F}_p[0]
\end{equation} in $D^b_c(\Spec(A), \mathbb{F}_p)$. From the Artin–Schreier sequence $$0 \rightarrow \mathbb{F}_p \rightarrow \mathcal{O}_{\Spec(A)} \xrightarrow{F-\id} \mathcal{O}_{\Spec(A)} \rightarrow 0$$ and the fact that $\Spec(A)$ is affine, we deduce that
$$
R^m \Gamma ( \mathbb{F}_p[0]) = \left\{
        \begin{array}{ll}
            \ker(A \xrightarrow{F-\id} A) = \mathbb{F}_p & \quad m =0 \\
            \cok(A \xrightarrow{F-\id} A) & \quad m =1 \\
            0  & \quad m > 1.
        \end{array}
    \right.
$$

There is an exact sequence
$$0 \rightarrow H_{\mathfrak{m}}^0(A)  \rightarrow A \rightarrow H^0(U, \mathcal{O}_U) \rightarrow H_{\mathfrak{m}}^1(A) \rightarrow 0.$$ Moreover, for $m \geq 1$ there are isomorphisms $$H^m(U, \mathcal{O}_U) \cong H_{\mathfrak{m}}^{m+1}(A).$$

First suppose $d=1$. The map $R^{0} \Gamma(\mathbb{F}_p[0]) \rightarrow R^{0} \Gamma(Rj_*\mathbb{F}_p[0])$ is the natural map $$H^0_{\text{\'{e}}t}(\Spec(A), \mathbb{F}_p) \rightarrow H^0_{\text{\'{e}}t}(U, \mathbb{F}_p).$$ It is possible that $U$ is disconnected, but in any case this map is always injective. By appealing to (\ref{triangle}), this observation proves (i) in the case $d=1$. For (ii) we note that if $\Spec(A)$ is irreducible then $R^{0} \Gamma(\mathbb{F}_p[0]) \rightarrow R^{0} \Gamma(Rj_*\mathbb{F}_p[0])$ is an isomorphism.

Now assume $d \geq 2$. We will denote the derived global sections functor on $U$ by $R\Gamma_U$. We have $R \Gamma_U = R \Gamma \circ Rj_*$, and $$R^m \Gamma (Rj_* \mathbb{F}_p[0]) = R^m \Gamma_U(\mathbb{F}_p[0]).$$ Because $A$ is Cohen-Macaulay, $$H_{\mathfrak{m}}^m(A) = 0, \quad \quad \quad \quad m < d.$$ As $d \geq 2$ it follows that $H^0(U, \mathcal{O}_U) = A$ and $H^m(U, \mathcal{O}_U) = 0$ for $0 < m < d-1$. Thus $H^0_{\text{\'{e}}t}(\Spec(A), \mathbb{F}_p) \rightarrow H^0_{\text{\'{e}}t}(U, \mathbb{F}_p)$ is an isomorphism.  Additionally, 
\begin{equation} \label{myeq1} R^{m} \Gamma_U (\mathbb{F}_p[0])  = 0, \quad  \quad \quad  \quad  1 < m < d-1,
\end{equation}
and if $d > 2$,
\begin{equation} \label{myeq2} 
R^{d-1} \Gamma_U (\mathbb{F}_p[0])  = \ker(H^{d-1}(U, \mathcal{O}_U) \xrightarrow{F - \id} H^{d-1}(U, \mathcal{O}_U)).
\end{equation}

Suppose $d=2$. From the Artin–Schreier sequences on $U$ and $\Spec(A)$ we get a diagram of long exact sequences:
$$\xymatrix{
0 \ar[r]&  H^0_{\text{\'{e}t}}(U, \mathbb{F}_p) \ar[r]& A \ar[r]^{F-\id} & A \ar[r] & H^1_{\text{\'{e}t}}(U, \mathbb{F}_p) \ar[r] & H^1(U, \mathcal{O}_U) \ar[r]^(.7){F-\id} &\\ 
0 \ar[r]&  H^0_{\text{\'{e}t}}(\Spec(A), \mathbb{F}_p) \ar[r] \ar[u]^{f} & A \ar[u]^{\sim} \ar[r]^{F-\id} & A \ar[r] \ar[u]^{\sim} & H^1_{\text{\'{e}t}}(\Spec(A), \mathbb{F}_p) \ar[r] \ar[u]^g & 0 \ar[u] \ar[r] &
}
$$
From this we deduce that the map $R^m \Gamma ( \mathbb{F}_p[0]) \rightarrow R^m \Gamma_U (\mathbb{F}_p[0])$ is an isomorphism for $m=0$ and an injection for $m=1$ (these correspond to the maps $f$ and $g$). This implies $R^m\Gamma (Ri^! \mathbb{F}_p[0]) = 0$ for $m < 2$. Now for (ii) we have an exact sequence
$$ 0 \rightarrow  H^1_{\text{\'{e}t}}(\Spec(A), \mathbb{F}_p) \rightarrow H^1_{\text{\'{e}t}}(U, \mathbb{F}_p) \rightarrow  R^2\Gamma(Ri^! \mathbb{F}_p[0]) \rightarrow 0.$$ Hence by considering the previous diagram and using that $H^1(U,\mathcal{O}_U) \cong H_{\mathfrak{m}}^2(A)$ we get $R^2\Gamma(Ri^! \mathbb{F}_p[0]) = \ker(H_{\mathfrak{m}}^2(A) \xrightarrow{F-\id} H_{\mathfrak{m}}^2(A))$. 

Finally, consider the case $d > 2$. As $H^1(U, \mathcal{O}_U) = 0$ the maps $R^m \Gamma ( \mathbb{F}_p[0]) \rightarrow R^m \Gamma_U (\mathbb{F}_p[0])$ are isomorphisms for $m = 0$ and $m=1$. Thus $R^m\Gamma (Ri^! \mathbb{F}_p[0]) = 0$ for $m \leq 2$, and $R^m\Gamma (Ri^! \mathbb{F}_p[0]) = R^{m-1}\Gamma_U(\mathbb{F}_p[0])$ for $m > 2$. Now (i) follows from (\ref{myeq1}) and (ii) follows from (\ref{myeq2}). 
\end{proof}

\begin{myproof2}[Proof of Theorem \ref{mainthm}]
We will check the criteria for perversity as in Remark \ref{defrem}. Fix a point $x \in X$ of codimension $c$, so $\dim \{\overline{x}\} = d-c$. Let $i \colon \overline{x} \rightarrow \Spec(\mathcal{O}_{X,x}^{\text{sh}})$ be the inclusion of the closed point in a strict henselization of the local ring of $X$ at $x$. Clearly $H^i(Ri^* \mathbb{F}_p[d]) = 0$ for $i > -d+c$. By Lemma \ref{shprop}, $\mathcal{O}_{X,x}^{\sh}$ is a Cohen-Macaulay noetherian local ring of dimension $c$. Moreover, $H^i(Ri^! \mathbb{F}_p[d]) = R^i\Gamma (Ri^! \mathbb{F}_p[d])$ for all $i$ because $\mathcal{O}^{\sh}_{X,x}$ is strictly henselian. Now it follows from Proposition \ref{locvan} that $H^i(Ri^! \mathbb{F}_p[d]) = 0$ for all $i < -d+c$.  
\end{myproof2}

\begin{rmrk}
The constant sheaf $\mathbb{F}_p[\dim X]$ is not always perverse for an arbitrary $k$-scheme $X$. For example, let $X$ be a union of two planes meeting only at a single point $p$. Then if $i \colon \Spec(k) \rightarrow X$ is the inclusion of $p$, we claim that $H^{-1}(Ri^! \mathbb{F}_p[2]) \neq 0$. To prove this, let $j \colon U \rightarrow X$ be the inclusion of the complement of $p$. Then by using the exact triangle (\ref{triangle}), there is an embedding of $H^0_{\text{\'{e}t}}(U, \mathbb{F}_p)/H^0_{\text{\'{e}t}}(X, \mathbb{F}_p)$ into $H^{-1}(Ri^! \mathbb{F}_p[2])$. As $U$ is disconnected, this proves the claim.
\end{rmrk}

The remainder of this section is devoted to the proof of Theorem \ref{mainthm2}. We will use the following two lemmas to show the vanishing of the cohomology groups in part (ii) of Proposition \ref{locvan}. 

\begin{lem} \label{ASlemma}
Let $(A, \mathfrak{m})$ be a normal local domain of characteristic $p > 0$, and let $U$ be the complement of the closed point. Then $$\ker(H_{\textnormal{\'{e}t}}^1(\Spec(A), \mathbb{F}_p) \rightarrow H_{\textnormal{\'{e}t}}^1(U, \mathbb{F}_p)) = 0.$$
\end{lem}

\begin{proof}
From the Artin-Schreier sequences on $\Spec(A)$ and $U$ we get a diagram of long exact sequences:
$$\xymatrix{
0 \ar[r]&  H^0_{\text{\'{e}t}}(U, \mathbb{F}_p) \ar[r]& H^0(U, \mathcal{O}_U) \ar[r]^{F-\id} & H^0(U, \mathcal{O}_U) \ar[r] & H^1_{\text{\'{e}t}}(U, \mathbb{F}_p) \ar[r] &\\ 
0 \ar[r]&  H^0_{\text{\'{e}t}}(\Spec(A), \mathbb{F}_p) \ar[r] \ar[u] & A \ar[u] \ar[r]^{F-\id} & A \ar[r] \ar[u] & H^1_{\text{\'{e}t}}(\Spec(A), \mathbb{F}_p) \ar[r]^(.8){0} \ar[u] &
}
$$
If an element in the kernel is represented by $a \in A$ then there exists $x \in H^0(U, \mathcal{O}_U)$ such that $a = x^p-x$. As $A$ is normal then $x \in A$, so the image of $a$ in $H_{\text{\'{e}t}}^1(\Spec(A), \mathbb{F}_p)$ is already $0$.
\end{proof}

\begin{lem} \label{Finv}
Let $(A,\mathfrak{m})$ be an excellent $F$-rational local ring of characteristic $p>0$ and dimension $d > 0$. Then $\ker(H_{\mathfrak{m}}^d(A) \xrightarrow{F-\id} H_{\mathfrak{m}}^d(A)) = 0$. 
\end{lem}

\begin{proof}
An element in the kernel generates an $A[F]$-submodule of $H_{\mathfrak{m}}^d(A)$ which is finitely generated as an $A$-module. As $H_{\mathfrak{m}}^d(A)$ is a simple $A[F]$-module and $H_{\mathfrak{m}}^d(A)$ is never finitely generated as an $A$-module \cite[7.3.3]{loccohbook}, then the kernel is zero.
\end{proof}

\begin{myproof2}[Proof of Theorem \ref{mainthm2}]
By \cite[4.2]{Frationalloc}, $X$ is normal and Cohen-Macaulay, so $\mathbb{F}_p[d]$ is perverse by Theorem \ref{mainthm}. By Theorem \ref{simplechar} it suffices to show that $\mathbb{F}_p[d]$ is the intermediate extension of its own restriction to the smooth locus. Let $x \in X$ be a point of codimension $c$ contained in the singular locus. Let $i \colon \overline{x} \rightarrow \Spec(\mathcal{O}_{X,x}^{\text{sh}})$ be the inclusion of the closed point in a strict henselization of the local ring of $X$ at $x$. By appealing to Lemma \ref{intext}, we need to check that $H^{-d+c}(Ri^* \mathbb{F}_p[d]) = 0$ and $H^{-d+c}(Ri^! \mathbb{F}_p[d]) = 0$. We have $H^{-d+c}(Ri^* \mathbb{F}_p[d]) = 0$ because $c > 0$. Since $\mathcal{O}^{\sh}_{X,x}$ is strictly henselian, $H^{-d+c}(Ri^! \mathbb{F}_p[d]) = R^{-d+c}\Gamma (Ri^! \mathbb{F}_p[d])$. By Lemma \ref{shprop}, $\mathcal{O}_{X,x}^{\text{sh}}$ is Cohen-Macaulay, normal, $F$-rational and has dimension $c$. Moreover, $\Spec(\mathcal{O}_{X,x}^{\text{sh}})$ is irreducible because $X$ is normal, and in particular, geometrically unibranch. The ring $\mathcal{O}_{X,x}^{\text{sh}}$ is also reduced by \cite[18.8.12]{EGA4IV}, so $\mathcal{O}_{X,x}^{\text{sh}}$ is a domain. Thus if $c=1$ then  $H^{-d+1}(Ri^! \mathbb{F}_p[d]) = 0$ by Proposition \ref{locvan} and Lemma \ref{ASlemma}. If $c > 1$ then  $H^{-d+c}(Ri^! \mathbb{F}_p[d]) = 0$ by Proposition  \ref{locvan} and Lemma \ref{Finv}. 
\end{myproof2}

\section{Global $F$-regularity of affine Schubert varieties} \label{GlobFReg}

\subsection{Background} \label{GlobFReg1.1} There have been multiple proofs of the normality of classical Schubert varieties in flag varieties (see \cite{AndersonSchubert}, \cite{RRNormality}, \cite{Seshadri}). If the base field has positive characteristic $p > 0$, one method due to Mehta-Srinivas \cite{FrobNormal} of proving normality uses the fact that classical Schubert varieties are Frobenius split ($F$-split). A scheme $X$ over $k$ is said to be \emph{$F$-split} if the natural $p$th-power map of $\mathcal{O}_X$-modules $\mathcal{O}_X \rightarrow F_*\mathcal{O}_X$ admits a splitting. $F$-splitting also implies $H^i(L, X) = 0$ for all $i > 0$ and $L$ an ample line bundle on $X$. Mehta-Ramanathan first introduced Frobenius splitting techniques for classical Schubert varieties in \cite{MehtaRam}. 

By extending the methods of Mehta-Ramanathan to the affine case, Faltings \cite{FaltingsLoop} showed affine Schubert varieties are normal, Cohen-Macaulay, $F$-split, and have rational singularities when $G$ is simple and simply connected. Pappas-Rapoport \cite{PappasRapoport} further extended this to the case $p \nmid |\pi_1(G_{\text{der}})|$. $F$-splitting is one example of a broader family of properties related to the Frobenius endomorphism, such as strong $F$-regularity and $F$-rationality. See \cite{TightClosure}, \cite{Frationalloc} for definitions and consequences of these and other similar notions. While these properties are typically defined locally on varieties, Smith \cite{SmithGlobally} has defined the notion of global $F$-regularity of a projective variety. We will give precise definitions in Section \ref{GlobFReg1.3}.

A globally $F$-regular variety enjoys several favorable properties. For example, such a variety is  strongly $F$-regular \cite{SmithGlobally}, and hence it is also reduced, normal, Cohen-Macaulay, $F$-split  and $F$-rational.  Lauritzen, Raben-Pedersen, and Thomsen \cite{SchubertFreg} showed that classical Schubert varieties are globally $F$-regular. Their proof is remarkably short, and relies on the fact that Mehta-Ramanathan proved something more powerful than $F$-splitting: they showed that Demazure varieties are compatibly Frobenius split along certain divisors. In this section we prove that affine Schubert varieties are globally $F$-regular.

For the rest of this section we fix a connected split reductive group $G$ defined over a perfect field $k$ of characteristic $p > 0$. Our proof of global $F$-regularity is analogous to the case of classical Schubert varieties in \cite{SchubertFreg}. To begin, we use results in \cite{PappasRapoport} to reduce to the case when $G$ is simply connected. In this setting the relevant facts about Frobenius splittings of Demazure varieties are known and follow from the original techniques of Mehta-Ramanathan.

We conclude this subsection by defining the relevant affine flag varieties and affine Schubert varieties.  Let $G_{\der}$ be the derived group of $G$, and let $T \subset B \subset G$ be a maximal torus and a Borel subgroup. Let $W = N_{G(k(\!(t)\!))}T(k(\!(t)\!))/T(k[\![t]\!])$ be the Iwahori-Weyl group. Define the loop group functor $LG$ on the category of $k$-algebras
$$LG \colon R \mapsto G(R (\!( t )\!))$$
and the positive loop group functor
$$L^+G \colon R \mapsto G(R[\![t]\!]).$$

Let $P \supset B$ be a parabolic subgroup, with preimage $\mathcal{P} \subset L^+G$. If $P=B$ (resp. $P=G$) then we usually write $\mathcal{B}$ (resp. $L^+G$) instead of $\mathcal{P}$. The (partial) affine flag variety is the fpqc-quotient
$$\mathcal{F}\ell_{\mathcal{P}}:=  LG/\mathcal{P}.$$ Given $w \in W$, choose a representative of $w$ in $LG(k)$ and let $n_w \in \mathcal{F}\ell_{\mathcal{P}}(k)$ be its image in the quotient. The affine Schubert variety $S_w$ is the reduced orbit closure
$$S_w : = \overline{\mathcal{B} \cdot n_w} \subset \mathcal{F}\ell_{\mathcal{P}}.$$ This is independent of the choice of representative of $w$. The affine Schubert varieties $S_w$ are projective and integral (see \cite[1.2.2]{ZhuGra}).

\begin{thm}[\cite{FaltingsLoop}, \cite{PappasRapoport}] \label{propsofgr}
Let $w \in W$. If $p \nmid |\pi_1(G_{\der})|$ then the affine Schubert variety $S_w$ is normal, Cohen-Macaulay, $F$-split and has rational singularities.
\end{thm}

Starting in Section \ref{FponGr} we focus on the case $P=G$. In this case, $\mathcal{F}\ell_{\mathcal{P}}$ is called the affine Grassmannian. Because of its importance and for consistency with existing literature, we will instead use the notation $$\Gr := LG/L^+G.$$ The $L^+G$-orbit closures in $\Gr$ have the following concrete description. Let $X_*(T)^+$ be the monoid of dominant cocharacters that are determined by the choice of positive roots corresponding to $B$. A dominant cocharacter $\mu \in X_*(T)^+$ induces by functoriality a map $\mathbb{G}_m (k (\!( t )\!) ) \rightarrow T(k (\!( t )\!))$, also denoted by  $\mu$. By the inclusion $T(k (\!( t )\!) ) \subset LG(k)$ we view $\mu(t)$ as a point in $\Gr(k)$. Let $\Gr_{\mu}$ be the reduced $L^+G$-orbit of $\mu(t)$ in $\Gr$. The affine Schubert variety $\Gr_{\leq \mu}$ is the reduced orbit closure
$$\Gr_{\leq \mu} : = \overline{L^+G \cdot \mu(t)} \subset \Gr.$$

\subsection{Demazure varieties} \label{GlobFReg1.2}
We have a semi-direct product decomposition $$W = W_{\af} \rtimes \Omega$$ where $W_{\af}$ is the affine Weyl group associated to the root system for $(G,B,T)$ and $\Omega$ is the normalizer of $\mathcal{B}$ in $W$. The choice of $\mathcal{B}$ determines a Coxeter group structure on $W_{\af}$. Given $w \in W_{\af}$, let $w = s_1 \cdots s_r$ be an arbitrary expression of $w$ as a product of simple reflections. We will use the notation $\tilde{w}$ to denote a pair consisting of $w \in W_{\af}$ together with an expression for $w$ as a product of simple reflections. 

Let $\mathcal{P}_i \subset LG$ be the unique parahoric subgroup containing both $\mathcal{B}$ and a fixed representative for $s_i$ such that $\mathcal{P}_i/\mathcal{B} \cong \mathbb{P}^1$. The Demazure variety associated to $\tilde{w}$ is
$$D_{\tilde{w}} : = \mathcal{P}_1 \times \cdots \times \mathcal{P}_r /\mathcal{B}^r.$$ Here $\mathcal{B}^r$ acts by the formula $(p_1, \ldots p_r) \cdot (b_1, \ldots, b_r) = (p_1 b_1, b_1^{-1} p_2 b_2, \ldots, b_{r-1}^{-1} p_r b_r).$ There is a multiplication map
$$\pi_w \colon D_{\tilde{w}} \rightarrow S_w \subset \mathcal{F}\ell_{\mathcal{B}}, \quad \quad \quad \quad [(p_1, \ldots, p_r)] \mapsto [p_1 \cdots p_r].$$ If $v = s_1 \cdots s_{r-1}$ so that $w = v \cdot s_r$ then there is a $\mathbb{P}^1$-bundle $D_{\tilde{w}} \rightarrow D_{\tilde{v}}$ obtained by forgetting the last coordinate. As $D_{\tilde{w}}$ is an iterated $\mathbb{P}^1$-bundle it is smooth. If $\tilde{v} \leq \tilde{w}$ is obtained from $w$ by removing one or more of the $s_i$ there is a natural inclusion $D_{\tilde{v}} \rightarrow D_{\tilde{w}}$ (see \cite[6.4.1]{ZhuCoherence}). 

\subsection{Global $F$-regularity} \label{GlobFReg1.3}
Let $R$ be a $k$-algebra, and let $F$ be the absolute Frobenius endomorphism on $\Spec(R)$. For every integer $e \geq 0$ there is a map of $R$-modules $R \rightarrow F_*^e R$ which sends $a$ to $a^{p^e}$. For any $c \in R$, by applying $F_*^e$ to the map $R \rightarrow R$ given by multiplying by $c$ and precomposing with $R \rightarrow F_*^e R$, we get a map $R \rightarrow F_e^*R$ sending $1$ to $c$. 

\begin{defn} \cite{TightClosure}
A finitely generated $k$-algebra $R$ is \emph{strongly $F$-regular} if for every $c \in R$ not contained in a minimal prime of $R$ there is an integer $e > 0$ such that the map $$R \rightarrow F_*^e R, \quad \quad \quad \quad 1 \mapsto c$$ splits.
\end{defn} 

Let $X$ be a connected projective variety over $k$. An ample invertible $\mathcal{O}_X$-module $\mathcal{L}$ determines a section ring
$$S = \bigoplus_{n \in \mathbb{N}} H^0(X, \mathcal{L}^n).$$

\begin{defn} \cite{SmithGlobally}
The variety $X$ is \emph{globally $F$-regular} if it admits a section ring which is strongly $F$-regular.
\end{defn}

For our purposes we will need several alternative methods of showing global $F$-regularity. To describe the first of these methods, let $D$ be an effective Cartier divisor on $X$ defined by a section $s$. Then there is a map $\mathcal{O}_X \rightarrow \mathcal{O}_X(D)$ that sends $1$ to $s$. By applying $F^e_*$ to this map and precomposing with the natural map $\mathcal{O}_X \rightarrow F^e_* \mathcal{O}_X$, we get a map $\mathcal{O}_X \rightarrow F_*^e \mathcal{O}_X(D)$ of $\mathcal{O}_X$-modules such that $1 \mapsto s$. The variety $X$ is said to be \emph{Frobenius split along $D$} if the map $\mathcal{O}_X \rightarrow F_* \mathcal{O}_X(D)$ splits. More generally, $X$ is \emph{stably Frobenius split along $D$} if $\mathcal{O}_X \rightarrow F^e_* \mathcal{O}_X(D)$ splits for some integer $e > 0$. Note that splitting for one integer $e$ implies splitting for all integers $e' \geq e$ (see \cite[3.5]{SmithGlobally}). 

\begin{lem} \label{splitlem1}
Let $X$ be a smooth, connected, projective variety over $k$. Then $X$ is globally $F$-regular if and only if $X$ is stably Frobenius split along an ample divisor.
\end{lem}

\begin{proof}
This follows from \cite[3.10]{SmithGlobally}.
\end{proof}

\begin{lem} \label{splitlem2}
Let $\pi \colon X \rightarrow Y$ be a morphism between connected projective varieties over $k$ such that $\pi_* \mathcal{O}_X = \mathcal{O}_Y$. Then if $X$ is globally $F$-regular, $Y$ is also globally $F$-regular.
\end{lem}

\begin{proof}
See \cite[Lemma 1.2]{SchubertFreg}.
\end{proof}

\begin{lem} \label{prodlem1}
Let $X$, $Y$ be globally $F$-regular varieties. Then $X \times Y$ is globally $F$-regular. 
\end{lem}

\begin{proof}Let $A$ and $B$ be strongly $F$-regular section rings for $X$ and $Y$, respectively. These are naturally $\mathbb{N}$-graded rings. A section ring for the product $X \times Y$ is the Segre product $$A \# B : = \bigoplus_{n \in \mathbb{N}} A_n \otimes B_n.$$ Being a direct summand of the strongly $F$-regular ring $A \otimes_k B$, the Segre product $A \# B$ is also strongly $F$-regular.\footnote{We thank Karen Smith for pointing out this proof.} See \cite[5.2]{ProdSegre} for a proof of these claims, which are attributed to Hochster.
\end{proof}

Before starting the proof we need one more notion. Let $Y \subset X$ be a closed subvariety with ideal sheaf $\mathcal{I}$. Then $Y$ is said to be compatibly split in $X$ if there is a splitting $\phi \colon F_* \mathcal{O}_X \rightarrow \mathcal{O}_X$ such that $\phi(F_* \mathcal{I}) \subset \mathcal{I}$. In this case $\phi$ induces a splitting of $Y$. The following lemma relates this property to Frobenius splitting along a divisor.

\begin{lem} \label{divlem}
Let $X$ be a variety and let $D \subset X$ be an effective Cartier divisor on $X$ defined by a section $s$. If $D$ is compatibly split in $X$ then $X$ is Frobenius split along the divisor $D$.
\end{lem}

\begin{proof}
The variety $X$ is Frobenius split along $(p-1)D$ by \cite[Lemma 1.1]{SchubertFreg}.  This implies $X$ is Frobenius split along $D$ as well (see \cite[3.7]{SmithGlobally}). 
\end{proof}

\subsection{Proof of Global $F$-regularity} \label{GlobFReg1.4}

\begin{prop} \label{simpprop}
Let $G$ be a split, simply connected, simple reductive group and let $w \in W_{\af}$ with reduced expression $w = s_1 \cdots s_r$. Then $D_{\tilde{w}}$ is globally $F$-regular.
\end{prop}

\begin{proof}
Let $v_i = s_1 \cdots s_{i-1} \hat{s}_i s_{i+1} \cdots s_r$. By \cite[9.6]{PappasRapoport} there is a single Frobenius splitting of the Demazure variety $D_{\tilde{w}}$ which is compatible with all the closed immersions $D_{\tilde{v}_i} \rightarrow D_{\tilde{w}}$ (see also \cite[Ch. 8.18]{OlivierKac}, and \cite[3.20]{GortzShimura} for the case of $\SL_n$). Strictly speaking, Pappas and Rapoport only make this claim when $\tilde{v}_i$ gives a reduced expression for $v_i$, so we will say some more to justify our claim.

The issue is that Pappas and Rapoport require reduced expressions in their definition of Demazure varieties. For their proof they cite the original criterion of Mehta-Ramanathan \cite[Prop. 8]{MehtaRam}, which requires one to show $$\omega_{D_{\tilde{w}}}^{-1} \cong \sum_{i=1}^r \mathcal{O}(D_{\tilde{v}_i}) \otimes \mathcal{O}(E),$$ where $E$ is an effective divisor whose support does not meet $\cap_i D_{\tilde{v}_i}$. The precise form $E$ takes in our situation can be found in \cite[6.4.2]{ZhuCoherence}. By \cite[Prop. 8]{MehtaRam},  $D_{\tilde{w}}$ is compatibly Frobenius split with all the closed immersions $D_{\tilde{v}_i} \rightarrow D_{\tilde{w}}$, regardless of whether $\tilde{v}_i$ gives a reduced expression for $v_i$.

It follows from the definitions that $D_{\tilde{w}}$ is compatibly split along the closed subscheme defined by the product of the ideal sheaves of the $D_{\tilde{v}_i}$. Now by Lemma \ref{divlem}, $D_{\tilde{w}}$ is Frobenius split along the divisor $\sum_{i=1}^r D_{\tilde{v}_i}$. Then by \cite[3.7, 3.9]{SmithGlobally}, $D_{\tilde{w}}$ is stably Frobenius split along $\sum_{i=1}^r m_i D_{\tilde{v}_i}$  for any integers $m_i > 0$. By Lemma \ref{splitlem1}, we need only show this divisor is ample for some $m_i$.

In the classical case, Lauritzen-Thomsen \cite[6.1]{BottSamelsonLine} show that $D_{\tilde{v}_r} + \sum_{i=1}^{r-1} m_i D_{\tilde{v}_i}$ is ample for some integers $m_i > 0$. A similar proof works in our setting. The case $r=1$ is clear because then $D_{\tilde{w}} \cong \mathbb{P}^1$ and $D_{\tilde{v}_1}$ is a point. For $r>1$ we note that since $\pi \colon D_{\tilde{w}} \rightarrow D_{\tilde{v}_r}$ admits a section then for each closed point $b \in D_{\tilde{v}_r}$ there exist isomorphisms $\pi^{-1}(b) \cong \mathbb{P}^1$ and $\restr{\mathcal{O}(D_{\tilde{v}_r})}{\pi^{-1}(b)} \cong \mathcal{O}_{\mathbb{P}^1}(1)$. Using this fact one can show that $V: = \pi_{*}\mathcal{O}(D_{\tilde{v}_r})$ is a vector bundle and $D_{\tilde{w}} \cong \mathbb{P}V$. In particular, $\mathcal{O}(D_{\tilde{v}_r}) \cong \mathcal{O}_{\mathbb{P}V}(1)$ is $\pi$-ample, so the claim follows by induction and \cite[0892]{stacks-project}.
\end{proof}

\begin{myproof}[Proof of Theorem \ref{thm1}] \vspace{.1cm}
\emph{Reduction to $P=B$.} Let $S_w^{\mathcal{P}} \subset \mathcal{F}\ell_{\mathcal{P}}$ be an affine Schubert variety. The natural projection $\pi \colon \mathcal{F}\ell_{\mathcal{B}} \rightarrow \mathcal{F}\ell_{\mathcal{P}}$ is an \'{e}tale locally trivial $P/B$-bundle. We claim $\pi^{-1}(S_w^{\mathcal{P}})$ is isomorphic to an affine Schubert variety in $\mathcal{F}\ell_{\mathcal{B}}$. To prove this, we first observe that $\pi^{-1}(S_w^{\mathcal{P}})$ is integral as both $S_w^{\mathcal{P}}$ and $P/B$ are integral. It is also clear that $\pi^{-1}(S_w^{\mathcal{P}})$ is a union of $\mathcal{B}$-orbits. The $\mathcal{B}$-orbits in  $\mathcal{F}\ell_{\mathcal{B}}$ and their closure relations are determined by the Bruhat order on $W_{\af}$ (see \cite[2.18]{AffineFlagManifolds}). Since $\pi^{-1}(S_w^{\mathcal{P}}) \subset \mathcal{F}\ell_{\mathcal{B}}$ is of finite type it meets only finitely many $\mathcal{B}$-orbits, and being integral and closed it is necessarily the closure of a unique open dense $\mathcal{B}$-orbit. Thus $ \pi^{-1}(S_w^{\mathcal{P}}) = S_{w'}^{\mathcal{B}} \subset \mathcal{F}\ell_{\mathcal{B}}$ for some $w' \in W$. Since $H^0(P/B, \mathcal{O}_{P/B}) = k$ then by flat base change, ${(\restr{\pi}{S_{w'}^{\mathcal{B}}})}_*(\mathcal{O}_{S_{w'}^{\mathcal{B}}}) = \mathcal{O}_{S_{w}^{\mathcal{P}}}$. Thus the global $F$-regularity of $S_w^{\mathcal{P}}$ follows from that of $S_{w'}^{\mathcal{B}}$ by Lemma \ref{splitlem2}, so we can drop the superscript and assume that $S_w \subset \mathcal{F}\ell_{\mathcal{B}}$.

\vspace{.1cm}

\emph{Reduction to $G$ simply connected.} Let $\tilde{G}_{\der}$ be a simply connected cover of $G_{\der}$. By \cite[8.e.3, 8.e.4]{PappasRapoport}, $S_w$ is isomorphic to an affine Schubert variety in the affine flag variety for $\tilde{G}_{\der}$. This is where we need $p \nmid |\pi_1(G_{\der})|$. 

\vspace{.1cm}

\emph{Reduction to $G$ simply connected and simple.} If $G$ simply connected but not simple choose isomorphisms $G \cong \prod_i G_i$ and $B \cong \prod B_i$ where each $G_i$ is simple and simply connected, and the $B_i \subset G_i$ are Borel subgroups. This induces isomorphisms on flag varieties $\mathcal{F}\ell_{\mathcal{B}} \cong \prod_i \mathcal{F}\ell_{\mathcal{B}_i}$ and Iwahori-Weyl groups $W \cong \prod_i W_i$. If we write $w = \prod_i w_i$ for $w_i \in W_i$  then $S_w \cong \prod_i S_{w_i}$ (see\cite[8.e.2]{PappasRapoport}). Now apply Lemma \ref{prodlem1}.

\vspace{.1cm}

\emph{Proof when $G$ is simply connected and simple.}
Because $G$ is simply connected then $W = W_{\af}$. Let $w = s_1 \cdots s_r$ be a reduced expression for $w$. By Proposition \ref{simpprop} the Demazure variety $D_{\tilde{w}}$ is globally $F$-regular. The map $\pi_w \colon D_{\tilde{w}} \rightarrow S_w$ is a rational resolution of singularities by \cite[8.4]{PappasRapoport} and \cite[9.7(d)]{PappasRapoport}, so we conclude by Lemma \ref{splitlem2}. 
\end{myproof}

\section{Perverse $\mathbb{F}_p$-sheaves on the affine Grassmannian} \label{FponGr}
For the rest of this paper we return to the case $k$ is an algebraically closed field of characteristic $p > 0$, and we fix a connected reductive group $G$ over $k$. Throughout this section we assume $p \nmid |\pi_1(G_{\der})|$. We will explain how to remove this hypothesis in Remark \ref{rmrkh}. In this section we define the category $P_{L^+G}(\Gr, \mathbb{F}_p)$ and prove Theorem \ref{convsimp}. We also define a fiber functor on $P_{L^+G}(\Gr, \mathbb{F}_p)$ and investigate extensions between objects. The basic facts we will use about the geometry of $\Gr$ and the associated convolution Grassmannians are well-known (see, for example, \cite{RicharzGS}), so we will state them without proof. We use the same notation as in Section \ref{GlobFReg1.1}.

\subsection{Definition of $P_{L^+G}(\Gr, \mathbb{F}_p)$}
There is a partial order on $X_*(T)^+$ defined by setting $\lambda \leq \mu$ if $\mu - \lambda$ is a sum of positive coroots. For each $\mu \in  X_*(T)^+$, there are only finitely many $\lambda$ such that $\lambda \leq \mu$. There is a stratification
$$\Gr_{\leq \mu} = \coprod_{\lambda \leq \mu} \Gr_{\lambda}.$$
The orbit $\Gr_{\mu}$ is open in its closure. Let $\rho$ be half the sum of the positive roots. Then
$$\dim \Gr_{\mu} = 2 \langle \rho, \mu \rangle.$$ We note that if $\lambda < \mu $ then $\Gr_{\leq \lambda}$ has codimension at least $2$ in $\Gr_{\leq \mu}$. The reduced ind-closed subscheme of $\Gr$ is
$$\Gr_{\text{red}} = \lim_{\to} \Gr_{\leq \mu}.$$ Define
$$D_c^b(\Gr, \mathbb{F}_p) : = \lim_{\to} D_c^b( \Gr_{\leq \mu}, \mathbb{F}_p),$$ and
$$P_c^b(\Gr, \mathbb{F}_p) := \lim_{\to} P_c^b( \Gr_{\leq \mu}, \mathbb{F}_p).$$

The setup in Remark \ref{indequiv} applies and we can define $L^+G$-equivariant perverse $\mathbb{F}_p$-sheaves on $\Gr$. More precisely, for any integer $n \geq 0$, define the smooth connected affine group scheme
$$L^nG \colon R \rightarrow G(R[t]/t^n).$$ For fixed $\mu$, the group scheme $L^+G$ acts on $\Gr_{\leq \mu}$ through a quotient $L^nG$ for $n$ sufficiently large. Suppose $m > n$ and $L^+G$ acts on $\Gr_{\leq \mu}$ through both quotients $L^mG$ and $L^nG$. Note that there is an exact sequence
$$ 0 \rightarrow N \rightarrow L^mG \rightarrow L^nG \rightarrow 0$$ where $N$ is a connected, affine, smooth unipotent group. Thus by Proposition \ref{welldef} there is a natural isomorphism $P_{L^nG}(\Gr_{\leq \mu}, \mathbb{F}_p) \cong P_{L^mG}(\Gr_{\leq \mu}, \mathbb{F}_p)$. Hence we can unambiguously define $P_{L^+G}(\Gr_{\leq \mu}, \mathbb{F}_p) : = P_{L^nG}(\Gr_{\leq \mu}, \mathbb{F}_P)$ where $n$ is an integer such that $L^+G$ acts on $\Gr_{\leq \mu}$ through the quotient $L^nG$. We can now define
$$P_{L^+G}(\Gr, \mathbb{F}_p) =  \lim_{\to} P_{L^+G}(\Gr_{\leq \mu}, \mathbb{F}_p).$$ 
By similar reasoning we can define the category $P_{\mathcal{P}}(\mathcal{F} \ell_{\mathcal{P}}, \mathbb{F}_p)$.

\begin{myproof2}[Proof of Theorem \ref{irreducibleobjects}]
By \cite[2.3]{RicharzGeneral} the stabilizers for the action of $\mathcal{P}$ on $\mathcal{F}\ell_{\mathcal{P}}$ are connected. Thus by Corollary \ref{simpobjs}, the simple objects in $P_{\mathcal{P}}(\mathcal{F}\ell_{\mathcal{P}}, \mathbb{F}_p)$ are the intermediate extensions of constant sheaves along orbits. By the global $F$-regularity of $S_w$ we are in a position to apply Theorem \ref{mainthm2}. Thus the simple objects are the constant sheaves $\mathbb{F}_p[\dim S_w]$ supported on the $\mathcal{P}$-orbit closures.
\end{myproof2}

\begin{rmrk}
By the same reasoning as in the proof of Theorem \ref{irreducibleobjects} it follows that for any $\mathcal{P}$ the constant sheaves supported on $\mathcal{B}$-orbit closures in $\mathcal{F} \ell_{\mathcal{P}}$ are simple. If the stabilizers for the action of $\mathcal{B}$ are connected then these are all the simple objects in $P_{\mathcal{B}}(\mathcal{F} \ell_{\mathcal{P}}, \mathbb{F}_p)$. 
\end{rmrk}

\subsection{Convolution} The convolution diagram is
\begin{equation}\ \label{convdiagram}
\Gr \times \Gr \xleftarrow{p} LG \times \Gr \xrightarrow{q} LG \times^{L^+G} \Gr \xrightarrow{m} \Gr.
\end{equation} Here $p$ is the quotient map on the first factor, and $q$ is the quotient by the diagonal action of $L^+G$ given by $g \cdot (g_1, g_2) = (g_1 g^{-1}, g g_2)$. The map $m$ is the multiplication map $m(g_1, g_2) = g_1g_2$. Note that $p$ and $q$ are $L^+G$-torsors. For simplicity we ignore the fact that $LG \times \Gr$ is not of ind-finite type and allow ourselves to speak of perverse sheaves on $LG \times \Gr$. The same technical method as in \cite[3.21]{RicharzGS} works here for overcoming this issue. The idea is that for any particular perverse sheaves we are working with, we can replace $p$ and $q$ by $L^n G$-torsors for $n$ sufficiently large. 

More generally there is a convolution morphism
$LG \times^{L^+G} \cdots \times^{L^+G}  \Gr \rightarrow \Gr$. To describe the ind-scheme structure on the convolution Grassmannian, let $\mu_i \in X_*(T)^+$ for $i =1, \ldots, n$ be a collection of dominant cocharacters and let $\mu$ be their sum. Let $f \colon LG \rightarrow \Gr$ be the quotient map. Then we define
$$\Gr_{\leq \mu_\bullet} : = f^{-1}(\Gr_{\leq \mu_1}) \times^{L^+G} \cdots \times^{L^+G}  f^{-1}(\Gr_{\leq \mu_{n-1}}) \times^{L^+G} \Gr_{\leq \mu_n},$$ where we always take the reduced subscheme structure. The convolution morphism restricts to a map $m \colon \Gr_{\leq \mu_\bullet} \rightarrow \Gr_{\leq \mu}$ which is proper, birational, and an isomorphism over $\Gr_{\mu}$. Using the construction in Remark \ref{indequiv} we can define $$P_{L^+G}(LG \times^{L^+G} \cdots \times^{L^+G}  \Gr, \mathbb{F}_p),$$ where $L^+G$ acts on the leftmost factor $LG$. 

\begin{lem} \label{twistedprod}
Let $\mathcal{F}^\bullet$, $\mathcal{G}^\bullet \in P_{L^+G}(\Gr, \mathbb{F}_p)$. Then there is a unique perverse sheaf $$\mathcal{F}^\bullet \overset{\sim}{\boxtimes} \mathcal{G}^\bullet \in P_{L^+G}(LG \times^{L^+G} \Gr, \mathbb{F}_p)$$ such that
$$Rp^*(\mathcal{F}^\bullet \overset{L}{\boxtimes} \mathcal{G}^\bullet) \cong Rq^*(\mathcal{F}^\bullet \overset{\sim}{\boxtimes} \mathcal{G}^\bullet).$$
\end{lem}

\begin{proof}
We claim that $\mathcal{F}^\bullet \overset{L}{\boxtimes} \mathcal{G}^\bullet$ is perverse. In general we do not know if the box product of two perverse $\mathbb{F}_p$-sheaves is perverse (the analogous fact is true for $\overline{\mathbb{Q}}_\ell$-coefficients with the appropriate $t$-structure \cite[4.2.8]{BBD}). However, we can prove that $\mathcal{F}^\bullet \overset{L}{\boxtimes} \mathcal{G}^\bullet$ is perverse in our situation. First, note that the schemes $\Gr_{\leq \mu} \times \Gr_{\leq \lambda}$ for $\mu, \lambda \in X_*(T)^+$ are Cohen-Macaulay because a product of Cohen-Macaulay schemes is Cohen-Macaulay \cite[045Q]{stacks-project}. Now if $\mathcal{F}^\bullet$ and $\mathcal{G}^\bullet$ are simple then $\mathcal{F}^\bullet \overset{L}{\boxtimes} \mathcal{G}^\bullet$ is perverse because it is a constant sheaf supported on a Cohen-Macaulay scheme (and it is in the correct degree). 

Now suppose $\mathcal{F}^\bullet$ is simple and $\mathcal{G}^\bullet$ has length $n > 1$. Because subquotients of equivariant perverse sheaves are equivariant (Proposition \ref{subquotprop}), there is an exact sequence
$$ 0 \rightarrow \mathcal{G}^{\bullet}_1 \rightarrow \mathcal{G}^\bullet \rightarrow \mathcal{G}_2^{\bullet} \rightarrow 0$$ in $P_{L^+G}(\Gr, \mathbb{F}_p)$ such that $\mathcal{G}^{\bullet}_1$, $\mathcal{G}^{\bullet}_2$ have length $< n$. This gives rise to an exact triangle
$$\mathcal{F}^\bullet \overset{L}{\boxtimes} \mathcal{G}_1^\bullet \rightarrow \mathcal{F}^\bullet \overset{L}{\boxtimes} \mathcal{G}^\bullet \rightarrow \mathcal{F}^\bullet \overset{L}{\boxtimes} \mathcal{G}_2^\bullet$$ in $D_c^b(\Gr \times \Gr, \mathbb{F}_p)$. By induction the outer two terms are perverse, thus so is $\mathcal{F}^\bullet \overset{L}{\boxtimes} \mathcal{G}^\bullet$. A similar induction on the length of $\mathcal{F}^\bullet$ completes the proof that $\mathcal{F}^\bullet \overset{L}{\boxtimes} \mathcal{G}^\bullet$ is perverse. 

Now that we know $\mathcal{F}^\bullet \overset{L}{\boxtimes} \mathcal{G}^\bullet$ is perverse then so is $Rp^*(\mathcal{F}^\bullet \overset{L}{\boxtimes} \mathcal{G}^\bullet)$ (up to a shift) because $p$ is smooth. Because $\mathcal{G}^\bullet$ is $L^+G$-equivariant then $Rp^*(\mathcal{F}^\bullet \overset{L}{\boxtimes} \mathcal{G}^\bullet)$ is equivariant for the diagonal action of $L^+G$. This equivariance is precisely the data needed to use Corollary \ref{smoothdec} and descend $Rp^*(\mathcal{F}^\bullet \overset{L}{\boxtimes} \mathcal{G}^\bullet)$ along the quotient map $q$ to a perverse sheaf $\mathcal{F}^\bullet \overset{\sim}{\boxtimes} \mathcal{G}^\bullet \in P_c^b(LG \times^{L^+G} \Gr, \mathbb{F}_p)$. Finally, one can use smooth descent and the fact that $\mathcal{F}^\bullet$ is $L^+G$-equivariant to show $\mathcal{F}^\bullet \overset{\sim}{\boxtimes} \mathcal{G}^\bullet$ is $L^+G$-equivariant.
\end{proof}

If $\mathcal{F}^\bullet$, $\mathcal{G}^\bullet \in P_{L^+G}(\Gr, \mathbb{F}_p)$ we define their \emph{convolution product} to be
$$\mathcal{F}^\bullet * \mathcal{G}^\bullet := Rm_*(\mathcal{F}^\bullet \overset{\sim}{\boxtimes} \mathcal{G}^\bullet) \in D_c^b(\Gr, \mathbb{F}_p).$$
Our next goal is to prove that $\mathcal{F}^\bullet * \mathcal{G}^\bullet \in P_{L^+G}(\Gr, \mathbb{F}_p)$. Once we know that $\mathcal{F}^\bullet * \mathcal{G}^\bullet$ is perverse, the fact that it is equivariant will follow by Lemma \ref{equivfun} since $m$ is proper and $L^+G$-equivariant.

For later use, we also define the $n$-fold convolution product. Let $\mathcal{F}_i \in P_{L^+G}(\Gr, \mathbb{F}_p)$ for $i=1, \ldots, n$. By the same reasoning as in Lemma \ref{twistedprod} we can define their twisted product $$\overset{\sim}{\boxtimes}_i \mathcal{F}_i^\bullet \in P_{L^+G}(LG \times^{L^+G} \cdots \times^{L^+G} \Gr, \mathbb{F}_p).$$ Let $m \colon LG \times^{L^+G} \cdots \times^{L^+G} \Gr \rightarrow \Gr$ be the $n$-fold convolution map. Then we define 
$$ \underset{i}{\mathlarger{\mathlarger{*}}} \mathcal{F}_i^\bullet := R{m_{*}}(\overset{\sim}{\boxtimes}_i \mathcal{F}_i^\bullet) \in D_c^b(\Gr, \mathbb{F}_p).$$

For $\overline{\mathbb{Q}}_\ell$-coefficients, the perversity of the convolution product can be proved by using the fact that $m$ is a stratified semi-small map (see \cite{Lusztig:Singularities}, \cite{GeometricSatake}) or the notion of universally locally acyclic complexes (see \cite{RicharzGS}). We will take a different approach which uses the fact that $\Gr$ has rational singularities. Before giving the proof of Theorem \ref{convsimp}, we illustrate why the case of $\mathbb{F}_p$-coefficients differs from $\overline{\mathbb{Q}}_\ell$-coefficients with the following example. 

\begin{exmp} \label{example1} Let $G = \GL_2$ and use the standard choice of $T \subset B \subset G$ and the identification $X_*(T) \cong \mathbb{Z}^2$. Consider the two minuscule cocharacters $\mu_1 = (1,0)$ and $\mu_2 = (0,-1)$ with corresponding IC sheaves $\IC_{\mu_1}$ and $\IC_{\mu_2}$. These are both constant sheaves (shifted by 1) supported on schemes that are isomorphic to $\mathbb{P}^1$. Let $\mu = \mu_1 + \mu_2$. Then $\Gr_{\leq \mu}$  is a two-dimensional surface with a unique singular point $e_0$. The convolution Grassmannian $\Gr_{\leq \mu_\bullet}$ over $\Gr_{\leq \mu}$ is a resolution of singularities which is an isomorphism away from $e_0$ and whose fiber over $e_0$ is isomorphic to $\mathbb{P}^1$. 

For both $\overline{\mathbb{Q}}_\ell$ and $\mathbb{F}_p$-coefficients the convolution $\IC_{\mu_1} * \IC_{\mu_2}$ is given by the derived pushforward of the constant sheaf (shifted by 2) along the convolution map $m \colon \Gr_{\leq \mu_\bullet} \rightarrow \Gr_{\leq \mu}$. For $\overline{\mathbb{Q}}_\ell$-coefficients the result is $Rm_*(\overline{\mathbb{Q}}_\ell[2]) \cong \overline{\mathbb{Q}}_\ell[2] \oplus C[0]$ where $C$ is a skyscraper sheaf supported at $e_0$. The stalk of the summand $C[0]$ at $e_0$ is isomorphic to $H^2_{\text{\'{e}t}}(\mathbb{P}^1, \overline{\mathbb{Q}}_\ell)$. When we work with \'{e}tale $\mathbb{F}_p$-sheaves, $Rm_*(\mathbb{F}_p[2]) \cong \mathbb{F}_p[2]$ because $H_{\text{\'{e}t}}^i(\mathbb{P}^1, \mathbb{F}_p) = 0$ for $i > 0$. More generally, we will prove in this section that for any group $G$ the constant sheaf $\mathbb{F}_p[0]$ is preserved under the derived pushforward along any convolution morphism.
\end{exmp} 

\begin{lem} \label{convproperties}
Let $\mu_i \in X_*(T)^+$ for $i =1, \ldots, n$. Then the scheme $\Gr_{\leq \mu_\bullet}$  is projective, integral, normal, and Cohen-Macaulay. 
\end{lem}

\begin{proof} There is an isomorphism of ind-schemes $LG \times^{L^+G} \cdots \times^{L^+G} \Gr \cong \Gr^n$ (see \cite[1.2.14]{ZhuGra}). This implies $\Gr_{\leq \mu_\bullet}$ is projective. For the other properties we proceed by induction on the number $n$ of dominant cocharacters. Let $\lambda_i = \mu_{i+1}$ for $i = 1, \ldots, n-1$. Let $k$ be an integer large enough so that $L^+G$ acts on $\Gr_{\leq \lambda_\bullet}$ on the left through the quotient $L^kG$. Let $$\Gr_{\leq \mu_1}^k := LG \mid_{\Gr_{\leq \mu_1}} \times^{L^+G} L^kG.$$ The map $\Gr_{\leq \mu_1}^k \rightarrow \Gr_{\leq \mu_1}$ is a right $L^kG$-torsor. Thus $\Gr_{\leq \mu_1}^k$ is normal and Cohen-Macaulay because these properties are local in the \'{e}tale topology. As $\Gr_{\leq \mu_1}$ is also integral and $L^kG$ is connected then $\Gr_{\leq \mu_1}^k$ is integral. By induction $\Gr_{\leq \mu_1}^k \times \Gr_{\leq \lambda_\bullet}$ is normal, integral, and Cohen-Macaulay. Note that $\Gr_{\leq \mu_\bullet}$ is the quotient of $\Gr_{ \leq \mu_1}^k \times \Gr_{\leq \lambda_\bullet}$ by the diagonal action of $L^kG$. Hence $\Gr_{\leq \mu_\bullet}$ is normal and Cohen-Macaulay. It is integral because it is reduced and it is the image of an irreducible scheme. 
\end{proof}

Let $Y$ be a Cohen-Macaulay $k$-scheme. We say that $Y$ has \emph{rational singularities} if there exists a smooth $k$-scheme $X$ and a proper birational morphism $f \colon X \rightarrow Y$ such that $\mathcal{O}_Y \rightarrow Rf_* \mathcal{O}_X$ is an isomorphism. Various authors have considered alternative definitions of rational singularities which do not rely on the existence of a resolution of singularities. One such definition is due to Lipman-Teissier \cite{pseudorat}, which they call \emph{pseudo-rational singularities} (see also \cite[1.2]{ratsing}). We will not need the definition here. All we will need is that rational singularities are pseudo-rational (\cite[9.6]{ratsing}).

\begin{prop} \label{convrat}
Let $m \colon  \Gr_{\leq \mu_\bullet} \rightarrow \Gr_{\leq \mu}$ be the convolution map corresponding to $\mu_i \in X_*(T)^+$. Let $\mathcal{O}_{\Gr_{\leq \mu_\bullet}}$ be the structure sheaf of $\Gr_{\leq \mu_\bullet}$. Then the natural map
$$\mathcal{O}_{\Gr_{\leq \mu}} \xrightarrow{\sim} R{m}_* (\mathcal{O}_{\Gr_{\leq \mu_\bullet}})$$ is an isomorphism. 
\end{prop}

\begin{proof}
Since $\Gr_{\leq \mu_\bullet}$ and $\Gr_{\leq \mu}$ are projective, $m$ is projective. By Theorem \ref{propsofgr}, $\Gr_{\leq \mu}$ is Cohen-Macaulay and has rational singularities. As rational singularities are pseudo-rational then the result follows from \cite[1.4]{ratsing}. 
\end{proof}

We can now prove Theorem \ref{convsimp}.

\begin{myproof2}[Proof of Theorem \ref{convsimp}]
Let $\mu = \mu_1+\mu_2$ and let $m \colon \Gr_{\leq \mu_\bullet}  \rightarrow \Gr_{\leq \mu}$ be the convolution map. By Lemma \ref{convproperties} and Theorem \ref{mainthm} the constant sheaf $\mathbb{F}_p[\dim \Gr_{\leq \mu_\bullet}]$ supported on $\Gr_{\leq \mu_\bullet}$ is perverse. From the definition of $\IC_{\mu_1} * \IC_{\mu_2}$ and the fact that the $\IC$ sheaves are constant we see that
$$\IC_{\mu_1} * \IC_{\mu_2} = Rm_*(\mathbb{F}_p[\dim \Gr_{\leq \mu_\bullet}]).$$ Thus, to show that $\IC_{\mu_1} * \IC_{\mu_2} = \IC_\mu$ it suffices to show that $Rm_* \mathbb{F}_p[0] \cong \mathbb{F}_p[0]$. This fact follows immediately from Proposition \ref{convrat} by applying $Rm_*$ to the Artin–Schreier sequence on $ \Gr_{\leq \mu_\bullet}$. Finally, for general $\mu \in X_*(T)^+$, Theorem \ref{cohvan} in the next section implies that $\dim_{\mathbb{F}_p} H(\IC_\mu) = 1$.
\end{myproof2}

\begin{cor} \label{convperv}
Let $\mathcal{F}^\bullet$, $\mathcal{G}^\bullet \in P_{L^+G}(\Gr, \mathbb{F}_p)$. Then $\mathcal{F}^\bullet * \mathcal{G}^\bullet \in P_{L^+G}(\Gr, \mathbb{F}_p)$.
\end{cor}

\begin{proof}
We first suppose $\mathcal{F}^\bullet$ is simple and induct on the length $n$ of $\mathcal{G}^\bullet$. Pick an exact sequence
$$ 0 \rightarrow \mathcal{G}^{\bullet}_1 \rightarrow \mathcal{G}^\bullet \rightarrow \mathcal{G}_2^{\bullet} \rightarrow 0$$ in $P_{L^+G}(\Gr, \mathbb{F}_p)$ such that $\mathcal{G}^{\bullet}_1$, $\mathcal{G}^{\bullet}_2$ have length $< n$. We claim that the sequence
$$ 0 \rightarrow \mathcal{F}^\bullet \overset{\sim}{\boxtimes} \mathcal{G}^{\bullet}_1 \rightarrow \mathcal{F}^\bullet \overset{\sim}{\boxtimes} \mathcal{G}^\bullet \rightarrow \mathcal{F}^\bullet \overset{\sim}{\boxtimes}  \mathcal{G}_2^{\bullet} \rightarrow 0$$ in $P_{L^+G}(LG \times^{L^+G} \Gr, \mathbb{F}_p)$ is exact. Indeed, exactness follows because this sequence is constructed by descending an exact sequence of perverse sheaves on $LG \times \Gr$ (again, we are suppressing the ind-finite type issue with $LG \times \Gr$). Now the fact that $Rm_*(\mathcal{F}^\bullet \overset{\sim}{\boxtimes} \mathcal{G}^\bullet)$ is perverse follows by induction. A similar induction on the length of $\mathcal{F}^\bullet$ completes the proof that  $\mathcal{F}^\bullet * \mathcal{G}^\bullet$ is perverse. Finally, $\mathcal{F}^\bullet * \mathcal{G}^\bullet$ is equivariant by Lemma \ref{equivfun} since $m$ is proper. 
\end{proof} 

From the proof of Corollary \ref{convperv}, we see that the convolution product is exact:

\begin{cor}
Let $\mathcal{F}^\bullet \in P_{L^+G}(\Gr, \mathbb{F}_p)$. Then the functor 
$$P_{L^+G}(\Gr, \mathbb{F}_p) \rightarrow P_{L^+G}(\Gr, \mathbb{F}_p), \quad \quad \quad \quad \mathcal{G}^\bullet \mapsto \mathcal{F}^\bullet * \mathcal{G}^\bullet$$ is exact. The same is true when we convolve on the right by $\mathcal{F}^\bullet $. 
\end{cor}

We conclude this section by giving $(P_{L^+G}(\Gr, \mathbb{F}_p), *)$ the structure of a monoidal category.

\begin{thm} \label{monoidal}
There is a natural associativity constraint so that $(P_{L^+G}(\Gr, \mathbb{F}_p), *)$ is a monoidal category. 
\end{thm}

\begin{proof}
This works the same as in the case of $\overline{\mathbb{Q}}_{\ell}$-coefficients. First, we note that the argument in the proof of Corollary \ref{convperv} shows that $\underset{i}{\mathlarger{\mathlarger{*}}} \mathcal{F}_i^\bullet \in P_{L^+G}(\Gr, \mathbb{F}_p)$ for any $\mathcal{F}_i^\bullet \in P_{L^+G}(\Gr, \mathbb{F}_p)$. By using the natural isomorphisms
$$(\mathcal{F}_1^\bullet \overset{L}{\boxtimes} \mathcal{F}^\bullet_2) \overset{L}{\boxtimes}  \mathcal{F}^\bullet_3 \cong \mathcal{F}_1^\bullet \overset{L}{\boxtimes} \mathcal{F}^\bullet_2 \overset{L}{\boxtimes}  \mathcal{F}^\bullet_3 \cong \mathcal{F}_1^\bullet \overset{L}{\boxtimes} (\mathcal{F}^\bullet_2 \overset{L}{\boxtimes}  \mathcal{F}^\bullet_2)$$ and the proper base change theorem, one can construct isomorphisms
$$(\mathcal{F}_1^\bullet * \mathcal{F}^\bullet_2) * \mathcal{F}^\bullet_3 \cong \mathcal{F}_1^\bullet * \mathcal{F}^\bullet_2 * \mathcal{F}^\bullet_3 \cong \mathcal{F}_1^\bullet * (\mathcal{F}^\bullet_2 *  \mathcal{F}^\bullet_3).$$ The unit object is the constant sheaf $IC_0 = \mathbb{F}_p[0]$ supported on the point $\Gr_{\leq 0}$. From the definitions it is clear how to construct isomorphisms $$\mathcal{F}^\bullet * IC_0 \cong \mathcal{F}^\bullet \cong IC_0 * \mathcal{F}^\bullet.$$ The coherence conditions can be checked using the corresponding properties for $- \overset{L}{\boxtimes} -$. 
\end{proof}

\subsection{Fiber functor}
In this section we study the cohomology of objects in $P_{L^+G}(\Gr, \mathbb{F}_p)$. The most important result is:

\begin{thm} \label{cohvan}
Let $\mu \in X_*(T)^+$ be a dominant cocharacter. Then
$$H^0_{\textnormal{\'{e}t}}(\Gr_{\leq \mu}, \mathbb{F}_p)= \mathbb{F}_p, \quad \quad \quad \quad  H^i_{\textnormal{\'{e}t}}(\Gr_{\leq \mu}, \mathbb{F}_p) = 0  \text{ for } i > 0.$$
\end{thm}

\begin{proof}
As $\Gr_{\leq \mu}$ is projective, $H^0(\Gr_{\leq \mu}, \mathcal{O}_{\Gr_{\leq \mu}}) = k$. Since $\Gr_{\leq \mu}$ is globally $F$-regular, $H^i(\Gr_{\leq \mu}, \mathcal{O}_{\Gr_{\leq \mu}}) = 0$ for $i > 0$ by \cite[4.3]{SmithGlobally}. The theorem now follows by considering the Artin-Schreier sequence. 
\end{proof}

\begin{cor} \label{evenodd}
Suppose $\mathcal{F}^\bullet \in P_{L^+G}(\Gr, \mathbb{F}_p)$ is supported on a connected component of $\Gr$ having strata whose dimensions have parity $n \in \mathbb{Z}/2$. Then $$R^i \Gamma(\mathcal{F}^\bullet) = 0, \quad \quad \quad \quad i \not\equiv n \pmod{2}.$$
\end{cor}

\begin{proof}
This follows by Theorem \ref{cohvan} and induction on the length of $\mathcal{F}^\bullet$.
\end{proof}

\begin{thm} \label{exactfaith}
The functor  $$H: = \bigoplus_i R^i \Gamma ( - ) \colon P_{L^+G}(\Gr, \mathbb{F}_p) \rightarrow \Vect_{\mathbb{F}_p}$$ is exact and faithful. Furthermore
$$\dim_{\mathbb{F}_p} H(\mathcal{F}^\bullet) = \length \mathcal{F}^\bullet.$$
\end{thm}

\begin{proof}
Exactness follows by Corollary \ref{evenodd}. The statement $\dim_{\mathbb{F}_p} H(\mathcal{F}^\bullet) = \length \mathcal{F}^\bullet$ holds for $\mathcal{F}^\bullet$ simple by Theorem \ref{cohvan}. For general $\mathcal{F}^\bullet$ this statement holds by exactness and induction on the length of $\mathcal{F}^\bullet$. Finally, faithfulness follows from exactness and the fact that $H(\mathcal{F}^\bullet) \neq 0$ if $\mathcal{F}^\bullet \neq 0$. 
\end{proof}

\subsection{Extensions} \label{extsection}

We now investigate extensions between perverse sheaves in $P_{L^+G}(\Gr, \mathbb{F}_p)$. 

\begin{prop} \label{extlem}
Let $\mu$, $\lambda \in X_*(T)^+$. Then $\Ext^1(\IC_\mu, \IC_\lambda) = 0$ in the category $P_c^b(\Gr, \mathbb{F}_p)$ if any one of the following holds:
\begin{enumerate}[{\normalfont (i)}]
\item $\mu > \lambda $,
\item $\mu = \lambda$,
\item $\mu\not\geq \lambda$ and $\mu \not\leq \lambda.$
\end{enumerate}

\end{prop}

\begin{proof}
The same proof as in \cite[4.1]{RicharzGS} works. We reproduce it here for completeness.

(i): Let $i \colon \Gr_{\leq \lambda} \rightarrow \Gr_{\leq \mu}$ be the inclusion. Then by adjunction $\Ext^1(\IC_\mu, \IC_\lambda)$ is 
$$\Hom_{D_c^b(\Gr_{\leq \mu}, \mathbb{F}_p)}(\IC_\mu, Ri_* \IC_\lambda[1]) = \Hom_{D_c^b(\Gr_{\leq \lambda}, \mathbb{F}_p)}(\mathbb{F}_p[\dim \Gr_{\leq \mu}], \mathbb{F}_p[\dim \Gr_{\leq \lambda}+1]).$$ This is zero because $\dim \Gr_{\leq \lambda} + 1 < \dim \Gr_{\leq \mu}$.

(ii): If $\mu = \lambda$ then
$\Ext^1(\IC_\mu, \IC_\mu)$ is $$\Hom_{D_c^b(\Gr_{\leq \mu}, \mathbb{F}_p)}(\mathbb{F}_p[\dim \Gr_{\leq \mu}], \mathbb{F}_p[\dim \Gr_{\leq \mu}+1]) = H^1_{\text{\'{e}t}}(\Gr_{\leq \mu}, \mathbb{F}_p).$$ This group is trivial by Theorem \ref{cohvan}. 

(iii): We can assume that $\Gr_{\leq \mu}$ and $\Gr_{\leq \lambda}$ lie in the same connected component of $\Gr$. Thus we can pick a dominant cocharacter $\nu > \mu, \lambda$. This gives rise to a Cartesian diagram
$$
\xymatrix{
Z \ar[r]^{i_1'} \ar[d]^{i_2'} & \Gr_{\leq \lambda} \ar[d]^{i_2} \\
\Gr_{\leq \mu} \ar[r]^{i_1} & \Gr_{\leq \nu}.
}$$ By proper base change and the adjunction between $Ri_2^*$, $Ri_{2,*}$ the group
$\Ext^1(\IC_\mu, \IC_\lambda)$ is$$\Hom_{D_c^b(\Gr_{\leq \nu}, \mathbb{F}_p)}(Ri_{1,*} \IC_\mu, Ri_{2,*} \IC_\lambda[1]) = \Hom_{D_c^b(\Gr_{\leq \lambda}, \mathbb{F}_p)}(Ri'_{1,*} Ri_2'^* \IC_\mu, \IC_\lambda[1]). $$ By the adjunction between $Ri'_{1,*}, Ri_1'^!$ this is the same as 
$$ \Hom_{D_c^b(Z, \mathbb{F}_p)}(Ri_2'^* \IC_\mu, Ri_1'^! \IC_\lambda[1]).$$ Note that $\dim Z < \dim \Gr_{\leq \mu}$. Let $U$ be the complement of $Z$ in $\Gr_{\leq \mu}$. The restriction of $\IC_\mu$ to $U$ is nonzero because $\IC_\mu$ is a constant sheaf on $\Gr_{\leq \mu}$. Thus by Lemma \ref{simplelem2}  the perverse sheaf $\IC_\mu$ is the intermediate extension of its own restriction to $U$. Now by Lemma \ref{intext}, $$Ri_2'^* \IC_\mu \in {^pD}^{\leq -1}(Z, \mathbb{F}_p).$$ By similar reasoning $Ri_1'^! \IC_\lambda[1] \in {^pD}^{\geq 0}(Z, \mathbb{F}_p)$. Hence $\Ext^1(\IC_\mu, \IC_\lambda) = 0$. 
\end{proof}

If $G=T$ is a torus $P_{L^+T}(\Gr, \mathbb{F}_p)$ is semi-simple because then $\Gr$ is a disjoint union of points.  We are not sure when $P_{L^+G}(\Gr, \mathbb{F}_p)$ is semi-simple in general, but we can show that $P_{L^+G}(\Gr, \mathbb{F}_p)$ is not semi-simple when $G = \GL_2$. To begin we have the following observation, valid for any group $G$.

\begin{lem} \label{extcomp}
Let $\mu$, $\lambda \in X_*(T)^+$ be such that $\lambda < \mu$ and $\Gr_{\leq \lambda}$ has codimension $c$ in $\Gr_{\leq \mu}$. Then
$$\Ext_{D_c^b(\Gr, \mathbb{F}_p)}^1(\IC_{\lambda}, \IC_{\mu}) = H_{\textnormal{\'{e}t}}^c(\Gr_{\leq \mu} - \Gr_{\leq \lambda}, \mathbb{F}_p).$$
\end{lem}

\begin{proof}
Let $i \colon \Gr_{\leq \lambda} \rightarrow \Gr_{\leq \mu}$ and $j \colon \Gr_{\leq \mu} - \Gr_{\leq \lambda} \rightarrow \Gr_{\leq \mu}$ be the inclusions . By adjunction
$$\Ext^1(\IC_\lambda, \IC_\mu) =  R^0\Gamma(Ri^! \mathbb{F}_p[c+1]).$$ As $H^i(\Gr_{\leq \mu}, \mathbb{F}_p) = 0$ for $i > 0$ by Theorem \ref{cohvan}, then by considering the exact triangle $$Ri_* Ri^! \mathbb{F}_p[0] \rightarrow \mathbb{F}_p[0] \rightarrow Rj_* \mathbb{F}_p[0]$$ we get
$$R^0\Gamma(Ri^! \mathbb{F}_p[c+1]) = H^{c}_{\text{\'{e}t}}(\Gr_{\leq \mu} - \Gr_{\leq \lambda}, \mathbb{F}_p).$$ 
\end{proof}

Now let $R$ be a $k$-algebra and let $f \colon \Gr_{\leq \mu} \times \Spec(R) \rightarrow \Gr_{\leq \mu}$ be the projection. As $c>1$ a similar computation as in the proof of Lemma \ref{extcomp} shows
\begin{equation} \label{extcompeq} \Ext^1_{D_c^b(\Gr_{\leq \mu} \times \Spec(R), \mathbb{F}_p)}(Rf^* \IC_\lambda, Rf^* \IC_{\mu}) = H^c_{\text{\'{e}t}}((\Gr_{\leq \mu} - \Gr_{\leq \lambda}) \times \Spec(R), \mathbb{F}_p).
\end{equation}

\begin{prop} \label{ssexample}
Let $G = \GL_2$ and $\mu = (1,-1)$ as in Example \ref{example1}. Then $$\Ext^1_{P_{L^+G}(\Gr, \mathbb{F}_p)}(\IC_0, \IC_\mu) \neq 0.$$
\end{prop}

\begin{proof}
Let $n \geq 0$ be an integer so that $L^+G$ acts on $\Gr_{\leq \mu}$ through the quotient $L^nG$. By Proposition \ref{equivextprop} and (\ref{extcompeq}), it suffices to exhibit a nonzero class $s \in H^2_{\text{\'{e}t}}(\Gr_{\mu}, \mathbb{F}_p)$ whose image in $H^2_{\text{\'{e}t}}(\Gr_\mu \times \Spec(R), \mathbb{F}_p)$ is fixed by $L^nG(R)$ for every $k$-algebra $R$. We first exhibit a subspace of $H^1(\Gr_\mu, \mathcal{O}_{\Gr_\mu})$ whose image in $$H^1(\Gr_\mu \times \Spec(R), \mathcal{O}_{\Gr_\mu \times \Spec(R)}) = H^1(\Gr_\mu, \mathcal{O}_{\Gr_\mu}) \otimes_k R$$ is fixed pointwise by $L^nG(R)$ for every $R$.  As $H^1(\Gr_\mu, \mathcal{O}_{\Gr_\mu})$ is an algebraic representation of $L^nG$, it suffices to exhibit a trivial subrepresentation of $H^1(\Gr_\mu, \mathcal{O}_{\Gr_\mu})$.

Let $\mu_1 = (1,0)$ and $\mu_2 = (0,-1)$. Then the convolution map $m \colon \Gr_{\leq \mu_\bullet} \ \rightarrow \Gr_{\leq \mu}$ is an isomorphism over $\Gr_\mu$. Let $f \colon \Gr_\mu \rightarrow \Gr_{\mu_{1}}$ be the restriction of the projection map $\Gr_{\leq \mu_\bullet}  \rightarrow \Gr_{\mu_1}$ to $\Gr_\mu$. The space $\Gr_\mu$ is naturally isomorphic to $\underline{\Spec}_{\Gr_{\mu_1}}(\Sym \mathcal{L})$ for an invertible sheaf $\mathcal{L}$ on $\Gr_{\mu_1}$ via the map $f$. Thus as vector spaces
$$H^1(\Gr_\mu, \mathcal{O}_{\Gr_\mu}) \cong \bigoplus_{m \geq 0} H^1(\Gr_{\mu_1}, \mathcal{L}^{m}).$$

We claim that $H^1(\Gr_{\mu_1}, \mathcal{L})$ is a representation of $L^nG$. To prove this, let $g \in L^nG(k)$, which induces an automorphism $\phi_g \colon \Gr_{\mu} \rightarrow \Gr_{\mu}$. The action of $g$ on $H^1(\Gr_\mu, \mathcal{O}_{\Gr_\mu})$ is induced by an automorphism of $f_* \mathcal{O}_{\Gr_\mu} = \bigoplus_{m \geq 0} \mathcal{L}^{m}.$ There exist abstract isomorphisms $\Gr_{\mu_1} \cong \mathbb{P}^1$ and $\mathcal{L} \cong \mathcal{O}(-2)$ (see \cite[7.2]{ngores}). Thus $\Hom(\mathcal{L}, \mathcal{L}^{m}) = 0$ for $m \geq 2$. As $H^1(\Gr_{\mu_1}, \mathcal{O}_{\Gr_{\mu_1}}) = 0$, the claim follows.  

It also follows that $H^1(\Gr_{\mu_1}, \mathcal{L})$ is one-dimensional over $k$. To prove it is the trivial representation of $L^nG$, we first note that the action of $L^nG$ on $\Gr_{\mu}$ factors through $L^nG^{\text{ad}} = L^n\PGL_2$. This is an extension of $\PGL_2$ by a connected unipotent group, so every one-dimensional representation of $L^n \PGL_2$ is trivial.

Since $\Gr_{\mu}$ has an open cover by two affines then $H^2(\Gr_{\mu}, \mathcal{O}_{\Gr_\mu}) = 0$. Thus by the Artin-Schreier sequence
$$H^2_{\text{\'{e}t}}(\Gr_{\mu}, \mathbb{F}_p) = \cok(H^1(\Gr_\mu, \mathcal{O}_{\Gr_\mu})  \xrightarrow{F-\id} H^1(\Gr_\mu, \mathcal{O}_{\Gr_\mu}) ).$$ By construction the image of $H^1(\Gr_{\mu_1}, \mathcal{L}) \subset H^1(\Gr_\mu, \mathcal{O}_{\Gr_\mu})$ in $H^2_{\text{\'{e}t}}(\Gr_{\mu}, \mathbb{F}_p)$ consists of equivariant extensions. It remains to show the image is nonzero.

We will prove the stronger statement that the composition $$H^1(\Gr_{\mu_1}, \mathcal{L}) \rightarrow H^1(\Gr_\mu, \mathcal{O}_{\Gr_\mu}) \rightarrow H^2_{\text{\'{e}t}}(\Gr_{\mu}, \mathbb{F}_p)$$ is injective. First, note that $F$ acts injectively on $H^1(\Gr_\mu, \mathcal{O}_{\Gr_\mu})$, as can be computed directly from our explicit description of $f_* \mathcal{O}_{\Gr_\mu} = \bigoplus_{m \geq 0} \mathcal{L}^{m}.$ Note also that $H^1(\Gr_\mu, \mathcal{O}_{\Gr_\mu})$ is graded in positive degrees, and $F$ sends elements of degree $m$ to elements of degree $mp$. It follows that the elements of degree $1$ have trivial intersection with the image of $F-\id$. 
\end{proof}

\section{Commutativity constraint} \label{AssocSection} 
In this section we use the Beilinson-Drinfeld Grassmannians to construct a commutativity constraint on $P_{L^+G}(\Gr, \mathbb{F}_p)$. We assume $p \nmid |\pi_1(G_{\der})|$ until Remark \ref{rmrkh}. Fix a smooth curve $X$ over $k$, and let $I = \{1, \ldots, n\}$ be an ordered finite index set. For a $k$-algebra $R$, let $X_R : = X \times \Spec(R)$. If $x \colon \Spec(R) \rightarrow X$ is a morphism, let $\Gamma_x \subset X_R$ be the graph of $x$. Fix a trivial $G$-torsor $\mathcal{F}_0$ on $X$. We will also denote by $\mathcal{F}_0$ the base change of $\mathcal{F}_0$ to $X_R$ for any $k$-algebra $R$. The Beilinson-Drinfeld Grassmannian $\Gr_{X^n}$ over $X^n$ is defined by the functor
$$\Gr_{X^n} \colon R \mapsto (x, \mathcal{F}, \beta) $$ where
$$
   \left\{
     \begin{array}{lr}
       x=(x_1, \ldots, x_n) \in X^n(R), \\
       \mathcal{F} \text{ is a } G\text{-torsor on } X_R, \\
       \beta \colon \restr{ \mathcal{F}}{X_R \setminus \cup_i \Gamma_{x_i}} \xrightarrow{\sim} \restr{ \mathcal{F}_0}{X_R \setminus \cup_i \Gamma_{x_i}}.
     \end{array}
   \right.
$$
The functor $\Gr_{X^n}$ is represented by an ind-scheme of ind-finite type over $X^n$. The fiber of $\Gr_{X^n}$ over $(x_1, \ldots, x_{n}) \in X^I(k)$ is $\prod_{i=1}^m \Gr$ where $\{y_1, \ldots, y_m\} = \{x_1, \ldots, x_{n}\}$ with the $y_i$ distinct (see \cite[3.1.13]{ZhuGra}). This relationship between $\Gr_{X^n}$ and our definition of $\Gr$ is a consequence of a theorem of Beauville and Laszlo  \cite{BeauvilleLaszlo}. 

For $I = \{1,2\}$ there is a global version of the convolution diagram  (\ref{convdiagram}) whose morphisms are defined over $X^2$. 

\begin{equation} \label{globalconv}
\Gr_{X} \times\Gr_X \xleftarrow{p_I} \tilde{L}G_{X^2} \xrightarrow{q_I} \tilde{\Gr}_{X^2} \xrightarrow{m_I} \Gr_{X^2}.
\end{equation}

Before we give the functor of points of $\tilde{L}G_{X^2}$, we need some notation. If $x \colon \Spec(R) \rightarrow X$ is a morphism, let $\hat{\mathcal{O}}_{X,x}$ be the algebra underlying the formal completion of $X_R$ along $\Gamma_{x}$, and let $\hat{D}_x = \Spec(\hat{\mathcal{O}}_{X,x})$. Let $\hat{D}^{\circ}_x = \hat{D}_x \setminus \Gamma_x$. Then $\tilde{L}G_{X^2}$ is the functor
$$\tilde{L}G_{X^2} \colon R \mapsto (x, \mathcal{F}_i, \beta_i, \sigma)$$ where
$$
   \left\{
     \begin{array}{lr}
       x=(x_1, x_2) \in X^2(R), \\
       \mathcal{F}_i \text{ is a } G\text{-torsor on } X_R \text{ for } i=1, 2\\
       \beta_i \colon \restr{ \mathcal{F}_i}{X_R \setminus  \Gamma_{x_i}} \xrightarrow{\sim} \restr{ \mathcal{F}_0}{X_R \setminus \Gamma_{x_i}} \text{ for } i=1, 2 \\
       \sigma \colon \restr{\mathcal{F}_0}{\hat{D}_{x_2}} \xrightarrow{\sim} \restr{\mathcal{F}_{1}}{\hat{D}_{x_{2}}}. \\
       
     \end{array}
   \right.
$$
The map $p_I$ forgets $\sigma$. The convolution Grassmannian $\tilde{\Gr}_{X^2}$ is  the functor
$$\tilde{\Gr}_{X^2} \colon R \mapsto (x, \mathcal{F}_i, \beta_i)$$
where
$$
   \left\{
     \begin{array}{lr}
       x=(x_1, x_2) \in X^2(R), \\
       \mathcal{F}_i \text{ is a } G\text{-torsor on } X_R \text{ for } i=1, 2\\
       \beta_i \colon \restr{ \mathcal{F}_i}{X_R \setminus  \Gamma_{x_i}} \xrightarrow{\sim} \restr{ \mathcal{F}_{i-1}}{X_R \setminus \Gamma_{x_i}} \text{ for } i=1, 2.\\
     \end{array}
   \right.
$$ 
The map $q_I$ sends $(x, \mathcal{F}_i, \beta_i, \sigma)$ to $(x, \mathcal{F}_i', \beta_i')$ where $\mathcal{F}_1' = \mathcal{F}_1$ and the torsor $\mathcal{F}_2'$ is obtained by gluing $\restr{\mathcal{F}_2}{X_R \setminus \Gamma_{x_2}}$ to $\restr{\mathcal{F}_{1}'}{\hat{D}_{x_2}}$ along $\restr{\sigma}{\hat{D}^{\circ}_{x_2}} \circ \restr{\beta_2}{\hat{D}^{\circ}_{x_2}}$.
The convolution Grassmannian $\tilde{\Gr}_{X^2}$ is an ind-scheme of ind-finite type over $X^2$. The ind-scheme $\tilde{L}G_{X^2}$ is not of ind-finite type, but we will ignore this issue as we did for $LG$. See \cite[3.21]{RicharzGS} for how to overcome this technical issue. The convolution morphism $m_I$ sends $(x, \mathcal{F}_i, \beta_i)$ to $(x, \mathcal{F}_2, \restr{\beta_1}{X_R \setminus \cup_i \Gamma_{x_i}} \circ  \restr{\beta_2}{X_R \setminus \cup_i \Gamma_{x_i}})$.  

There is also a global version $L^+G_{X^n}$ of the positive loop group functor, which is represented by an ind-group scheme of ind-finite type over $X^n$ (in \cite{RicharzGS} this functor is denoted by $\mathcal{L}^+G_I$). The group $L^+G_{X^2}$ acts on the left on $\tilde{L}G_{X^2}, \tilde{\Gr}_{X^2}$, and $\Gr_{X^2}$. It also acts on $\Gr_X \times \Gr_X$ via a map $L^+G_{X^2} \rightarrow (L^+G_X)^2$. The morphisms in the convolution diagram (\ref{globalconv}) are $L^+G_{X^2}$-equivariant. We refer the reader to \cite{RicharzGS} for more details on all of these functors.

For the rest of this section we specialize to the case $X=\mathbb{A}^1$, and we set $I=\{1,2\}$. Then there are natural isomorphisms $\Gr_X = X \times \Gr$ and $L^+G_X = X \times L^+G$. Let $$p \colon \Gr_X \rightarrow \Gr$$ be the projection.  The ind-scheme structures on $\tilde{\Gr}_{X^2}$ and $\Gr_{X^2}$ are both parametrized by pairs $\mu_1$, $\mu_2 \in X_*(T)^+$. We denote the corresponding reduced closed subschemes of $\tilde{\Gr}_{X^2}$ and $\Gr_{X^2}$ by $\tilde{\Gr}_{X^2}^{\mu_\bullet}$ and $\Gr_{X^2}^{\mu_\bullet}$, respectively. The scheme $\tilde{\Gr}_{X^2}^{\mu_\bullet}$ is obtained by taking the reduced subscheme structure on 
$$q_I (p_I^{-1} (p^{-1}(\Gr_{\leq \mu_1}) \times p^{-1}(\Gr_{\leq \mu_2})) ) \subset \tilde{\Gr}_{X^2}.$$ The scheme ${\Gr}_{X^2}^{\mu_\bullet}$ is the image of $\tilde{\Gr}_{X^2}^{\mu_\bullet}$ under $m_I$ with the reduced subscheme structure. The map $m_I \colon \tilde{\Gr}_{X^2}^{\mu_\bullet} \rightarrow \Gr_{X^2}^{\mu_\bullet}$ is proper and birational. Our discussion in Remark \ref{indequiv} applies, and it makes sense to consider the categories
$$P_{L^+G_{X^2}}(\Gr_{X^2}, \mathbb{F}_p), \quad \quad \quad \quad P_{L^+G_{X^2}}(\tilde{\Gr}_{X^2}, \mathbb{F}_p).$$

\begin{rmrk} \label{fiberrmrk}

Let $U \subset X^2$ be the locus where the coordinates are pairwise distinct, and let $\Delta \colon X \rightarrow X^2$ be the diagonal embedding. Then we have a diagram with Cartesian squares
$$ \xymatrix{
\Gr_X \ar[r]^{i_I} \ar[d]  & \Gr_{X^2} \ar[d] & \restr{\Gr_X^2}{U} \ar[l]_{j_I} \ar[d] \\
X \ar[r]^{\Delta} & X^2 & U \ar[l]
} $$
By the methods introduced by Zhu in \cite[1.2.4]{ZhuAffine}, one can show that for $\mu_1$, $\mu_2 \in X_*(T)^+$ the fiber of $ \Gr_{X^2}^{\mu_\bullet} \rightarrow X^2$ over $\Delta$ is $X \times \Gr_{\leq \mu_1 + \mu_2}$. In particular this fiber is reduced, and hence the Cohen-Macaulayness of $\Gr_{X^2}^{\mu_\bullet} $ follows from that of $\Gr_{\leq \mu_1}$, $\Gr_{\leq \mu_2}$, and $\Gr_{\leq \mu_1 + \mu_2}$. See also \cite[3.1.14]{ZhuGra}.
\end{rmrk}
 
Finally, we observe that the group $\mathbb{G}_a = X$ also acts on $\Gr_X = X \times \Gr$. More generally, the group $\mathbb{G}_a$ acts on $\tilde{L}G_{X^2}$, $\tilde{\Gr}_{X^2}$, and $\Gr_{X^2}$ by diagonal translation on the coordinates. The action of $\mathbb{G}_a $ on these objects is relative to the ground field $k$ instead of $X^2$, and this action also preserves the schemes $\tilde{\Gr}_{X^2}^{\mu_\bullet}$ and $\Gr_{X^2}^{\mu_\bullet}$. The morphisms $p_I$,  $q_I$, and $m_I$ in the convolution diagram (\ref{globalconv}) are equivariant for this action of $\mathbb{G}_a$. With these preliminary facts in mind we are now ready to start constructing the commutativity constraint.

\begin{lem}[cf. {\cite[{3.27 (i)}]{RicharzGS}}] \label{essentialimage}
The functor $Rp^*[1]$ induces a fully faithful functor
$$Rp^*[1] \colon P_{L^+G}(\Gr, \mathbb{F}_p) \rightarrow P_{L^+G_X}(\Gr_X, \mathbb{F}_p)$$ whose essential image consists of those $\mathbb{G}_a$-equivariant objects in $P_{L^+G_X}(\Gr_X, \mathbb{F}_p)$. 
\end{lem}

\begin{proof}
By Lemma \ref{4.2.3} the functor $Rp^*[1] \colon P_c^b(\Gr, \mathbb{F}_p) \rightarrow P_c^b(\Gr_X, \mathbb{F}_p)$ is fully faithful and has essential image consisting of $\mathbb{G}_a$-equivariant objects in $P_c^b(\Gr_X, \mathbb{F}_p)$. To see that the image of an $L^+G$-equivariant perverse sheaf is $L^+G_X$-equivariant, note that we have a diagram
$$\xymatrix{
L^+G_X \times_X \Gr_X \ar[r] \ar[d]^{\rho_X} & L^+G \times \Gr \ar[d]^{\rho} \\
\Gr_X \ar[r]^p & \Gr
}$$
Here the horizontal maps are the projections and the vertical maps are the action maps. We can get another commutative diagram by replacing the vertical maps with the projections onto $\Gr_X$ and $\Gr$. Now it is a simple diagram chase to show that the image of an $L^+G$-equivariant perverse sheaf is $L^+G_X$-equivariant. A similar diagram allows one to show that the inverse to $Rp^*[1]$ as defined in Lemma \ref{4.2.3} sends $L^+G_X$-equivariant perverse sheaves to $L^+G$-equivariant perverse sheaves.
\end{proof}

We denote the essential image of the functor $Rp^*[1]$ in Lemma \ref{essentialimage} by $P_{L^+G_X}(\Gr_X, \mathbb{F}_p)^{\mathbb{G}_a}.$ 
 
\begin{lem} \label{BDprops}
The scheme $\tilde{\Gr}_{X^2}^{\mu_\bullet}$ is integral, normal, and Cohen-Macaulay. 
\end{lem}

\begin{proof}
The $\Gr_{\leq \mu_i}$ satisfy these properties by Theorem \ref{propsofgr}. Thus the scheme $\tilde{\Gr}_{X^2}^{\mu_\bullet}$ is normal and Cohen-Macaulay because these properties are local in the smooth topology (see the global convolution diagram (\ref{globalconv})). To show that $\tilde{\Gr}_{X^2}^{\mu_\bullet}$ is integral, it suffices to show that $p_I^{-1} (p^{-1}(\Gr_{\leq \mu_1}) \times p^{-1}(\Gr_{\leq \mu_2}))$ is irreducible. This fact follows because $p^{-1}(\Gr_{\leq \mu_1}) \times p^{-1}(\Gr_{\leq \mu_2})$ is irreducible and $p_I$ is a torsor for a smooth group scheme having geometrically connected fibers over $X^2$ (we again need to make a technical modification so that everything is of finite type). 
\end{proof}

\begin{thm} \label{BDprops2}
The scheme ${\Gr}_{X^2}^{\mu_\bullet}$ is integral, $F$-rational, and has pseudo-rational singularities. 
\end{thm}

\begin{proof}
$\Gr_{X^2}^{\mu_\bullet}$ is irreducible because it is the image of the irreducible scheme $\tilde{\Gr}_{X^2}^{\mu_\bullet}$. As we are always working with reduced schemes, then ${\Gr}_{X^2}^{\mu_\bullet}$ is integral. To prove ${\Gr}_{X^2}^{\mu_\bullet}$ is $F$-rational, we first note that it suffices to show the local rings at every closed point are $F$-rational \cite[4.2]{Frationalloc}. By \emph{loc. cit.}, $F$-rationality may be proved after passing to the quotient by a regular sequence. The fiber of ${\Gr}_{X^2}^{\mu_\bullet}$ over a closed point in $X^2$ is either isomorphic to  $\Gr_{\leq \mu_1} \times \Gr_{\leq \mu_2}$ or $\Gr_{\leq \mu_1 + \mu_2}$ (see Remark \ref{fiberrmrk}). These fibers are globally $F$-regular by Theorem \ref{thm1} and Lemma \ref{prodlem1}, and in particular, $F$-rational. The same remark also implies that each of these fibers is cut out locally by a regular element on $\Gr_{X^2}^{\mu_\bullet}$, so this proves ${\Gr}_{X^2}^{\mu_\bullet}$ is $F$-rational. Finally, $F$-rational singularities are pseudo-rational by \cite[3.1]{SmithFrat}.
\end{proof}

\begin{prop} \label{globconvrat}
The global convolution morphism $m_I \colon \tilde{{\Gr}}_{X^2}^{\mu_\bullet} \rightarrow {\Gr}_{X^2}^{\mu_\bullet}$ satisfies $$Rm_{I,*}(\mathbb{F}_p[0]) \cong \mathbb{F}_p[0].$$
\end{prop}

\begin{proof}
    By adjunction there is a natural map $\phi \colon \mathbb{F}_p[0] \rightarrow Rm_{I,*}(\mathbb{F}_p[0])$. Since $m_I$ is proper then $Rm_{I,*}(\mathbb{F}_p[0])$ has constructible cohomology sheaves. In particular, it suffices to check $\phi$ is an isomorphism over closed points in $X^2$. The fiber of the convolution morphism $m_I$ over a closed point in the diagonal in $X^2$ is a convolution morphism for $\Gr$. By Proposition \ref{convrat} and the Artin-Schreier sequence, the convolution morphisms for $\Gr$ preserve the constant sheaf. As $m_I$ is an isomorphism over $U$ then we are done. 
\end{proof}

We now construct a convolution product on $P_{L^+G_X}(\Gr_X, \mathbb{F}_p)^{\mathbb{G}_a}.$ Recall the notation from the global convolution diagram (\ref{globalconv}).  

\begin{lem} \label{convlemma2}
Let $\mathcal{F}^\bullet_1$, $\mathcal{F}^\bullet_2 \in P_{L^+G_X}(\Gr_X, \mathbb{F}_p)^{\mathbb{G}_a}$. Then there is a unique perverse sheaf $$\mathcal{F}^\bullet_1 \overset{\sim}{\boxtimes}\mathcal{F}^\bullet_2 \in P_{L^+G_{X^2}}(\tilde{\Gr}_{X^2}, \mathbb{F}_p)$$ such that
$$p_I^* (\mathcal{F}^\bullet_1 {\boxtimes}\mathcal{F}^\bullet_2) \cong q_I^*(\mathcal{F}^\bullet_1 \overset{\sim}{\boxtimes}\mathcal{F}^\bullet_2 ).$$
The perverse sheaf $\mathcal{F}^\bullet_1 \overset{\sim}{\boxtimes}\mathcal{F}^\bullet_2$ is $\mathbb{G}_a$-equivariant.
\end{lem}

\begin{proof}
The proof is analogous to that of Lemma \ref{twistedprod}, and makes use of Lemmas \ref{essentialimage} and \ref{BDprops}. The descent steps and the proof that $\mathcal{F}^\bullet_1 \overset{\sim}{\boxtimes}\mathcal{F}^\bullet_2$ is $L^+G_{X^2}$-equivariant are the same as in the case of $\overline{\mathbb{Q}}_{\ell}$-coefficients (see \cite[3.20]{RicharzGS}). 
\end{proof}

\begin{lem} \label{convpervBD}
Let $\mathcal{F}^\bullet_1$, $\mathcal{F}^\bullet_2 \in P_{L^+G_X}(\Gr_X, \mathbb{F}_p)^{\mathbb{G}_a}$. Then
$$Rm_{I,*} (\mathcal{F}^\bullet_1 \overset{\sim}{\boxtimes}\mathcal{F}^\bullet_2 )  \in P_{L^+G_{X^2}}(\Gr_{X^2}, \mathbb{F}_p)^{\mathbb{G}_a}.$$  
\end{lem}

\begin{proof}
We first suppose each $\mathcal{F}_i^\bullet$ is simple and corresponds to the dominant cocharacter $\mu_i$. Then because $\tilde{\Gr}_{X^2}^{\mu_\bullet}$ is Cohen-Macaulay it follows from the definitions that $ \mathcal{F}^\bullet_1 \overset{\sim}{\boxtimes}\mathcal{F}^\bullet_2$ is the constant sheaf supported on $\tilde{\Gr}_{X^2}^{\mu_\bullet}$ (shifted by the dimension). As $\Gr_{X^2}^{\mu_\bullet}$ is integral and Cohen-Macaulay (see Remark \ref{fiberrmrk}), and $Rm_{I,*}$ preserves the constant sheaf, then we are done in this case. 

The fact that $Rm_{I,*} ( \mathcal{F}^\bullet_1 \overset{\sim}{\boxtimes}\mathcal{F}^\bullet_2 )$ is perverse for general $\mathcal{F}_i^\bullet$ now follows from an induction on the lengths of these objects as in the proof of Corollary \ref{convperv}. Finally, the perverse sheaf $Rm_{I,*} (\mathcal{F}^\bullet_1 \overset{\sim}{\boxtimes}\mathcal{F}^\bullet_2 )$ is $L^+G_{X^2}$-equivariant and $\mathbb{G}_a$-equivariant because $m_I$ is $L^+G_{X^2}$-equivariant and $\mathbb{G}_a$-equivariant.
\end{proof}

The next lemma is the key result which will allow us to construct the commutativity constraint (cf. \cite[5.10]{GeometricSatake} and \cite[3.22]{RicharzGS}).  Recall the notation in Remark \ref{fiberrmrk}. 
 
\begin{lem} \label{mainlemconv}
Let $\mathcal{F}^\bullet_1$, $\mathcal{F}^\bullet_2 \in P_{L^+G_X}(\Gr_X, \mathbb{F}_p)^{\mathbb{G}_a}$. Then
$$Rm_{I,*} ( \mathcal{F}^\bullet_1 \overset{\sim}{\boxtimes}\mathcal{F}^\bullet_2 ) \cong j_{I, !*} ( \mathcal{F}^\bullet_1 \boxtimes \restr{\mathcal{F}^\bullet_2}{U} ).$$ 
\end{lem}

\begin{proof}
First consider the case in which the $\mathcal{F}_i^\bullet$ are simple and correspond to dominant cocharacters $\mu_i$. Then by Proposition \ref{globconvrat}, $Rm_{I,*} ( \mathcal{F}^\bullet_1 \overset{\sim}{\boxtimes}\mathcal{F}^\bullet_2 )$ is the constant sheaf $\mathbb{F}_p$ on $\Gr_{X^2}^{\mu_\bullet}$ shifted by the dimension. Similarly $ \mathcal{F}^\bullet_1 \boxtimes \restr{\mathcal{F}^\bullet_2}{U}$ is the constant sheaf $\mathbb{F}_p$ on $\restr{\Gr_{X^2}^{\mu_\bullet}}{U}$ shifted by the dimension. As $\Gr_{X^2}^{\mu_\bullet}$ is integral and $F$-rational by Theorem \ref{BDprops2} then the isomorphism $Rm_{I,*} ( \mathcal{F}^\bullet_1 \overset{\sim}{\boxtimes}\mathcal{F}^\bullet_2 ) \cong j_{I, !*} ( \mathcal{F}^\bullet_1 \boxtimes \restr{\mathcal{F}^\bullet_2}{U} )$ follows from Theorem \ref{mainthm2}. 

The general case now follows by induction on the lengths of the perverse sheaves $\mathcal{F}_i^\bullet$. More precisely, if $\mathcal{F}_2^\bullet$ is not simple, by Lemma \ref{essentialimage} we may choose an exact sequence
$$0 \rightarrow \mathcal{G}_1^\bullet \rightarrow \mathcal{F}_2^\bullet \rightarrow \mathcal{G}_2^\bullet \rightarrow 0$$ where each $\mathcal{G}_i^\bullet  \in P_{L^+G_X}(\Gr_X, \mathbb{F}_p)^{\mathbb{G}_a}$ is nonzero. As in the proof of Corollary \ref{convperv}, we obtain an exact sequence
$$0 \rightarrow Rm_{I,*} ( \mathcal{F}^\bullet_1 \overset{\sim}{\boxtimes}\mathcal{G}^\bullet_1 ) \rightarrow Rm_{I,*} ( \mathcal{F}^\bullet_1 \overset{\sim}{\boxtimes}\mathcal{F}^\bullet_2 ) \rightarrow Rm_{I,*} ( \mathcal{F}^\bullet_1 \overset{\sim}{\boxtimes}\mathcal{G}^\bullet_2 ) \rightarrow 0.$$

Note that $Rm_{I,*} ( \mathcal{F}^\bullet_1 \overset{\sim}{\boxtimes}\mathcal{F}^\bullet_2 )$ is an extension of $\mathcal{F}^\bullet_1 \boxtimes \restr{\mathcal{F}^\bullet_2}{U}$ because $m_I$ is an isomorphism over $U$. Now we make a general observation. Suppose $$0 \rightarrow \mathcal{F}^\bullet \rightarrow \mathcal{G}^\bullet \rightarrow \mathcal{H}^\bullet \rightarrow 0$$ is an exact sequence in $P_c^b(\Gr_{X^2}^{\mu_\bullet}, \mathbb{F}_p)$. Then if $\mathcal{F}^\bullet \cong j_{I,!*} ({^p}j_I^*\mathcal{F}^\bullet) $ and $\mathcal{H}^\bullet \cong j_{I,!*} ({^p}j_I^*\mathcal{H}^\bullet)$, we also have $\mathcal{G}^\bullet \cong j_{I,!*} ({^p}j_I^*\mathcal{G}^\bullet)$. This follows from the characterization of intermediate extensions in Lemma \ref{intext}. Hence by the induction hypothesis and Lemma \ref{homintext} we obtain a canonical isomorphism $Rm_{I,*} ( \mathcal{F}^\bullet_1 \overset{\sim}{\boxtimes}\mathcal{F}^\bullet_2 ) \cong j_{I, !*} ( \mathcal{F}^\bullet_1 \boxtimes \restr{\mathcal{F}^\bullet_2}{U} )$. Now we conclude by a similar induction on the length of $\mathcal{F}_1^\bullet$. 
\end{proof}

\begin{prop} \label{constprop}
Let $\mathcal{F}^\bullet_1$, $\mathcal{F}^\bullet_2 \in P_{L^+G_X}(\Gr_X, \mathbb{F}_p)^{\mathbb{G}_a}$, and let $f_I \colon \Gr_{X^2} \rightarrow X^2$ be the structure map. Then $Rf_{I,*}(j_{I,!*}(\mathcal{F}^\bullet_1 \boxtimes \restr{\mathcal{F}^\bullet_2}{U})) \in D_c^b(X^2, \mathbb{F}_p)$ has constant cohomology sheaves.
\end{prop}

\begin{proof}
We first show that the cohomology sheaves of $Rf_{I,*}(j_{I,!*}(\mathcal{F}^\bullet_1 \boxtimes \restr{\mathcal{F}^\bullet_2}{U}))$ are local systems. To begin, note that $Rf_{I,*}(j_{I,!*}(\mathcal{F}^\bullet_1 \boxtimes \restr{\mathcal{F}^\bullet_2}{U})) \in D_c^b(X^2, \mathbb{F}_p)$ because $f_I$ is proper. Next, note that when the $\mathcal{F}_i^\bullet$ are simple and correspond to dominant cocharacters $\mu_i$ the complex $j_{I,!*}(\mathcal{F}^\bullet_1 \boxtimes \restr{\mathcal{F}^\bullet_2}{U})$ is the constant sheaf $\mathbb{F}_p$ supported on $\Gr_{X^2}^{\mu_\bullet}$ shifted by the dimension. By Remark \ref{fiberrmrk} the fibers of $f_I \colon \Gr_{X^2}^{\mu_\bullet} \rightarrow X^2$ over closed points in $X^2$ are isomorphic to $\Gr_{\leq \mu_1} \times \Gr_{\leq \mu_2}$ or $\Gr_{\leq \mu_1 + \mu_2}$. Thus, by Theorem \ref{cohvan} and the K\"unneth formula we see that the natural morphism
$$\mathbb{F}_p[\dim \Gr_{X^2}^{\mu_\bullet}] \rightarrow Rf_{I,*} \mathbb{F}_p[\dim \Gr_{X^2}^{\mu_\bullet}]$$ is an isomorphism. Hence $Rf_{I,*}(j_{I,!*}(\mathcal{F}^\bullet_1 \boxtimes \restr{\mathcal{F}^\bullet_2}{U}))$ has constant cohomology sheaves when the $\mathcal{F}_i^\bullet$ are simple. Now the claim for general $\mathcal{F}_i^\bullet$ follows by induction on the lengths and the fact that the subcategory of $D_c^b(X^2, \mathbb{F}_p)$ consisting of objects whose cohomology sheaves are locally constant is a triangulated subcategory.

Now let $f^I \colon \Gr_X^2 \rightarrow X^2$ be the product of the projection maps $\Gr_X \rightarrow X$, and let $f \colon \Gr_X \rightarrow X$ be the case in which $I$ is a singleton. Observe that we have an isomorphism $$\restr{Rf_{I,*}(j_{I,!*}(\mathcal{F}^\bullet_1 \boxtimes \restr{\mathcal{F}^\bullet_2}{U}))}{U} \cong \restr{Rf^I_* (\mathcal{F}^\bullet_1 \boxtimes \mathcal{F}^\bullet_2)}{U}.$$ By the K\"unneth formula
$$\restr{Rf^I_* (\mathcal{F}^\bullet_1 \boxtimes \mathcal{F}^\bullet_2)}{U} \cong Rf_*(\mathcal{F}^\bullet_1) \boxtimes \restr{Rf_*(\mathcal{F}^\bullet_1)}{U}.$$ We claim that each $Rf_*(\mathcal{F}^\bullet_i)$ has constant cohomology sheaves. To prove this, consider the Cartesian diagram
$$\xymatrix{
\Gr_X \ar[r]^{p} \ar[d]^f & \Gr \ar[d] \\
X \ar[r] & \Spec(k)
}$$ By Lemma \ref{essentialimage} we can write $\mathcal{F}_i^\bullet \cong Rp^*[1] \mathcal{G}_i^\bullet$ for some $\mathcal{G}_i^\bullet \in P_{L^+G}(\Gr, \mathbb{F}_p)$. Thus by the proper base change theorem $Rf_*(\mathcal{F}^\bullet_i)$ has constant cohomology sheaves. Hence $Rf_* (\mathcal{F}^\bullet_1) \boxtimes Rf_* (\mathcal{F}^\bullet_2)$ also has constant cohomology sheaves, and by the above so does $\restr{Rf_{I,*}(j_{I,!*}(\mathcal{F}^\bullet_1 \boxtimes \restr{\mathcal{F}^\bullet_2}{U}))}{U}$. 

Let $\overline{x}$ be a geometric point of $U$. Each $H^m(Rf_{I,*}(j_{I,!*}(\mathcal{F}^\bullet_1 \boxtimes \restr{\mathcal{F}^\bullet_2}{U})))$ corresponds to a representation of $\pi_{1,\text{\'{e}t}}(X^2, \overline{x})$. As $X^2$ is integral and normal the map $\pi_{1,\text{\'{e}t}}(U, \overline{x}) \rightarrow \pi_{1,\text{\'{e}t}}(X^2, \overline{x})$ is surjective \cite[0BQI]{stacks-project}. Since we already know $\restr{H^m(Rf_{I,*}(j_{I,!*}(\mathcal{F}^\bullet_1 \boxtimes \restr{\mathcal{F}^\bullet_2}{U})))}{U}$ is constant, it follows that  $H^m(Rf_{I,*}(j_{I,!*}(\mathcal{F}^\bullet_1 \boxtimes \restr{\mathcal{F}^\bullet_2}{U})))$ is constant as well. 
\end{proof}

\begin{prop} \label{stalkprop}
Let $\mathcal{F}^\bullet_1$, $\mathcal{F}_2^\bullet \in P_{L^+G}(\Gr, \mathbb{F}_p)$, and let $f_I \colon \Gr_{X^2} \rightarrow X^2$ be the structure map. Then there are canonical identifications
\begin{enumerate}[{\normalfont (i)}]
\item $\restr{H^{m-2}(Rf_{I,*}(j_{I,!*}(\restr{Rp^*[1] \mathcal{F}_1^\bullet \boxtimes Rp^*[1] \mathcal{F}_2^\bullet}{U})))}{x} \cong \bigoplus_{i+j=m} R^i\Gamma(\mathcal{F}_1^\bullet) \otimes R^j\Gamma(\mathcal{F}_2^\bullet)$ for $x \in U(k)$.
\item $\restr{H^{m-2}(Rf_{I,*}(j_{I,!*}(\restr{Rp^*[1] \mathcal{F}_1^\bullet \boxtimes Rp^*[1] \mathcal{F}_2^\bullet}{U})))}{x} \cong R^m \Gamma(\mathcal{F}_1^\bullet * \mathcal{F}_2^\bullet)$ for $x \in X^2(k)$ in the image of the diagonal $\Delta \colon X(k) \rightarrow X^2(k)$.
\end{enumerate}
\end{prop}

\begin{proof}
If $x \in U(k)$, note that by the proper base change theorem
$$\restr{Rf_{I,*}(j_{I,!*}(\restr{Rp^*[1] \mathcal{F}_1^\bullet \boxtimes Rp^*[1] \mathcal{F}_2^\bullet}{U}))}{x} \cong R\Gamma(\mathcal{F}_1^\bullet[1] \boxtimes \mathcal{F}_2^\bullet[1]).$$  By the K\"unneth formula,
$$R\Gamma(\mathcal{F}_1^\bullet[1] \boxtimes \mathcal{F}_2^\bullet[1]) \cong R\Gamma(\mathcal{F}_1^\bullet[1]) \otimes^L_{\mathbb{F}_p} R\Gamma(\mathcal{F}_2^\bullet[1]).$$ Now by taking cohomology we get the first isomorphism.

Now suppose $x$ is in the diagonal and note that by Lemma \ref{mainlemconv},
$$j_{I,!*}(Rp^*[1] \mathcal{F}_1^\bullet \boxtimes Rp^*[1] \mathcal{F}_2^\bullet \mid_{U}) \cong Rm_{I,*} (Rp^*[1] \mathcal{F}_1^\bullet \overset{\sim}{\boxtimes} Rp^*[1] \mathcal{F}_2^\bullet).$$ There is a diagram with Cartesian squares
$$\xymatrix{
LG \times^{L^+G} \Gr \ar[r]^(.65){m} \ar[d]^{i} & \Gr \ar[r] \ar[d] & \Spec(k) \ar[d]^x \\
\tilde{\Gr}_{X^2} \ar[r]^{m_I} & \Gr_{X^2} \ar[r]^{f_I} & X^2
}$$
Define $$\mathcal{F}^\bullet := Ri^*[-2] ( Rp^*[1] \mathcal{F}_1^\bullet \overset{\sim}{\boxtimes} Rp^*[1] \mathcal{F}_2^\bullet) \in D_c^b(LG \times^{L^+G} \Gr, \mathbb{F}_p).$$
The isomorphism in (ii) will follow from the proper base change theorem if we can produce a canonical isomorphism $$\mathcal{F}^\bullet \cong \mathcal{F}_1^\bullet \overset{\sim}{\boxtimes}  \mathcal{F}_2^\bullet.$$

From the definitions and the previous diagram it follows immediately that 
$$Rq^*(\mathcal{F}^\bullet) \cong Rp^*(\mathcal{F}_1^\bullet \boxtimes \mathcal{F}_2^\bullet).$$ Moreover, $\mathcal{F}^\bullet$ is perverse because this property is local in the smooth topology. Thus there is a natural isomorphism $\mathcal{F}^\bullet \cong  \mathcal{F}_1^\bullet \overset{\sim}{\boxtimes}  \mathcal{F}_2^\bullet.$
\end{proof}

\begin{thm} \label{symfun}
One can endow the category $(P_{L^+G}(\Gr, \mathbb{F}_p), *)$ with the structure of a symmetric monoidal category such that the global cohomology functor
$H \colon (P_{L^+G}(\Gr, \mathbb{F}_p), *) \rightarrow (\Vect_{\mathbb{F}_p}, \otimes) $ is a symmetric monoidal functor. 
\end{thm}

\begin{proof}
Let $\mathcal{F}^\bullet_1$, $\mathcal{F}^\bullet_2$, $\mathcal{F}^\bullet_3 \in P_{L^+G}(\Gr, \mathbb{F}_p)$. By Propositions \ref{constprop} and \ref{stalkprop} there is a canonical isomorphism $$H(\mathcal{F}^\bullet_1 * \mathcal{F}^\bullet_2) \cong H(\mathcal{F}_1^\bullet) \otimes H(\mathcal{F}_2^\bullet).$$ There is also a canonical isomorphism $$H(\IC_0) \cong \mathbb{F}_p.$$ We constructed an associativity constraint on  $(P_{L^+G}(\Gr, \mathbb{F}_p), *)$ in Theorem \ref{monoidal}. We claim that $H$ is a monoidal functor, i.e., $H$ sends the associativity constraint on $(P_{L^+G}(\Gr, \mathbb{F}_p), *)$ to the usual associativity constraint on $(\Vect_{\mathbb{F}_p}, \otimes)$. 

To prove that $H$ is a monoidal functor, we first note that Lemma \ref{convlemma2} also applies in the case $I=\{1,2,3\}$, with essentially the same proof. We can thus form the perverse sheaf
$$\mathcal{F}_{1,2,3}^\bullet : = Rp^*[1]\mathcal{F}^\bullet_1 \overset{\sim}{\boxtimes} Rp^*[1]\mathcal{F}^\bullet_2 \overset{\sim}{\boxtimes} Rp^*[1]\mathcal{F}^\bullet_3 \in P_c^b(\tilde{\Gr}_{X^3}, \mathbb{F}_p).$$ We refer the reader to \cite{RicharzGS} for a detailed description of the functors $\tilde{\Gr}_{X^3}$, $\tilde{L}G_{X^3}$, etc.

We claim that $Rf_{I,*}(Rm_{I,*}(\mathcal{F}_{1,2,3}^\bullet)) \in D_c^b(X^3, \mathbb{F}_p)$ has constant cohomology sheaves. Indeed, if the $\mathcal{F}_i^\bullet$ are simple then one can show that $Rm_{I,*}(\mathcal{F}_{1,2,3}^\bullet)$ is a shifted constant sheaf by using the same methods as in the proof of Proposition \ref{globconvrat}. Now the fact that $Rf_{I,*}(Rm_{I,*}(\mathcal{F}_{1,2,3}^\bullet))$ has constant cohomology sheaves follows by the same arguments as in Proposition \ref{constprop} (where we use the functor $Rm_{I,*}$ instead of $j_{I,!*}$). Let $U' \subset X^3$ be the locus where the coordinates are pairwise distinct. Then by arguments analogous to those in Proposition \ref{stalkprop} we have
\begin{enumerate}[{\normalfont (i)}]
\item $\restr{H^{m-3}(Rf_{I,*}(Rm_{I,*}(\mathcal{F}_{1,2,3}^\bullet)))}{x} \cong \bigoplus_{i+j+k=m} R^i\Gamma(\mathcal{F}_1^\bullet) \otimes R^j\Gamma(\mathcal{F}_2^\bullet) \otimes R^k\Gamma(\mathcal{F}_3^\bullet)$ for $x \in U'(k)$.
\item $\restr{H^{m-3}(Rf_{I,*}(Rm_{I,*}(\mathcal{F}_{1,2,3}^\bullet)))}{x} \cong R^m \Gamma(\mathcal{F}_1^\bullet * \mathcal{F}_2^\bullet * \mathcal{F}_3^\bullet)$ for $x \in X^3(k)$ in the image of the diagonal  $X(k) \rightarrow X^3(k)$.
\end{enumerate}
Now we conclude that $H$ is monoidal by noting that the associativity constraint in Theorem \ref{monoidal} arises from the associativity isomorphism
$$(Rp^*[1]\mathcal{F}^\bullet_1 \overset{\sim}{\boxtimes} Rp^*[1]\mathcal{F}^\bullet_2) \overset{\sim}{\boxtimes} Rp^*[1]\mathcal{F}^\bullet_3 \cong Rp^*[1]\mathcal{F}^\bullet_1 \overset{\sim}{\boxtimes} (Rp^*[1]\mathcal{F}^\bullet_2 \overset{\sim}{\boxtimes} Rp^*[1]\mathcal{F}^\bullet_3)$$ by restriction to the diagonal in $X^3.$

For the rest of the proof we let $I = \{1,2\}$. To construct the commutativity constraint, let $\tau \colon I \rightarrow I$ be the bijection which swaps coordinates. Then $\tau$ induces morphisms of $X^2$ and $\Gr_{X^2}$ by permuting the coordinates. We have a commutative diagram
$$\xymatrix{
\Gr_X \ar[d]_{\id} \ar[r]^{i_I} & \Gr_{X^2} \ar[d]^\tau & \restr{\Gr_{X^2}}{U} \ar[l]_{j_I} \ar[d]^\tau \\
\Gr_X \ar[r]^{i_I} & \Gr_{X^2} & \restr{\Gr_{X^2}}{U} \ar[l]_{j_I}
}$$
Observe that 
$$R\tau^* (j_{I,!*}(Rp^*[1]\mathcal{F}^\bullet_1 \boxtimes Rp^*[1]\restr{\mathcal{F}^\bullet_2}{U}) \cong  j_{I,!*}(Rp^*[1]\mathcal{F}^\bullet_2 \boxtimes Rp^*[1]\restr{\mathcal{F}^\bullet_1}{U}).$$ Combining this with commutativity of the left side of the diagram we get
\begin{align*} Ri_I^* (j_{I,!*}(Rp^*[1]\mathcal{F}^\bullet_1 \boxtimes Rp^*[1]\restr{\mathcal{F}^\bullet_2}{U})) & \cong Ri_I^* R\tau^* (j_{I,!*}(Rp^*[1]\mathcal{F}^\bullet_1 \boxtimes Rp^*[1]\restr{\mathcal{F}^\bullet_2}{U})) \\ & \cong Ri_I^* (j_{I,!*}(Rp^*[1]\mathcal{F}^\bullet_2 \boxtimes Rp^*[1]\restr{\mathcal{F}^\bullet_1}{U})).
\end{align*} By combining this with Proposition \ref{stalkprop} (ii) we get an isomorphism
$$\mathcal{F}^\bullet_1 * \mathcal{F}^\bullet_2 \cong \mathcal{F}^\bullet_2 * \mathcal{F}^\bullet_1.$$

However, we need to modify this commutativity constraint. The problem is that the commutativity constraint on $(P_{L^+G}(\Gr, \mathbb{F}_p), *)$ is constructed from the commutativity of the derived tensor product $\otimes^L$. For two chain complexes $\mathcal{C}^\bullet$, $\mathcal{D}^\bullet$ the usual isomorphism $(\mathcal{C}^\bullet \otimes \mathcal{D}^\bullet)^\bullet \cong (\mathcal{D}^\bullet \otimes \mathcal{C}^\bullet)^\bullet$ sends
$$c_i \otimes d_j \mapsto (-1)^{ij} d_j \otimes c_i$$ where $c_i$ has degree $i$ and $d_j$ has degree $j$ (\cite[0BYI]{stacks-project}). Thus, as it stands now the functor $H$ sends the commutativity constraint on  $(P_{L^+G}(\Gr, \mathbb{F}_p), *)$ to the commutativity constraint on $(\Vect_{\mathbb{F}_p}, \otimes)$ up to some possible sign changes. 

We can remedy the situation by the same method as in \cite{GeometricSatake}, \cite{RicharzGS}. More specifically, suppose $\mathcal{F}_i^\bullet \in P_{L^+G}(\Gr, \mathbb{F}_p)$ for $i \in  \{1,2\}$ is supported on a component of $\Gr$ whose $L^+G$-orbits have dimensions of parity $p(i) \in \mathbb{Z}/2 = \{0, 1\}$. Then we modify the commutativity isomorphism $\mathcal{F}_1^\bullet * \mathcal{F}_2^\bullet \xrightarrow{\sim} \mathcal{F}_2^\bullet * \mathcal{F}_1^\bullet$ by multiplication by $(-1)^{p(1)p(2)}$. Making use of Corollary \ref{evenodd} one can check that $H$ sends this modified commutativity constraint to the usual constraint on $(\Vect_{\mathbb{F}_p}, \otimes)$.

To show that $(P_{L^+G}(\Gr, \mathbb{F}_p), *)$ is a symmetric monoidal category, we need to verify certain coherence conditions such as the hexagon axiom. By Theorem \ref{exactfaith}, $H$ is faithful, so these conditions may be checked after applying $H$. As $H$ sends our constraints to the usual constraints on $(\Vect_{\mathbb{F}_p}, \otimes)$, the coherence conditions are satisfied. Thus $(P_{L^+G}(\Gr, \mathbb{F}_p), *)$ is a symmetric monoidal category and $H$ is a symmetric monoidal functor. 
\end{proof}

We can now prove Theorem \ref{repthm}.

\begin{myproof2}[Proof of Theorem \ref{repthm}]
Combining Theorems \ref{exactfaith} and \ref{symfun} the functor $H$ is an exact, faithful, $\mathbb{F}_p$-linear symmetric monoidal functor. Moreover, the identity object $\mathds{1} \in P_{L^+G}(\Gr, \mathbb{F}_p)$ satisfies $ \End(\mathds{1})= \mathbb{F}_p $ because $\mathds{1}$ is the constant sheaf $\mathbb{F}_p[0]$ supported on a point. Now we conclude by noting that in their proof of the main result on neutral Tannakian categories (see the paragraph before Remark 2.17 in \cite{DeligneMilne}) Deligne and Milne observed that if one removes the rigidity hypothesis then one gets an equivalence with the category of representations of an affine monoid scheme. As in the rigid case, the monoid scheme $M_G$ represents the functor $\underline{\End}^{\otimes}(H)$ of endomorphisms of the fiber functor.
\end{myproof2}

\begin{rmrk} \label{rmrkh}
So far we have proved Theorems \ref{repthm}, \ref{convsimp}, and \ref{irreducibleobjects} under the assumption $p \nmid |\pi_1(G_{\der})|$. We now explain how to remove this hypothesis. 

By \cite[3.1]{MilneShih} we may let $$1 \rightarrow N \rightarrow G' \rightarrow G \rightarrow 1$$ be a central extension of split reductive groups over $k$ such that $G'_{\der}$ is simply connected and $N$ is a torus. Let $T' \subset B' \subset G'$ be the maximal torus and Borel subgroup given by the preimages of $T$ and $B$. Denote by $\Gr_{G'}$ and $\Gr_{G}$ the affine Grassmannians for $G'$ and $G$ respectively, and by $\mathcal{F}\ell_{\mathcal{B}'}$ and $\mathcal{F}\ell_{\mathcal{B}}$ the affine flag varieties. 

Since $N$ is a torus the map $\pi_1(G') \rightarrow \pi_1(G)$ is surjective. Thus by \cite[3.1]{UnivHomeo} the map $G' \rightarrow G$ induces a surjection $\phi \colon \Gr_{G'} \rightarrow \Gr_G$ such that each connected component of $\Gr_{G'}$ maps onto its image via a universal homeomorphism. This map also sends $L^+{G'}$ orbits onto $L^+G$ orbits. The analogous facts hold for $\mathcal{F}\ell_{\mathcal{B'}} \rightarrow \mathcal{F}\ell_{\mathcal{B}}$, and more generally for partial affine flag varieties.

By the topological invariance of the \'{e}tale site \cite[04DY]{stacks-project} we may therefore deduce Theorem \ref{irreducibleobjects} from the corresponding properties for $G'$. To prove Theorems \ref{repthm} and \ref{convsimp}, note that $G' \rightarrow G$ also induces a map $LG' \times^{L^+G'} \Gr_{G'} \rightarrow LG \times^{L^+G} \Gr_G$. There is a commutative diagram (see \cite[1.2.14]{ZhuGra})
$$\xymatrix{
LG' \times^{L^+G'} \Gr_{G'} \ar[r]^(.65){\sim} \ar[d] & \Gr_{G'}^2 \ar[d]^{\phi \times \phi}
\\ LG \times^{L^+G} \Gr_{G} \ar[r]^(.65){\sim} & \Gr_G^2
}$$
Thus the map $LG' \times^{L^+G'} \Gr_{G'} \rightarrow LG \times^{L^+G} \Gr_G$ is also surjective and restricts to a universal homeomorphism on connected components. 
The analogous facts hold for the $n$-fold convolution Grassmannians. 

Let $X=\mathbb{A}^1$, and let $\Gr^{G'}_{X^2}$ and $\Gr_{X^2}^{G}$ be the Beilinson-Drinfeld Grassmannians over $X^2$ for $G'$ and $G$. The map $G' \rightarrow G$ induces a map $\Gr^{G'}_{X^2} \rightarrow \Gr_{X^2}^{G}$. By Remark \ref{fiberrmrk} there is a diagram (\ref{diagramp}) with Cartesian squares from which we deduce that $\Gr^{G'}_{X^2} \rightarrow \Gr_{X^2}^{G}$ is surjective and restricts to a universal homeomorphism on connected components. The analogous facts are also true for the global convolution Grassmannians. 

\begin{equation} \label{diagramp} \xymatrix{
X \times \Gr_{G'} \ar[d]^{\id_X \times \phi} \ar[r] & \Gr^{G'}_{X^2} \ar[d] & \restr{(X \times \Gr_{G'})^2}{U} \ar[l] \ar[d]^{\restr{(\id_X \times \phi)^2}{U}}\\
X \times \Gr_{G} \ar[r] \ar[d] & \Gr^{G}_{X^2} \ar[d] & \restr{(X \times \Gr_{G})^2}{U} \ar[l] \ar[d] \\
X \ar[r]^{\Delta} & X^2 & U \ar[l]
}
\end{equation}

Now we may use the same arguments as in Sections \ref{FponGr} and \ref{AssocSection} to prove Theorems \ref{repthm} and \ref{convsimp} in the case $p \mid |\pi_1(G_{\der})|$. The idea is that any time we need to prove a complex is perverse, simple, an intermediate extension, etc., we use the topological invariance of the \'{e}tale site and appeal to the corresponding fact for $G'$.

\end{rmrk}

For the rest of this paper we make no assumption about $|\pi_1(G_{\der})|$. We conclude this section with a few results about $M_G$. 

\begin{prop} \label{whengroup}
The affine monoid scheme $M_G$ in Theorem \ref{repthm} is a group scheme if and only if $G$ is a torus. 
\end{prop}

\begin{proof}
If $M_G$ is a group scheme then for each $\mu \in X_*(T)^+$ there is some $\IC_\mu^* \in P_{L^+G}(\Gr, \mathbb{F}_p)$ which is dual to $\IC_{\mu}$. The perverse sheaf $\IC_\mu^*$ must also be simple, so $\IC_\mu^* = \IC_\lambda$ for some $\lambda \in X_*(T)^+$. We have seen that $\IC_{\mu} * \IC_{\lambda} \cong \IC_{\mu + \lambda}$, so we must have $\mu + \lambda = 0$. In particular, $X_*(T)^+$ is a group. This implies $G=T$.

Conversely, if $G=T$ then the corresponding (reduced) affine Grassmannian consists of a disjoint union of points indexed by $X_*(T)$. There are then no non-nontrivial extensions between objects of $P_{L^+T}(\Gr, \mathbb{F}_p)$, and it is straightforward to check that $$P_{L^+T}(\Gr, \mathbb{F}_p) \cong \Rep_{\mathbb{F}_p}(\hat{T}).$$ Here $\hat{T}$ is the torus defined over $\mathbb{F}_p$ with root datum dual to that of $T$. 
\end{proof}

Let $\Spec(\mathbb{F}_p[X_*(T)^+])$ be the affine monoid scheme with multiplication defined by the map
$$\mathbb{F}_p[X_*(T)^+] \rightarrow \mathbb{F}_p[X_*(T)^+] \otimes_{\mathbb{F}_p} \mathbb{F}_p[X_*(T)^+], \quad \quad \quad \quad \mu \mapsto \mu \otimes \mu, \quad \quad \mu \in X_*(T)^+.$$ 

\begin{prop}
There is a morphism of monoids $$M_G \rightarrow \Spec(\mathbb{F}_p[X_*(T)^+]).$$ This is an isomorphism if $G$ is a torus.
\end{prop}

\begin{proof}
Let $P_{L^+G}(\Gr, \mathbb{F}_p)^{\sem} \subset P_{L^+G}(\Gr, \mathbb{F}_p)$ be the full subcategory consisting of semi-simple objects. By Theorem \ref{convsimp}, $P_{L^+G}(\Gr, \mathbb{F}_p)^{\sem}$ is a tensor subcategory. The same Tannakian formalism applies to $P_{L^+G}(\Gr, \mathbb{F}_p)^{\sem}$, and we get an equivalence $$P_{L^+G}(\Gr, \mathbb{F}_p)^{\sem} \xrightarrow{\sim} \Rep_{\mathbb{F}_p}(\Spec(\mathbb{F}_p[X_*(T)^+])).$$ The inclusion $P_{L^+G}(\Gr, \mathbb{F}_p)^{\sem} \rightarrow P_{L^+G}(\Gr, \mathbb{F}_p)$ then induces the desired morphism of monoid schemes. It is an isomorphism if $G$ is a torus by the proof of Proposition \ref{whengroup}. 
\end{proof}

\begin{prop} \label{invlim} There exists an isomorphism $$M_G = \varprojlim_{i} M_{i}$$ where each $M_{i}$ is a closed submonoid of a monoid of upper triangular matrices.
\end{prop}

\begin{proof}
Let $V$ be a finite-dimensional representation of $M_G$. As every simple object of $\Rep_{\mathbb{F}_p}(M_G)$ is one-dimensional then one can pick a basis of $V$ for which the image of $M_G \rightarrow \End(V)$ consists of upper triangular matrices. Next, by a standard argument identical to that for affine group schemes, $M_G$ is the inverse limit of its finite-type quotients. Now we conclude by noting that, as for affine group schemes, every affine monoid scheme of finite type admits a faithful finite-dimensional representation. 
\end{proof}

\section{The mod $p$ Satake isomorphism} \label{SatakeSection}
In this section we prove Theorem \ref{Satthm}. Fix a local field $F$ of equal characteristic $p$ with ring of integers $\mathcal{O}$, and a uniformizing element $t \in \mathcal{O}$. Let $G$ be a split reductive group defined over the residue field $\mathbb{F}_q$ of $F$, and choose an embedding $\iota \colon \mathbb{F}_q \rightarrow \overline{\mathbb{F}}_q$ into an algebraic closure of $\mathbb{F}_q$. View $\Gr$ as an ind-scheme defined over $\overline{\mathbb{F}}_q$. For a different approach to the material in this section which uses the function-sheaf correspondence over the finite field $\mathbb{F}_q$, see \cite[\S 4]{CentralCass}.

Given $\mathcal{F}^\bullet \in P_{L^+G}(\Gr, \mathbb{F}_p)$, one can form an element $\mathcal{T}(\mathcal{F}^\bullet) \in \mathcal{H}_G$ as follows. Our choice of uniformizing element $t$ induces an isomorphism $F \cong \mathbb{F}_q(\!(t)\!)$, which then induces a map $G(F) \rightarrow LG(\mathbb{F}_q)$. As $\iota$ induces a map $LG(\mathbb{F}_q) \rightarrow LG(\overline{\mathbb{F}}_q)$, then by composing with the quotient $LG(\overline{\mathbb{F}}_q) \rightarrow \Gr( \overline{\mathbb{F}}_q)$ we get a map $$i_t \colon G(F) \rightarrow \Gr(\overline{\mathbb{F}}_q).$$

For $x \in G(F)$ we set
$$\mathcal{T}(\mathcal{F}^\bullet)(x) = \sum_{i \in \mathbb{Z}} \dim_{\mathbb{F}_p} H^i(\mathcal{F}^\bullet_{i_t(x)}).$$

As our choices also induce an embedding $G(\mathcal{O}) \rightarrow L^+G(\overline{\mathbb{F}}_q)$ and $\mathcal{F}^\bullet$ is $L^+G$-equivariant, then $\mathcal{T}(\mathcal{F}^\bullet)$ is $G(\mathcal{O})$ bi-invariant.

\begin{prop}
The function $\mathcal{T}(\mathcal{F}^\bullet) \in \mathcal{H}_G$ is independent of the choice of uniformizing element $t \in \mathcal{O}$ and the embedding $\iota \colon \mathbb{F}_q \rightarrow \overline{\mathbb{F}}_q$. The formation of $\mathcal{T}(\mathcal{F}^\bullet)$ also induces a well-defined map of $\mathbb{F}_p$-vector spaces $$\mathcal{T} \colon K_0(P_{L^+G}(\Gr, \mathbb{F}_p)) \otimes \mathbb{F}_p \rightarrow \mathcal{H}_G.$$ 
\end{prop}

\begin{proof}
Let $\mathcal{F}^\bullet_{\text{ss}}$ be a semi-simplification of $\mathcal{F}^\bullet$, i.e., the direct sum of the simple subquotients appearing in a composition series for $\mathcal{F}^\bullet$. By Theorem \ref{irreducibleobjects}, on a fixed connected component of $\Gr$ all perverse sheaves in $P_{L^+G}(\Gr, \mathbb{F}_p)$ are concentrated in degrees having the same parity. Thus $$\mathcal{T}(\mathcal{F}^\bullet)  = \mathcal{T}(\mathcal{F}^\bullet_{\text{ss}}).$$
This shows that $\mathcal{T}$ induces a map on $K_0(P_{L^+G}(\Gr, \mathbb{F}_p))$.

Different choices  of $t$ alter the embedding $i_{t}$ by an automorphism of $\Gr$ induced by multiplication by an element in $\mathbb{G}_m(\mathbb{F}_q[\![t]\!])$. Such an automorphism preserves $L^+G$-orbits, and so it also preserves the simple subquotients occurring in a composition series for $\mathcal{F}^\bullet$. Thus $\mathcal{T}(\mathcal{F}^\bullet)$ is independent of the choice of uniformizing element $t \in \mathcal{O}$. 

The map $i_t \colon G(F) \rightarrow LG(\overline{\mathbb{F}}_q)$ also induces a well-defined map $$G(\mathcal{O}) \backslash G(F) / G(\mathcal{O}) \rightarrow L^+G(\overline{\mathbb{F}}_q)  \backslash LG(\overline{\mathbb{F}}_q) / L^+G(\overline{\mathbb{F}}_q).$$ By the Cartan decomposition this is a bijection and it is independent of the choice of embedding $\iota \colon \mathbb{F}_q \rightarrow \overline{\mathbb{F}}_q$. 
\end{proof}

Before giving the proof of Theorem \ref{Satthm}, we explain the mod $p$ Satake isomorphism. Explicitly, the convolution product on $\mathcal{H}_G$ is defined by the formula
$$f_1*f_2(g) = \sum_{h \in G(F)/G(\mathcal{O})} f_1(gh) f_2(h^{-1}).$$ Fix a maximal torus and a Borel subgroup $T \subset B \subset G$, and let $U$ be the unipotent radical of $B$. Following Herzig \cite{modpsatake}, we define the map 
$$\mathcal{S} \colon \mathcal{H}_G \rightarrow \mathcal{H}_T$$
by
$$\mathcal{S}(f)(g) = \sum_{u \in U(F)/U(\mathcal{O})} f(gu).$$ As $T$ is commutative, there is a natural isomorphism $\mathbb{F}_p[X_*(T)] \cong \mathcal{H}_T$, where the element $\mu \in X_*(T)$ corresponds to the function which is $1$ on the coset $\mu(t)T(\mathcal{O})$ and $0$ everywhere else. In the case where $F$ is a $p$-adic field and the coefficients are taken in an algebraically closed field of characteristic $p$ (instead of $\mathbb{F}_p$), Herzig \cite{modpsatake} showed that $\mathcal{S}$ is injective with image spanned by the anti-dominant cocharacters. Henniart and Vign\'{e}ras \cite{modpsatake2} extended this result to an arbitrary non-Archimedean local field $F$ and coefficients taken in any field of characteristic $p$. The resulting isomorphism $$\mathcal{S} \colon \mathcal{H}_G \xrightarrow{\sim} \mathbb{F}_p[X_*(T)_-]$$ is what we will call the mod $p$ Satake isomorphism. In fact, there is a similar mod $p$ Satake isomorphism for any weight (a representation of the finite group $G(\mathbb{F}_q)$), but here we are only considering the case where the weight is trivial.

\begin{myproof2}[Proof of Theorem \ref{Satthm}]
We first define a map $\mathbb{F}_p[X_*(T)^+] \rightarrow K_0(P_{L^+G}(\Gr, \mathbb{F}_p)) \otimes \mathbb{F}_p$ as in the statement of Theorem \ref{Satthm} by sending $\mu \in X_*(T)^+$ to the element $[\IC_\mu] \otimes 1$. This is an $\mathbb{F}_p$-algebra homomorphism by Theorem \ref{convsimp}. It is an isomorphism because the perverse sheaves $\IC_\mu$ represent the distinct isomorphism classes of simple objects in $P_{L^+G}(\Gr, \mathbb{F}_p)$.

By the Cartan decomposition $\mathcal{H}_G$ has a basis $\{\tau_\mu\}_{\mu \in X_*(T)^+}$, where
$\tau_{\mu}$ is $1$ on the double coset $G(\mathcal{O}) \mu(t) G(\mathcal{O})$ and $0$ everywhere else. In the case where the derived group of $G$ is simply connected, Herzig \cite[5.1]{modpclass} has computed the inverse of $\mathcal{S}$. When the weight is trivial this assumption on $G$ is unnecessary. Herzig's computation works in our setup exactly as written, and the result is the formula
$$\mathcal{S}^{-1}(\mu) = \sum_{\lambda \geq \mu} \tau_\lambda.$$ Here the sum is over anti-dominant cocharacters $\lambda \geq \mu$. Herzig derives this formula as a consequence of the Lusztig–Kato formula over the complex numbers. The key observation is that the Kazhdan–Lusztig polynomials all have constant term $1$. 

Because the $\IC$ sheaves in $P_{L^+G}(\Gr, \mathbb{F}_p)$ are constant, it is immediate that $$\mathcal{T}([\IC_\mu] \otimes 1) = \sum_{\lambda \leq \mu} \tau_\lambda.$$ Here we are now using dominant cocharacters. Note that since $w_0 \in N_G(T)/T$ has a representative in $G(\mathbb{F}_q) \subset G(\mathcal{O})$, we have $$\tau_\mu = \tau_{w_0(\mu)}$$ for all $\mu \in X_*(T)$. Now by comparing the previous two formulas for $\mathcal{T}$ and $\mathcal{S}^{-1}$ it follows that the composition of maps defined in Theorem \ref{Satthm} is $\mathcal{S}^{-1}$. 
\end{myproof2}

\end{document}